\documentclass[11pt]{report}
\usepackage{amsmath}
\usepackage{amssymb}
\usepackage{latexsym}
\usepackage{graphics}
\usepackage[all]{xy}
\usepackage{url} 
\usepackage{graphicx}
\usepackage{float}
\usepackage{amsthm}
\newtheorem{df}{Definition}[section]
\newtheorem{thm}[df]{Theorem}
\newtheorem{prop}[df]{Proposition}
\newtheorem{cor}[df]{Corollary}
\newtheorem{ex}[df]{Example}

\newtheorem{lem}[df]{Lemma}
\newtheorem{rmk}[df]{Remark}

\newcommand{\pfend}%
{\hspace*{\fill}\lower3pt\hbox{$\Box$}\medskip}

\newcommand{\N}{\mathbb{N}}
\newcommand{\R}{\mathbb{R}}
\newcommand{\Z}{\mathbb{Z}}
\newcommand{\C}{\mathbb{C}}
\newcommand{\Comp}{\mathbb{C}}
\newcommand{\qt}{\mathbb{H}}
\newcommand{\RP}{\mathbb{RP}}
\newcommand{\CP}{\mathbb{CP}}
\newcommand{\HP}{\mathbb{HP}}
\begin{document}
\begin{titlepage}
\vspace*{\stretch{1}}
\begin{center}\bf
{\LARGE \textbf{The Gromov-Lawson-Rosenberg conjecture for some finite groups}}\\
\vspace{1.2cm}
{\Large Arjun Malhotra}\\
\vspace{1cm}
{\large A thesis submitted for the degree of \\

Doctor of Philosophy}\\
\bigskip
{\large Department of Pure Mathematics}\\
\bigskip
{\large February 2011}\\
\bigskip\bigskip\bigskip
{\Large Supervisor: Prof. John Patrick Campbell Greenlees}
\end{center}
\vspace*{\stretch{2.5}}

\begin{center}
\textbf{\large The University of Sheffield}
\end{center}
\end{titlepage}
\newpage
\begin{abstract}
{ The Gromov-Lawson-Rosenberg conjecture for a group $G$ states that a compact spin manifold with fundamental group $G$ admits a metric of positive scalar curvature if and only if a certain topological obstruction vanishes. It is known to be true for $G=1$ , if $G$ has periodic cohomology, and if $G$ is a free group, free abelian group, or the fundamental group of an orientable surface. It is also known to be false for a large class of infinite groups. However, there are no known counterexamples for finite groups.\\
In this dissertation we will give a general outline of the positive scalar curvature problem, and sketch proofs of some of the known positive and negative results.\\
We will then focus on finite groups, and proceed to prove the conjecture for the Klein $4$-group, all dihedral groups, the semi-dihedral group of order $16$, and the rank three group $(\Z_2)^3$. Throughout the thesis, $ko$ will represent the connective real K-theory spectrum, and $KO$ the periodic real K-theory, and $\Z_2$ the field with two elements. It turns out that that the topological obstruction in question lies in the connective real homology $ko_*(BG)$ of the classifying space of $G$.\\
Indeed, our method of proof is to first sketch calculations of $ko_*(BG)$, using the techniques and calculations of Bruner and Greenlees. We then give explicit geometric constructions to produce sufficiently many manifolds of positive scalar curvature.
}
\end{abstract}

\newpage
\vspace*{30 mm}
\begin{center}
\textbf{Acknowledgements}
\end{center}

First and foremost I would like to thank my supervisor John Greenlees. This thesis would not have been possible without the time and effort he put in to me and my project, and his guidance, enthusiasm, advice and friendship were a constant source of both comfort and inspiration.\\

I would also like to give special thanks to Michael Joachim. His advice and collaboration have been essential to this thesis. My visits to M\"{u}nster were hugely enjoyable both mathematically and socially, and they may not have been possible without the financial and personal support he so willingly gave me.\\

Kijti Rodtes was a constant inspiration to me throughout my time in Sheffield. I admire and envy his tenacity and dedication, and am extremely grateful to him for our collaboration, particularly for his thorough writing, his patient explanations of his calculations, his persistence in asking questions about my work, and most of all his friendship.\\

I'm grateful to Paul Mitchener for several stimulating conversations, his enthusiasm in organising seminars that were interesting and highly relevant to my research, along with several hugely enjoyable social outings. Thanks to both Paul and Thomas Schick for their thorough reading of this thesis, and for providing several helpful comments and suggestions.\\

I am indebted to the Overseas Research Studentship (ORS) council and the University of Sheffield for the financial support they have provided over the last three years.\\

The Pure Mathematics department in Sheffield provides, I believe, a unique combination of intellectual stimulation and relaxed social enjoyment.  A PhD is, I believe, quite difficult to adjust to personally, and the latter is particularly important in this regard. I'd like to thank all the members of the department for the conversations, advice and fun over the years, especially my fellow postgraduates. Special thanks in particular go to my housemates Leigh Shepperson, Roald Koudenburg, Harry Ullman and Tim Eardley for their friendship, as well as all the fun over the years.\\

My mathematical education began with four hugely enjoyable years at the University of Bath, and I'd like to thank everyone who taught and supported me there, in particular Gregory Sankaran, whose guidance and support were a constant source of encouragement not only throughout my time there, but subsequently as well.\\

On a personal note, I would like to thank Zofia Jones. Her love and support have been essential over the years, particularly during the unavoidable low points that a PhD consists of. I would also like to thank Colin Hargreaves and Pranoy Banerjee for their seemingly evergreen friendship.\\

My parents, Ashok and Neelam, have been a constant source of love and encouragement, and have always supported me and let me make my own decisions, and I am eternally grateful. Finally I would like to thank the rest of my family, as well as all my friends, who have always been supportive and available, particularly those from my school days in Calcutta.\\
\newpage
\tableofcontents
\chapter{Introduction}

Consider a Riemannian manifold $M^m$ with metric $g$. Then the scalar curvature of $(M^m,g)$ is a smooth function $s:M \rightarrow \R$. The value of $s$ at a point $x \in M$ can be described as follows:\\
1)As the trace of the Ricci tensor (which in turn is the trace of the Riemannian curvature tensor) evaluated at $x$;\\
2)As twice the sum of the sectional curvatures over all two-planes $e_i \wedge e_j, i < j$, where $e_1, \cdots , e_n$ is an orthonormal basis for the tangent space at $x$;\\
3)Up to a positive constant depending only on $m$, as the leading term in the expansion
$$\frac{Vol_r(M,x)}{Vol_r(\R^m,0)}=1-\frac{s(x)}{6(m+2)}r^2 + \cdots$$
that describes how the volume of a small ball of radius $r$ on $M$ differs from the corresponding volume around a point in $\R^m$ with its usual metric. Thus we can say that $(M^m,g)$ has positive scalar curvature if the volumes of small balls around any point $x\in M$ grow more slowly than balls in ordinary Euclidean space of the same radius, and this view is perhaps intuitively the simplest. Easy examples of manifolds admitting positive scalar curvature are the standard sphere $S^n$ with it's usual metric, and quotients of spheres such as projective and lens spaces, with the induced metric.\\
It turns out that any manifold of dimension at least $3$ can be given a metric of negative scalar curvature, see \cite{kw1},\cite{kw2} and \cite{kw3}. However, there are topological obstructions to a manifold admitting a positive scalar curvature metric, and when a manifold admits such a metric is the basis of what we study here.\\
First though, we outline some basic constructions that produce positive scalar curvature manifolds, and the following result, due to Gromov and Lawson \cite{6}, and independently, Schoen and Yau \cite{ry}, is fundamental:
\begin{thm}
Suppose that $M$ is a positive scalar curvature manifold, and $N$ is obtained from $M$ by performing surgeries in codimension at least $3$. Then $N$ also admits a positive scalar curvature metric.
\end{thm}
It is known that one can obtain a manifold $N^n$ from another manifold $M^n$ by surgeries if and only if $M$ and $N$ are bordant, which means there is a manifold $W^{n+1}$ with boundary the disjoint union $M \amalg N$. Thus this result suggests that we can think of the positive scalar curvature problem as a question in bordism, and indeed the known topological obstructions depend on certain bordism classes of a given manifold.\\
We also have the following basic result, which tells us that certain geometric constructions we use later automatically give positive scalar curvature manifolds:
\begin{lem}
If $(M,g),(N,h)$ are two manifolds, and $M$ admits positive scalar curvature, then so does $M \times N$.\\
More generally, if we have a fibre bundle $F \rightarrow E \rightarrow B$, with compact positive scalar curvature fibre $F$, and structure group $G$ acting by isometries on $F$, then $E$ admits positive scalar curvature also.
\end{lem}
The first part can easily be proved by taking the metric $tg \oplus h$ for small $t$ (so just shrinking the product metric $g \oplus h$ in the $M$ factor, and noting that $s(tg \oplus h)=s(g)/t+s(h)>0$ for $t$ sufficiently small. The second part can be proved in the same way, by shrinking the metric in the fibre uniformly (as G acts by isometries).

\section{The conjecture}
In this section we outline the topological obstructions to positive scalar curvature that gave rise to the Gromov-Lawson-Rosenberg conjecture, and state the conjecture explicitly.\\
It is a well known result of Lichnerowicz \cite{12} that there are spin manifolds which do not admit positive scalar curvature metrics. Indeed, by the Lichnerowicz formula, the existence of such a metric implies that the index of the Dirac operator vanishes. This, combined with the Atiyah-Singer index theorem implies that a topological invariant known as the $\hat{A}$ genus, which is a linear combination of the Pontrjagin classes of the manifold, vanishes.\\
The $\hat{A}$ obstruction was generalized by Hitchin \cite{7} to an obstruction $\alpha(M^n) \in KO_{n}$, where $\alpha$ denotes the Atiyah-Bott-Shapiro homomorphism \cite{abs}. This agrees with $\hat{A}$ in dimensions $0 \mod 4$, but is in fact a strict generalization, since the real periodic $K-$theory spectrum $KO$ has $\pi_n(KO)= \Z_2$ for $n \equiv 1,2 \mod 8$, and indeed Hitchin constructed exotic spheres admitting no metric of positive scalar curvature in these dimensions!\\
Letting $\pi$ denote any fundamental group, the homomorphism $\alpha$ gives rise to a transformation of cohomology theories
$$\alpha:\Omega ^{spin}_{n}(B\pi) \to KO_{n}(B\pi)$$ and Gromov and Lawson conjectured \cite{6} that $\alpha(M^n)=0$ was also a sufficient condition for $M$ to admit a metric of positive scalar curvature.\\
Rosenberg \cite{ro3} later generalized this further, showing that if a spin manifold $M$ with fundamental group $\pi$ admitted a metric of positive scalar curvature, then $ind([M,u])=0$, where $u$ is the classifying map of the universal cover of $M$, and $ind$ maps to the real $K-$theory of the reduced $C_*$ algebra of $\pi$:
$$ind=A \circ \alpha:\Omega ^{spin}_{*}(B\pi) \to KO_{*}(C^{*}_{red}\pi)$$
This can be thought of as an equivariant generalized index, and the map A is the assembly map of Baum-Connes \cite{bch}.\\
Modifying the Gromov-Lawson conjecture, Rosenberg conjectured that the converse was true also; namely that a compact spin manifold $M^{n}$ with $\pi_{1}(M)=\pi$ and $n \geq 5$ admits a positive scalar curvature metric if and only if $ind[u:M\to B\pi]=0 \in KO_{n}(C^{*}_{red}\pi)$.\\
The conjecture has been proven in the simply connected case \cite{17}, if $\pi$ has periodic cohomology \cite{2}, and if $\pi$ is a free group, free abelian group, or the fundamental group of an orientable surface \cite{15}. It is also known to be false in general, for example if $\pi=\Z^{4} \times \Z_{3}$, and for a large class of torsion free groups; see \cite{4}, \cite{8} and \cite{16}. Further, by a result of Kwasik and Schultz \cite{10}, the conjecture is true for a finite group $\pi$ if and only if it is true for all Sylow subgroups of $\pi$.\\
We focus on $2-$groups in this thesis, and prove the following main results, starting with elementary abelian groups $V(n)=(\Z_2)^n$, where $\Z_2$ is the field with two elements
\begin{thm}
The Gromov-Lawson-Rosenberg conjecture is true for $\pi=V(2)$.
\end{thm}
\begin{thm}
The Gromov-Lawson-Rosenberg conjecture is true for $\pi=V(3)$.
\end{thm}
\begin{thm}
The Gromov-Lawson-Rosenberg conjecture is true for all dihedral groups.
\end{thm}
Note that it suffices to prove this result for dihedral 2-groups, since by \cite{10}, the conjecture is true for a finite group G if and only if it is true for all Sylow subgroups of G.
\begin{thm}
The Gromov-Lawson-Rosenberg conjecture is true for the semi-dihedral group $\pi=SD_{16}$.
\end{thm}

\section{The structure and content of the thesis}

The thesis is structured as follows. The next section of this chapter will give an expository account of the known counterexamples to the conjecture, see \cite{4}, \cite{8} and \cite{16}, and the methods used therein, together with a few remarks about exotic spheres which do not admit positive scalar curvature metrics. The following section will give a sketch of the proof of the conjecture in the simply connected case \cite{17}.\\

The remaining sections describe the methods we use. The essential result is that the obstruction group is in fact the real connective K-theory $ko_*(B\pi)$ of the classifying space of the group $\pi$, \cite{18}, and our method is to first calculate $ko_*(B\pi)$ and then construct sufficiently many spin manifolds of positive scalar curvature. We give an outline of local cohomology and the local cohomology spectral sequence, and sketch the calculation of $ko_*(BV(2))$ using it, where $V(2)$ is the Klein 4-group. Details are in \cite{3}.\\

The next chapter will describe the eta invariant method to outline the proof for groups with periodic cohomology in \cite{2}, dealing with the cases of the cyclic group $C_4$ of order $4$ and the quaternion group $Q_8$ of order $8$ in more detail. We give explicit geometric generators for $ko_{4k \pm 1}(BC_4)$ and $ko_{4k-1}(BQ_8)$, and show how explicit calculations using the eta invariant can be used to resolve the extensions for these groups. This turns out to be useful in subsequent chapters, where for a group $G$ we start by considering the image of maps $ko_*(BH) \rightarrow ko_*(BG)$ induced by inclusions $H \hookrightarrow G$, where $H$ is a periodic subgroup of $G$.\\

Chapter three deals with elementary abelian groups, and builds upon the calculations in \cite{3} of $ko_*(BV(r))$. The argument for $V(2)$ will appear in \cite{mjam}, while we will here also deal with $V(3)$ as well as the image in periodic K-theory $ko_*(BV(r))[\beta^{-1}]$ for all $r$.\\

In chapter four we prove the conjecture for all dihedral groups. This will appear in \cite{mjam}, where we use the Adams spectral sequence both for calculations and geometric insight. The proof presented here will only use the local cohomology calculations in \cite{3}.\\

Finally, the last chapter uses the calculations in \cite{kijti} of $ko_*(BSD_{16})$. We prove the conjecture by constructing the necessary manifolds explicitly, and this is the main focus both of the chapter and of \cite{amkr}.

\section{Counterexamples, other obstructions and exotic spheres}
The Gromov-Lawson-Rosenberg conjecture is known to be false for certain infinite groups, for example $\pi=\Z^4 \times \Z/3$, see \cite{16}. The method used in this and other cases is known as the minimal surface obstruction. These definitions and results may be found in \cite{16}, and follows from work in \cite{ry}.\\
To state the obstruction it is useful to introduce the following notation, which defines the positive scalar curvature part of the homology of a space.
\begin{df}
Let $X$ be any space and $E$ any homology theory. The subgroup $E_*^+(X)$ of $E_*(X)$ is defined as the set $\{f_*([M])|f:M \rightarrow X,\textrm{ M is any positive scalar curvature manifold} \}$.
\end{df}
\begin{thm}
Let $X$ be any space and $\alpha \in H^1(X;\Z)$. Then, for $k \geq 3$, cap product with $\alpha$ sends $H_k^{+}(X;\Z)$ to $H_{k-1}^{+}(X;\Z)$.\\
\end{thm}

So one method of finding counterexamples is to pick a group $G$ with non-trivial first integral cohomology (this however rules out all finite groups immediately), and then pick a suitable class $[N^n,f]$ in the spin bordism of $G$ which has vanishing index and represents a class $0 \neq c \in H_n(BG;\Z)$. If the conjecture holds, this class must also be realized by some class $[M^n,u]$, where $M$ admits a positive scalar curvature metric. However, we can now apply the above theorem by capping $c$ with some $x_1 \in H^1(BG;\Z)$ to get a positive scalar curvature submanifold $L^{n-1}$ of $M$ of  codimension one whose image in homology under $u_*$ is $x \cap c$. Now we can try and repeat the procedure on $L$, and hope to reach sufficiently low dimensions to obtain a contradiction. So for example, if for suitable $x_i \in H^1(BG;\Z)$ we have $x_n \cap \cdots \cap x_1 \cap c=c' \neq 0 \in H_{2}(BG;\Z)$, then we will have a contradiction, because for a discrete group $\pi$ we must have $H_2^{+}(B\pi;\Z)=0$ since the only orientable surface of positive scalar curvature is $S^2$, and $\pi_2(B\pi)=0$. This is the method used for $\pi=\Z^4 \times \Z/3$, \cite{16}:\\
In some more detail, give the circle $S^1$ the spin structure induced from the disc, and choose a map $p:S^1 \rightarrow B\Z/3$ surjective on $\pi_1$. By the Atiyah-Hirzebruch spectral sequence, we have $\widetilde{\Omega}_{1}^{spin}(B\Z/3)=\Z/3$, so that $[S^1,p]$ is $3$-torsion. We can then consider the singular manifold given by
$$ f=id \times p:(S^1)^4 \times S^1 \rightarrow B\pi$$
which must also be $3-$torsion, and doing surgery we can construct a bordism in $\Omega_5^{spin}(B\pi)$ to some $u:M \rightarrow B\pi$ which is an isomorphism on $\pi_1$. It is a standard fact (see \cite{bch},\cite{6}) that $KO_{*}(C^{*}_{red}\pi)$ contains only $2-$torsion, so this manifold must have trivial index. Now note that
$$H_1(B\pi;\Z)=x_1\Z \oplus x_2\Z \oplus x_3\Z \oplus x_4\Z \otimes y\Z/3$$
$$H^1(B\pi;\Z)=a_1\Z \oplus a_2\Z \oplus a_3\Z \oplus a_4\Z $$
We then simply observe
$$0 \neq w=x_1 \times x_2 \times x_3 \times x_4 \times y \in H_5(B\pi;\Z)$$
$$0 \neq z=a_1 \cap (a_2 \cap(a_3 \cap w)) \in H_2(B\pi;\Z)$$
So that if $M$ admitted positive scalar curvature, we would have $u_*(M)=w \in H_5^{+}(B\pi;\Z)$, and so Theorem $1.3.2$ would imply $ 0 \neq z \in H_2^{+}(B\pi;\Z)$, a contradiction.\\
There are much more general examples, including torsion-free groups, where the conjecture can be shown to be false, using the Atiyah-Hirzebruch spectral sequence and a generalized version of the above argument. Details may be found in \cite{4}. Also, this counterexample is one instance of a toral class, namely a class in $H_n(B\pi;\Z)$ which is the image of the fundamental class of some torus $T^n$. Tori do not admit positive scalar curvature \cite{enl}, and toral classes are considered an essential test case for finite groups, and there are partial results for some groups that essentially say the conjecture is true atorally \cite{boro}, basically meaning that any class which is not toral can be realized by a manifold of positive scalar curvature. The toral case in general has not been dealt with, but there are some partial results, for example it is known that for elementary abelian $2-$groups, toral classes can be realized by positive scalar curvature manifolds, see \cite{mj}.\\
It is worth remarking briefly that in dimension $4$, there are additional Seiberg-Witten obstructions which mean for example that even in the simply connected case, the conjecture fails to be true, see \cite{tau}, \cite{sw}.\\
Further, while we think of scalar curvature as a very weak geometric invariant, it is surprising just what kind of manifolds do not admit positive scalar curvature metrics. For example, in dimensions $1,2 \mod 8$, there are exotic spheres which admit no such metric! These exotic spheres do not bound parallelizable manifolds, and in fact are not even spin boundaries. One construction of such exotic spheres is to take the circle $[S^1]$ with the non-trivial spin structure, so that $\alpha([S^1]) \neq 0 \in KO_n(pt)$, and then consider the product with the Bott manifold $B^8$. Then we have $\alpha([S^1 \times B^8]) \neq 0$ since $\hat{A}(B^8)=1$, and it turns out \cite{misp} that we can now perform surgeries to get a homotopy sphere $\Sigma$ with $\alpha(\Sigma) \neq 0$. More details on exotic spheres and curvature can be found in the article of Joachim and Wraith \cite{8b}.

\section{The simply connected case}
In this section we give a brief outline of Stolz's proof \cite{17} of the conjecture in the simply connected case. We have the following main result :
\begin{thm}
The kernel of $\alpha:\Omega_*^{Spin}(pt) \rightarrow KO_8(pt)$ is equal to the subgroup $T_n$ consisting of bordism classes represented by total spaces of $\HP^2$ bundles with structure group $PSp(3)$, by which me mean fibre bundles with fibre the quaternion projective space $\HP^2$ and structure group the projective symplectic group $PSp(3)$
\end{thm}
Quaternion projective spaces $\HP^n=S^{4n+3}/S^3$ have natural positive scalar curvature metrics, and since $G=PSp(3)$ acts on the fibres $\HP^2$ isometrically, Lemma $1.0.2$ tells us these bordism classes are represented by positive scalar curvature manifolds, thus implying the conjecture in the simply connected case. This result is proved localised at odd primes in \cite{9} by using explicit geometric constructions of $\HP^2$ bundles over products of quaternion projective spaces. By contrast, the proof localised at 2 in \cite{17} translates the statement into stable homotopy theory using the Pontrjagin-Thom construction, and then notes that a certain Adams spectral sequence collapses. Here is a brief sketch:\\
First note that the group $T_n$ can be identified with the image of the homomorphism
$$\psi:\Omega_{n-8}^{Spin}(BG) \rightarrow \Omega_n^{Spin}(pt)$$
which sends a bordism class $[N^{n-8},f]$ to $\hat{N}$, which is the pullback via $f$ in the diagram below of the universal $\HP^2$ bundle $EG \times_G \HP^2$:
$$
\xymatrix{
\HP^2 \ar[r] \ar [d] & \HP^2 \ar [d] \\
\hat{N} \ar[r] \ar [d] & EG \times_G \HP^2 \ar [d]\\
N \ar[r]^{f} & BG }
$$
We then have the following commutative diagram:
$$
\xymatrix{
\Omega_{n-8}^{Spin}(BG) \ar[r]^{\psi} \ar [d]^{\cong} & \Omega_n^{spin}(pt) \ar[r]^{\alpha} \ar [d]^{\cong} & KO_n(pt) \ar[d]^{\cong} \\
\pi_n(MSpin \wedge \Sigma^8 BG_+)  \ar[r]^{T_*} \ar [d]^{\hat{T}_*} & \pi_n(MSpin) \ar[r]^{D_*} \ar [d]^{id} & \pi_n(ko) \ar[d]^{id} \\
\pi_n(\hat{MSpin}) \ar[r] & \pi_n(MSpin) \ar[r] & \pi_n(ko) }
$$
The Pontrjagin-Thom construction tells us that for a space $X$ we have $\Omega_n^{spin}(X) \cong \pi_n(Mspin \wedge X_+)$, where $MSpin$ is the Thom spectrum over $Bspin$ and $+$ denotes union with a disjoint basepoint. Thus the left and vertical arrows of the above diagram are immediate via this construction.\\
It can be shown using the families index theorem that the composite $D \circ T$ is null-homotopic, which implies that $T$ factors through a map $\hat{T}$ into the homotopy fibre $\hat{Mspin}$, and the bottom row is part of the long exact sequence in homotopy of this fibration. However, the theorem is equivalent to saying that the top row of the diagram is exact, and so it would suffice to show that $\hat{T}_*$ is surjective at $2$. This follows from the following facts, proved in \cite{17}:\\
1) The homomorphism induced by $\hat{T}_*$ on $\Z_2$ cohomology is a split injection of modules over the Steenrod Algebra;\\
2) The Adams spectral sequence converging to the 2-local homotopy groups of $Mspin \wedge BG_+$ collapses at the $E_2$ page.\\
This implies the induced map of spectral sequences is a surjection on the 2-local homotopy groups at $E_{\infty}$, since 1) tells us this at $E_2$, and 2) says there are no more differentials.

\section{Our method}
We note that the index map factors through connective and periodic real K-theory as follows:
$$ind=A \circ p \circ D: \Omega_{n}^{spin}(B \pi) \rightarrow ko_{n}(B \pi)\rightarrow KO_{n}(B \pi) \rightarrow KO_{*}(C^{*}_{red}\pi) $$
Here $D$ is a generalized Dirac operator induced by the Atiyah-Bott-Shapiro homomorphism $D:\Omega_{*}^{spin} \rightarrow ko_{*}$ which sends a spin manifold to its $ko$ fundamental class, $p$ is the periodicity map inverting the Bott element, and $A$ is the assembly map.\\
We have that
$$ko_* \cong \Z[\alpha, \beta, \eta]/(\eta^3, 2\eta, \alpha^2-4\beta)$$
Where $\alpha$ is in degree $4$, $\beta$ in degree $8$ and $\eta$ in degree $1$. Further, $KO_*=ko_*[\beta^{-1}]$.\\
We again use the notation:
$$ko_{n}^{+}(X)=\{D[N,f]; [N,f] \in \Omega_n^{spin}(X),\textrm{ N is any positive scalar curvature spin manifold} \}$$
Thus $ko_{n}^{+}(B \pi)$ is the set $D(\Omega_{n}^{spin,+}(B \pi))$, where $ \Omega_{n}^{spin,+}(B \pi)$ is the subgroup with elements given by pairs $[N,f]$ for which $N$ admits a positive scalar curvature metric. Note that there is no restriction on the fundamental group of $N$ here. The following result, due to Jung and Stolz \cite{18}, is the basis of all our proofs.
\begin{thm}
A compact spin manifold $M^{n}$ with $\pi_{1}(M)=\pi$ and $n \geq 5$ admits a positive scalar curvature metric if and only if $D(M,u) \in ko_{n}^{+}(B \pi)$, where $u$ is the classifying map for the universal cover of $M$.
\end{thm}
The groups $ko_*(X)$ of a space $X$ are much smaller and easier to calculate than the spin borsism groups $\Omega_*^{spin}(X)$, so this result is a considerable simplification. Further, it means that one way of proving the conjecture is to first calculate $ko_{n}(B \pi)$, and then to identify the kernel of $A \circ p$, which we denote by $Ker(Ap)$, and realize all of it by positive scalar curvature manifolds, and this is what we shall do.\\
Calculations of $ko_*(B\pi)$ may be carried out using a combination of spectral sequences. For instance, the Atiyah-Hirzebruch spectral sequence is used in \cite{2}. Calculations for the cases we will focus on use a combination of the Adams, Bockstein and local cohomology spectral sequences, and may be found in \cite{3} and \cite{kijti}.\\
Note that the periodic real $K-$theory  of finite groups is known by the following result of Rosenberg and Stolz \cite{15}
\begin{thm}
For a finite p-group $G$, the periodic real $K-$ theory $KO_*(BG)$ splits as a $KO_*$ module with period $8$ into a direct sum:
\begin{center}
\begin{tabular}{|c|c|c|c|c|c|c|c|c|}
\hline
$ n \mod 8$ & $0$ & $1$ & $2$ & $3$ & $4$& $5$ & $6$ & $7$\\
$\R$& $0$ & $(\Z_2)$ & $(\Z_2)$ & $\Z_{p^{\infty}}$ & $0$ & $0$ & $0$ & $\Z_{p^{\infty}}$\\
$\Comp$& $0$ & $\Z_{p^{\infty}}$ & $0$ & $\Z_{p^{\infty}}$ & $0$ & $\Z_{p^{\infty}}$ & $0$ & $\Z_{p^{\infty}}$\\
$\qt$& $0$ & $0$ & $0$ & $\Z_{p^{\infty}}$ & $0$ & $(\Z_2)$ & $(\Z_2)$ & $\Z_{p^{\infty}}$\\
\hline
\end{tabular}
\end{center}

one for each irreducible real ($\R$) or quaternion ($\qt$) representation, and one for each pair of non-isomorphic complex conjugate representations ($\Comp$). Here $(\Z_2)$ is $\Z_2$ if $p=2$ and $0$ otherwise.
\end{thm}
Extension problems for $ko_*(BG)$ as well as differentials in the spectral sequences can often be determined by comparing with the periodic case. Note also that the map $p:ko_*(BG) \rightarrow KO_*(BG)$ is given by just inverting the Bott element $\beta$. Further, for finite $2-$groups, the assembly map $A$ is non-trivial exactly on the $(\Z_2)$ summands \cite{15}. This is shown up in calculations using the local cohomology spectral sequence, as the next two sections show.

\section{Local cohomology}

We  have seen that in order to prove the Gromov-Lawson-Rosenberg conjecture, we need to calculate $ko_*(BG)$ for a finite group $G$. This thesis will be based upon the calculations in \cite{3} and \cite{kijti}, and to this end, we will give a very brief idea of the theory behind the local cohomology spectral sequence, and say why it is useful in attacking the Gromov-Lawson-Rosenberg conjecture. We start with some preliminaries about the local cohomology functor, due to Groethendieck \cite{gr}.\\

The definition of local cohomology which is suitable for our calculations is defined via the stable Koszul complex.

\begin{df}
For a commutative ring (with unity) $R$ and ideal $I=(x_{1},x_{2},...,x_{n})$, the
\emph{stable Koszul complex} of $R$ at $I$ is
\begin{equation*}
    K^{\infty}(x_{1},x_{2},...,x_{n};R) = K^{\infty}(x_{1};R)\otimes_{R}
    K^{\infty}(x_{2};R) \otimes_{R} ...\otimes_{R}K^{\infty}(x_{n};R)
\end{equation*}
the tensor product of the cochain complexes $K^{\infty}(x_{i};R)$, where $K^{\infty}(x_{i};R)$ is the cochain complex
$(R\longrightarrow R[\frac{1}{x_{i}}])$, ($r\longmapsto \frac{r}{1}$), for each $i\in
\{1,2,...,n\}$.  For a module $M$ over the ring $R$, the \emph{local cohomology of $M$ at
$I$} is
\begin{equation*}
    H_{I}^{*}(R;M):=H^{*}(K^{\infty}(x_{1},x_{2},...,x_{n};R)\otimes_{R}M)
\end{equation*}
where $H^{*}(C)$ is the homology of a chain complex $C$. In particular, we define
\begin{center}
$H_{I}^{*}(R):=H_{I}^{*}(R;R)$.
\end{center}
It is clear from the definition that $H_{I}^{i}(R;M)=0$ for $i>n$.

\end{df}

\begin{rmk}\label{remark for natural maps of local chain complexes}
Let $R$ be a ring and $(x)$ be an ideal of $R$.  The chain complex
$K^{\infty}(x)=(R\longrightarrow R[\frac{1}{x}])$ gives a natural map
$\varepsilon:K^{\infty}(x)\longrightarrow R$. More precisely, there is a commutative
diagram:
$$ \xymatrix{ K^{\infty}(x) \ar[d]^{\varepsilon}  \\
                                                   R     } \xymatrix {= \\
                                                                         }\xymatrix{
                                                                       =(R\ar[r]\ar[d]&
                                                                       R[\frac{1}{x}]
                                                                       )\ar[d] \\
               (R\ar[r]   & 0 ) .  } $$
Hence, for any ideal $I=(x_{1},x_{2},...,x_{m})$ and $J=(y_{1},y_{2},...,y_{n})$ of $R$,
there exists a map of chain complexes
$$1\otimes\varepsilon^{n}:K^{\infty}(I+J)=K^{\infty}(I)\otimes_{R}K^{\infty}(J)\longrightarrow
K^{\infty}(I)=K^{\infty}(I)\otimes_{R}R.$$ After applying $\otimes_{R}M$, where $M$ is a
module over $R$, and taking homology, we obtain the map
$$\eta:H^{s}_{I+J}(R;M)\longrightarrow H^{s}_{I}(R;M).$$
\end{rmk}

\begin{ex}
 For $R=\mathbb{Z}$ and $I=(2)$, we have $ K^{\infty}(2;\mathbb{Z})=(\mathbb{Z}
 \longrightarrow \mathbb{Z}[\frac{1}{2}])$.  The map in this cochain complex is clearly a
 monomorphism and the cokernel is also easy to calculate.  That is
\begin{center}
 $H_{(2)}^{i}(\mathbb{Z})$=$\left\{
                      \begin{array}{ll}
                        \mathbb{Z}/2^{\infty}, & \hbox{if i=1 ;} \\
                        0, & \hbox{otherwise,}
                      \end{array}
                    \right.$
\end{center}

\end{ex}

\begin{ex}
For $R=k[x]$, a polynomial ring over a field $k$ with indeterminate $x$ of degree $r$ and
$I=(x)$, we have $ K^{\infty}(x;k[x])=(k[x] \longrightarrow k[x][\frac{1}{x}])$.  The
calculation is easier if we look at the picture below.\\

\setlength{\unitlength}{0.8cm}
\begin{picture}(10,6)
\put(3.5,-0.3){\textbf{Figure 3.1}: Koszul complex of $k[x]$ at $(x)$.} \put(4,6){$R$}
\put(6,6.3){$i$}
 \put(8,6){$R[\frac{1}{x}]$}
 \put(4.4,6.15){\vector(1,0){3.5}}
\put(4.1,3){\vector(0,1){2.5}} \multiput(4.11,3)(0,0.7){4}%
{\circle*{0.07}}
\put(3.65,2.95){\tiny $1$} \put(3.65,3.65){\tiny$x$} \put(3.65,4.35){\tiny$x^{2}$}
\put(3.65,5.05){\tiny$x^{3}$}

\put(8.1,3){\vector(0,1){2.5}} \multiput(8.11,3)(0,0.7){4}%
{\circle*{0.07}}
\put(8.2,2.95){\tiny$1$} \put(8.2,3.65){\tiny$x$} \put(8.2,4.5){\tiny$x^{2}$}
\put(8.2,5.05){\tiny$x^{3}$}

\put(8.1,3){\vector(0,-1){2.5}} \multiput(8.11,0.9)(0,0.7){3}%
{\circle{0.07}}
\put(8.2,2.25){\tiny$x^{-1}$} \put(8.2,1.55){\tiny$x^{-2}$}
\put(8.2,0.85){\tiny$x^{-3}$}

\multiput(4.15,5.1)(0.2,0){20}%
{\line(1,0){0.1}}
\multiput(4.15,4.4)(0.2,0){20}%
{\line(1,0){0.1}}
\multiput(4.15,3.7)(0.2,0){20}%
{\line(1,0){0.1}}
\multiput(4.15,3)(0.2,0){20}%
{\line(1,0){0.1}}

\put(7.9,3){\vector(1,0){0.1}} \put(7.9,3.7){\vector(1,0){0.1}}
\put(7.9,4.4){\vector(1,0){0.1}} \put(7.9,5.1){\vector(1,0){0.1}}

\end{picture}\\

This means the kernel of $i$ is zero and the cokernel of $i$ is $k[x,x^{-1}]/k[x]$ which
is $\Sigma_{-r}(k[x]^{\vee})$, the dual vector space of $k[x]$ shifted down by $r$ degrees,
where $k[x]^{\vee}:=\operatorname{Hom}_{k}(k[x],k)$.  It follows that
\begin{center}
 $H_{(x)}^{i}(k[x])$=$\left\{
                      \begin{array}{ll}
                        \Sigma_{-r}(k[x]^{\vee})=k[x,x^{-1}]/k[x], & \hbox{if i=1 ;} \\
                        0, & \hbox{otherwise.}
                      \end{array}
                    \right.$
\end{center}
\end{ex}

\begin{ex}

For $R=k[x,y]$, a polynomial ring over a field $k$ with indeterminates $x,y$ of degree
$r,s$ and $I=(x,y)$, we have
\begin{eqnarray*}
  K^{\infty}(I;R) &=& K^{\infty}(x;R)\otimes_{R} K^{\infty}(y;R)  \\
   &=& (R \longrightarrow R[\frac{1}{x}]\oplus R[\frac{1}{y}]\longrightarrow
   R[\frac{1}{xy}])
\end{eqnarray*}
As in the previous example, we illustrate the picture of the Koszul complex for this ring as
below.\\ 
\setlength{\unitlength}{1cm}
\begin{picture}(10,10)
\put(2,0.2){\textbf{Figure 3.2}: Koszul complex of $k[x,y]$ at $(x,y)$.}
\put(1,9.3){$R$} \put(2.3,9.65){$\{i,i\}$} \put(8,9.65){$<i,-i>$}
\put(4.7,9.3){$R[\frac{1}{x}]$} \put(6.2,9.3){$R[\frac{1}{y}]$}
\put(11.3,9.3){$R[\frac{1}{xy}]$} \put(1.6,9.45){\vector(1,0){2.75}}
\put(5.7,9.3){$\oplus$} \put(5.7,5.5){$\bigoplus$} \put(7.3,9.45){\vector(1,0){3.65}}

\put(0.3,6.5){\vector(0,1){2}} \put(0.3,6.5){\vector(1,0){2}} \put(0,8.6){ $y$}
\put(2.3,6.3){ $x$} \put(0.2,6.2){\tiny $1$} \put(0.7,6.2){\tiny $x$}
\put(1.2,6.2){\tiny $x^{2}$} \put(1.7,6.2){\tiny $x^{3}$} \put(0,6.95){\tiny $y$}
\put(0,7.45){\tiny $y^{2}$} \put(0,7.95){\tiny $y^{3}$}
\multiput(0.3,6.5)(0,0.5){4}%
{\circle*{0.07}}
\multiput(0.8,6.5)(0,0.5){4}%
{\circle*{0.07}}
\multiput(1.3,6.5)(0,0.5){4}%
{\circle*{0.07}}
\multiput(1.8,6.5)(0,0.5){4}%
{\circle*{0.07}}

\put(6,6.5){\vector(1,0){2}} \put(6,6.5){\vector(-1,0){2}} \put(6,6.5){\vector(0,1){2}}
\multiput(4.5,6.5)(0.5,0){7}%
{\circle*{0.07}}
\multiput(4.5,7)(0.5,0){7}%
{\circle*{0.07}}
\multiput(4.5,7.5)(0.5,0){7}%
{\circle*{0.07}}
\multiput(4.5,8)(0.5,0){7}%
{\circle*{0.07}} \put(5.6,8.5){ $y$} \put(8,6.3){ $x$}
\put(5.9,6.2){\tiny $1$} \put(6.4,6.2){\tiny $x$} \put(6.9,6.2){\tiny $x^{2}$}
\put(7.4,6.2){\tiny $x^{3}$} \put(5.7,6.95){\tiny $y$} \put(5.7,7.45){\tiny $y^{2}$}
\put(5.7,7.95){\tiny $y^{3}$} \put(5.4,6.2){\tiny $\frac{1}{x}$} \put(4.9,6.2){\tiny
$\frac{1}{x^{2}}$} \put(4.4,6.2){\tiny $\frac{1}{x^{3}}$}

\put(6,3){\vector(0,1){2}} \put(6,3){\vector(0,-1){2}} \put(6,3){\vector(1,0){2}}
\multiput(6,1.5)(0,0.5){7}%
{\circle*{0.07}}
\multiput(6.5,1.5)(0,0.5){7}%
{\circle*{0.07}}
\multiput(7,1.5)(0,0.5){7}%
{\circle*{0.07}}
\multiput(7.5,1.5)(0,0.5){7}%
{\circle*{0.07}} \put(5.6,5){ $y$} \put(8,2.8){ $x$}
\put(5.7,3){\tiny $1$} \put(6.4,2.7){\tiny $x$} \put(6.9,2.7){\tiny $x^{2}$}
\put(7.4,2.7){\tiny $x^{3}$} \put(5.7,3.45){\tiny $y$} \put(5.7,3.95){\tiny $y^{2}$}
\put(5.7,4.45){\tiny $y^{3}$} \put(5.7,2.5){\tiny $\frac{1}{y}$} \put(5.6,2){\tiny
$\frac{1}{y^{2}}$} \put(5.6,1.5){\tiny $\frac{1}{y^{3}}$}

\put(12,6.5){\vector(0,1){2}} \put(12,6.5){\vector(0,-1){2}}
\put(12,6.5){\vector(1,0){2}} \put(12,6.5){\vector(-1,0){2}}
\multiput(10.5,6.5)(0.5,0){7}%
{\circle*{0.07}}
\multiput(10.5,7)(0.5,0){7}%
{\circle*{0.07}}
\multiput(10.5,7.5)(0.5,0){7}%
{\circle*{0.07}}
\multiput(10.5,8)(0.5,0){7}%
{\circle*{0.07}}
\multiput(12,6)(0,-0.5){3}%
{\circle*{0.07}}
\multiput(12.5,6)(0,-0.5){3}%
{\circle*{0.07}}
\multiput(13,6)(0,-0.5){3}%
{\circle*{0.07}}
\multiput(13.5,6)(0,-0.5){3}%
{\circle*{0.07}}
\multiput(11.5,6)(0,-0.5){3}%
{\circle{0.07}}
\multiput(11,6)(0,-0.5){3}%
{\circle{0.07}}
\multiput(10.5,6)(0,-0.5){3}%
{\circle{0.07}} \put(11.6,8.5){ $y$} \put(14,6.3){ $x$} \put(11.7,6.6){\tiny $1$}
\put(12.4,6.2){\tiny $x$} \put(12.9,6.2){\tiny $x^{2}$} \put(13.4,6.2){\tiny $x^{3}$}
\put(11.7,6.95){\tiny $y$} \put(11.7,7.45){\tiny $y^{2}$} \put(11.7,7.95){\tiny $y^{3}$}
\put(11.4,6.3){\tiny $\frac{1}{x}$} \put(10.9,6.3){\tiny $\frac{1}{x^{2}}$}
\put(10.4,6.3){\tiny $\frac{1}{x^{3}}$} \put(11.7,6){\tiny $\frac{1}{y}$}
\put(11.6,5.5){\tiny $\frac{1}{y^{2}}$} \put(11.6,5){\tiny $\frac{1}{y^{3}}$}

\end{picture}

From this figure, it is easy to see that this cochain complex is exact at the first and
second term.  Thus, $H_{I}^{0}(R)$ and $H_{I}^{1}(R)$ are zero.  For the third term, the
cokernel of the $<i,-i>$ map is given by all the circled points in the third quadrant, which is
isomorphic to $\Sigma_{-(r+s)}(R^{\vee})$.  Hence,
\begin{center}
$H_{I}^{i}(R)$=$\left\{
  \begin{array}{ll}
    0, & \hbox{if i=0;} \\
    0, & \hbox{if i=1;} \\
    \Sigma_{-(r+s)}(R^{\vee}), & \hbox{if i=2;} \\
    0, & \hbox{otherwise.}
  \end{array}
\right.$
\end{center}
\end{ex}

For an $R$ module $M$, we say that $H_{I}^{*}(M)=H_{I}^{*}(R;M)$ and $H_{I}^{*}(R)=H_{I}^{*}(R;R)$.\\

Since we work only with Noetherian rings, we recollect some basic properties relating to calculation of local cohomology for modules over such rings \cite{gr}.
\begin{prop}
Let $R$ be a commutative  Noetherian ring (with unity), $I\lhd R$ and  $M$ a module over $R$.
\begin{description}
  \item[$1.$] If $L$ and $N$ are $R$ modules such that $0\longrightarrow
      L\longrightarrow M \longrightarrow N \longrightarrow 0$ is a short exact
      sequence, then we have an induced long exact sequence
\begin{center}
$0\longrightarrow H_{I}^{0}(L)\longrightarrow H_{I}^{0}(M)\longrightarrow
H_{I}^{0}(N)\longrightarrow H_{I}^{1}(L)\longrightarrow H_{I}^{1}(M)\longrightarrow
H_{I}^{1}(N)\longrightarrow ...$
\end{center}
  \item[$2.$] For $J$ an ideal of $R$, if $\sqrt{J}=\sqrt{I}$ then
      $H_{I}^{i}(M)=H_{J}^{i}(M)$ for all $i$.

 \item[$3.$] Let $\Lambda$ be a directed set and
     $\{M_{\lambda}\}_{\lambda\in\Lambda}$ a direct system of $R$ modules.\\
 Then $\displaystyle\lim_{ \to \lambda}H_{I}^{i}(M_{\lambda})\cong
 H_{I}^{i}(\displaystyle\lim_{ \to \lambda}M_{\lambda}).$
\end{description}
\end{prop}

\section{The local cohomology spectral sequence}

We can now give the spectral sequence \cite{3} for real connective $K-$theory of a finite group $G$. We remark there are analogues for the periodic and complex cases as well.

\begin{thm}
There is a spectral sequence of the form:
\begin{center}
$E_{s,t}^{2}=H_{I}^{-s}(ko^{*}(BG))_{t} \Longrightarrow ko_{s+t}(BG)$
\end{center}
with differentials $d^{r}:E_{s,t}^{r}\rightarrow E_{s-r,t+r-1}^{r}$ and augmentation ideal
$I =\ker(ko^{*}(BG)\rightarrow ko^{*})$. Here $I$ is Groethendieck's local cohomology functor, and $G$ is a finite group.
\end{thm}
It is known that local cohomology vanishes in degrees above the rank $r$ of the group \cite{3}, so that this is a finite spectral sequence, and indeed in our examples, the $E_{\infty}$ page of the spectral sequence has only three columns, namely the zero-th, first and $r-$th.\\
The idea for calculations is to consider the following commutative diagram of spectral sequences:\\

\setlength{\unitlength}{0.9cm}
\begin{picture}(10,4)
\linethickness{0.2mm}

\put(1,1.75){$H^{*}(BG;\mathbb{F}_{2})$}
\put(5,3){$ku^{*}(BG)$}
\put(10,3){$ku_{*}(BG)$}
\put(5,0.5){$ko^{*}(BG)$}
\put(10,0.5){$ko_{*}(BG),$}

\put(3.1,2){\vector(2,1){1.8}}
\put(6.7,3.1){\vector(1,0){3.2}}
\put(6.7,0.6){\vector(1,0){3.2}}
\put(5.7,2.8){\vector(0,-1){1.8}}
\put(10.7,2.8){\vector(0,-1){1.8}}
\put(3.3,2.5){ASS}
\put(8,3.2){LCSS}
\put(8,0.7){LCSS}
\put(5.8,1.8){BSS}
\put(10.8,1.8){BSS}
\end{picture}\\
where ASS refers to the Adams spectral sequence, BSS refers to the $\eta$-Bockstein spectral sequence and LCSS refers to the local cohomology spectral sequence.\\
We remark that all the terms in the square, including $ko_*(BG)$, may be computed by using the Adams spectral sequence directly, and indeed this is what we do for the dihedral groups in \cite{mjam}. In general though, it would appear that combining all the known methods is most effective.\\

From this diagram, we see that to obtain $ko_{*}(BG)$ via the local cohomology spectral sequence with input $H^{*}(BG;\mathbb{Z}_{2})$ we can proceed in two ways around the square, namely $ASS\longrightarrow BSS\longrightarrow GSS$ and $ASS \longrightarrow GSS \longrightarrow BSS$.  The first way is suitable for tackling the Gromov-Lawson-Rosenberg conjecture, since we have:
\begin{lem}\label{lem Ap}(\cite{3}, Lemma 2.7.1.) The image of $A \circ p$ is isomorphic to the 0-column at the $E_{\infty}$-page of the Local cohomology spectral sequence for $ko_{*}(BG)$, where $G$ is a finite group.
The kernel of $A \circ p$ has a filtration with subquotients given by the higher columns
at the $E_{\infty}$-page.
\end{lem}

\section{An example}
In order to illustrate these methods, we conclude this chapter by giving a brief sketch of the calculation in \cite{3} for $ko_*(BG)$ when $G$ is the Klein $4$-group $V(2)=\{e,x,y,z\}$.\\
For a group $G$, we denote the representation ring by $RU=RU(G)$, and the augmentation ideal by $JU=JU(G)$. Throughout $2^a$ will denote the elementary abelian $2-$group of rank $a$, while $[2^a]$ will denote the cyclic group of order $2^a$.\\

Firstly, recall that $H^{*}(BV(2); \Z_2)=\Z_2[x_1, x_2]$. We introduce the notation $PC=\Z_2[y_1, y_2]$ where $y_i=x_i^2$, and $PP=\Z_2[z_1, z_2]$ where $z_i=y_i^2$. The real connective cohomology can be described by the following exact sequence (\cite{3}, chapter 9):
$$0 \rightarrow TO \rightarrow ko^*(BV(2)) \rightarrow QO \rightarrow  0$$
where $QO$ is the image in the real periodic $K-$cohomology $KO^*(BV(2))$. In positive degrees $QO$ agrees with $KO^*(BV(2))$, while in negative degrees it is generated by $JSp$ (the symplectic part of the augmentation ideal of the representation ring) in degree $-4$.\\
$TO$ is detected in ordinary cohomology, and is given by
$$TO=PP(-6) \oplus PP(-12)$$
where $PP(a)$ means $PP$ shifted up $a$ degrees.

Thus in order to apply the local cohomology spectral sequence, we now have to calculate local cohomology. Again, we follow \cite{3}, chapter 12. There are short exact sequences
$$ 0 \rightarrow T \rightarrow ko^*(BV(2)) \rightarrow \overline{QO} \rightarrow  0$$
where $\overline{QO}$ is now the image in periodic \emph{complex} $K-$theory, and
$$0 \rightarrow \tau \rightarrow T \rightarrow TO \rightarrow 0$$
where $\tau$ consists of the eta multiples.\\
Now, $\tau$ is bounded below, and thus torsion, so that
$$H_I^{*}(\tau)=H_I^{0}(\tau)=\tau$$
and by comparison with periodic $K-$theory we know $\tau=2^4$, the elementary abelian group of rank $4$, in positive degrees $1,2 \mod 8$, and $\tau=0$ else.\\
Further, in positive degrees $0,4 \mod 8$, it is known that $\overline{QO}=RO$ and $\overline{QO}=RSp=2RO$ respectively (note that $RU=RO$ here). In negative degrees it is given by $2JU, JU^4, JU^6, \cdots$, in degrees $-4, -8, -12 \cdots$. The augmentation ideal is the radical of the ideal generated by $p_*^2 \in JU^4$, where $p_*$ is the representation with character $(0444)$. Thus local cohomology may be calculated by inverting $p_*$, and we get the exact sequence:
$$0 \rightarrow H_{JO}^0(\overline{QO}) \rightarrow \overline{QO} \rightarrow \overline{QO}[1/p^*] \rightarrow H_{JO}^1(\overline{QO}) \rightarrow 0$$
The kernel of inverting $p_*$ just consists of characters supported at $e$, so $H_{JO}^0(\overline{QO})= \Z$ for $0 \leq n \equiv 0 \mod 4$.\\
For $H^1$, we observe that for $k \geq 1$, $JU^{2k}$ is generated by the characters $(04^k00),(004^k0),(0004^k)$, meaning that multiplying by $p_*$ is an isomorphism in degrees $-8$ and below, meaning $H^1$ is in degrees $-4$ and above. Explicit calculation, together with the observation that for $i \geq 0$, the $i+8th$ group has order $16^3$ times the order of the $i-$th group (the determinant of $p_*^2$), then gives us the following conclusion:

\begin{center}
$$H^1_{JO}(\overline{QO})_i=\begin{cases}
    $$[2^{4k+1}]^2 \oplus [2^{4k}] \hbox{if i=8k} \geq 0$$ \\
    $$[2^{4k+4}]^2 \oplus [2^{4k+3}] \hbox{if i=8k+4} \geq 4$$\\
    $$[2]  \hbox{ if i=-4}$$\\
    $$0 \hbox{ otherwise.}$$
  \end{cases}$$
\end{center}
where $[a]$ means the cyclic group of order $a$.\\

Finally, by local duality it follows that
$$H_I^*(TO)=H_I^{2}(TO)=PP^{\vee}(-4) \oplus PP^{\vee}(2)$$
where the $\vee$ denotes the dual vector space. Thus we can now display what is referred to in \cite{3} as the $E_{1 \frac{1}{2}}$ page of the local cohomology spectral sequence, which essentially just means that we tabulate the local cohomology groups we have just calculated, and determining the differentials $d_1$ and $d_2$ will give us the $E_2$ page, which we know must be the $E_{\infty}$ page.

\setlength{\unitlength}{1cm}
\begin{picture}(20,19)
\multiput(0.2,2)(0,0.5){32}%
{\line(1,0){12.8}}
\put(3,2){\line(0,1){16}}
\put(6,2){\line(0,1){16}}
\put(10,2){\line(0,1){16}}
\put(13,2){\line(0,1){16}}
\multiput(11.6,6.6)(0,0.5){3}%
{$0$}
\multiput(11.6,10.6)(0,0.5){3}%
{$0$}
\multiput(11.6,14.6)(0,0.5){3}%
{$0$}
\multiput(11.6,2.1)(0,0.5){4}%
{$0$}
\multiput(11.6,5.6)(0,4){3}%
{$0$}
\multiput(11.6,4.6)(0,0.5){2}%
{$2^{4}$}
\multiput(11.6,8.6)(0,0.5){2}%
{$2^{4}$}
\multiput(11.6,12.6)(0,0.5){2}%
{$2^{4}$}
\multiput(11.6,16.6)(0,0.5){2}%
{$2^{4}$}
\multiput(11.6,4.1)(0,2){7}%
{$\mathbb{Z}$}

\multiput(8,2.6)(0,1){5}%
{$0$}
\multiput(8,9.6)(0,1){3}%
{$0$}
\multiput(8,13.6)(0,1){3}%
{$0$}
\multiput(8,3.1)(0,1){1}%
{$0$}
\put(8,5.1){$0$}
\put(7.9,2.1){$[2]$}
\put(6.7,4.1){$[2] \oplus [2] \oplus [1]$}
\put(6.7,6.1){$[16]\oplus [16]\oplus[8]$}
\put(6.7,8.1){$[32]\oplus[32]\oplus[16]$}
\put(6.7,10.1){$[2^8]\oplus [2^8]\oplus[2^7]$}
\put(6.7,12.1){$[2^9]\oplus [2^9]\oplus[2^8]$}
\put(6.7,14.1){$[2^{12}]\oplus[2^{12}]\oplus[2^{11}]$}
\put(6.7,16.1){$[2^{13}]\oplus[2^{13}]\oplus[2^{12}]$}

\multiput(1.8,2.1)(0,0.5){31}%
{$0$}
\multiput(1.8,3.1)(0,1){1}%
{$0$}
\multiput(8,7.1)(0,2){6}%
{$0$}
\multiput(8,7.6)(0,2){5}%
{$0$}
\multiput(8,8.6)(0,4){3}%
{$0$}
\put(8,17.6){$\vdots$}
\put(4.5,17.6){$\vdots$}
\put(1.8,17.6){$\vdots$}
\put(11.8,17.6){$\vdots$}
\put(4.5,2.1){$2$}
\put(4.5,4.1){$2^2$}
\put(4.5,5.1){$2$}

\put(4.5,6.1){$2^{3}$}
\put(4.5,7.1){$2^2$}
\put(4.5,8.1){$2^{4}$}
\put(4.5,9.1){$2^{3}$}
\put(4.5,10.1){$2^{5}$}
\put(4.5,11.1){$2^{4}$}
\put(4.5,12.1){$2^{6}$}
\put(4.5,13.1){$2^{5}$}
\put(4.5,14.1){$2^{7}$}
\put(4.5,15.1){$2^{6}$}
\put(4.5,16.1){$2^{8}$}
\put(4.5,17.1){$2^{7}$}
\multiput(4.5,4.6)(0,1){13}%
{$0$}
\put(4.5,2.6){$0$}
\put(4.5,3.1){$0$}
\put(4.5,3.6){$0$}
\multiput(6.1,2.1)(0,2){8}%
{\vector(-1,0){1.0}}

\put(6.2,2.1){$d_1$}
\put(6.2,4.1){$d_1$}
\put(6.2,6.1){$d_1$}
\put(6.2,8.1){$d_1$}
\put(6.2,10.1){$d_1$}
\put(6.2,12.1){$d_1$}
\put(6.2,14.1){$d_1$}
\put(6.2,16.1){$d_1$}

\put(13.3,2.1){-4}
\put(13.3,2.6){-3}
\put(13.3,3.1){-2}
\put(13.3,3.6){-1}
\put(13.3,4.1){0}
\put(13.3,4.6){1}
\put(13.3,5.1){2}
\put(13.3,5.6){3}
\put(13.3,6.1){4}
\put(13.3,6.6){5}
\put(13.3,7.1){6}
\put(13.3,7.6){7}
\put(13.3,8.1){8}
\put(13.3,8.6){9}
\put(13.2,9.1){10}
\put(13.2,9.6){11}
\put(13.2,10.1){12}
\put(13.2,10.6){13}
\put(13.2,11.1){14}
\put(13.2,11.6){15}
\put(13.2,12.1){16}
\put(13.2,12.6){17}
\put(13.2,13.1){18}
\put(13.2,13.6){19}
\put(13.2,14.1){20}
\put(13.2,14.6){21}
\put(13.2,15.1){22}
\put(13.2,15.6){23}
\put(13.2,16.1){24}
\put(13.2,16.6){25}
\put(13.2,17.1){26}
\put(13.2,17.6){27}

\put(12.5,18.2){degree(t)}

\put(6.7,1.2){$H^{1}_{I}(\overline{QO})$}
\put(10.5,1.2){$H^{0}_{I}(\tau) \oplus H^{0}_{I}(\overline{QO})$}
\put(3.8,1.2){$H^{2}_{I}(TO)$}
\put(1.5,1.2){$H^{\epsilon\geq3}_{I}$}

\put(0.2,0.2){where$[n]:=$ cyclic group of order $n$, $2^{r}$:= elementary abelian group of order $r$.}
\put(10.2,4.6){\line(-6,1){4.6}}
\put(10.3,4.6){$d_2$}
\linethickness{0.5mm}
\put(10,4){\line(1,0){3}}
\put(10,4.5){\line(-1,0){4}}
\put(6,5){\line(-1,0){3}}
\put(10,4){\line(0,1){0.5}}
\put(6,4.5){\line(0,1){0.5}}
\put(3,5){\line(0,1){0.5}}

\end{picture}\\

Next we consider differentials. Since $ko_*$ is zero in negative degrees, and $ko_0= \Z$, we have that the differentials $d_1:H_{-4}^{1} \rightarrow H_{-4}^{2}$ and $d_1:H_{0}^{1} \rightarrow H_{0}^{2}$ are isomorphisms, and the long differential $d_2:H_{1}^{0} \rightarrow H_{2}^{2}$ is surjective, and thus has kernel $2^3$.\\
The other differentials may be easily described \cite{3}:
\begin{lem}
There are no more differentials leaving the zero column.
\end{lem}
\begin{proof}
We know $ko_*$ is a summand, so we only need to deal with dimensions $1,2 \mod 8$. The latter case is clear since $H^2$ is zero in odd degrees, and since $ko_*(BV(2))[\beta^{-1}]=KO_*(BV(2))$, it follows that in dimensions $1 \mod 8$, all the $2^4$ survives in sufficiently high degrees. The image in periodic $K-$theory is a $ko_*$ module, so we need only consider degree $9$. The differentials are maps of $ko^*(BV(2))$ modules, so that every element of $H_{10}^2$ is detected by some map $H_{10}^2 \rightarrow H_{6}^2$ induced by multiplication by $\lambda \in ko*4(BV(2))$, so that if $d^2(x) \neq 0 \in H_{10_2}$ detected by $\lambda$, then there would have to be a non-zero $d^2(\lambda x)$. However $\lambda x$ is in the zero group.
\end{proof}

\begin{lem}
The differential $d_1:H_{4k}^{1} \rightarrow H_{4k}^{2}$ is of maximal rank. Thus it has rank $3$ for $k \geq 1$.
\end{lem}
\begin{proof}
It suffices to see this for $k=1$ by using the $ko^*(BV(2))$ module structure. In the $k=1$ case, we know $|ko_2(BV(2))| \leq 2^4$ by the Atiyah-Hirzebruch spectral sequence, which implies the differential $d^1$ has rank $3$.
\end{proof}
Thus we can now read off the values of $ko_*(BV(2))$ from the collapsed spectral sequence. It is easy to see that all extensions split, and the kernel $Ker(Ap)$ can be read off from the $H^1$ and $H^2$ columns.\\
\begin{center}
\begin{tabular}{|c|c|c|}
\hline

$n$&$ko_n(BV(2))$ & $Ker(Ap)$\\
\hline
$1$ & $2^3$ & $0$\\

$2$ & $2^4$ & $0$\\

$3$ & $[8]^2 \oplus [4]$ & $[8]^2 \oplus [4]$\\

$4$ & $\Z \oplus 2^2$ & $2^2$\\

$5$ & $0$ & $0$\\

$6$ & $2$ & $2$\\

$7$ & $[16]^2 \oplus [8]$ & $[16]^2 \oplus [8]$\\
\hline

$8m+0 \geq 8$ & $\Z \oplus 2^{2m+1}$ & $2^{2m+1}$\\

$8m+1 \geq 9$ & $2^4$ & $0$\\

$8m+2$ & $2^{2m+4}$ & $2^{2m}$\\

$8m+3$ & $[2^{4m+3}]^2 \oplus [2^{4m+2}]$ & $[2^{4m+3}]^2 \oplus [2^{4m+2}]$\\

$8m+4$ & $\Z \oplus 2^{2m+2}$ & $2^{2m+2}$\\

$8m+5$ & $0$ & $0$\\

$8m+6$ & $2^{2m+1}$ & $2^{2m+1}$\\
$8m+7$ & $[2^{4m+4}]^2 \oplus [2^{4m+3}]$ & $[2^{4m+4}]^2 \oplus [2^{4m+3}]$\\
\hline
\end{tabular}
\end{center}

\chapter{The eta invariant and groups with periodic cohomology}
The aim of this chapter is to give a thorough outline of the Botvinnik, Gilkey and Stolz proof \cite{2} for cyclic and quaternion groups. Indeed, we use these methods in the subsequent chapters in order to detect the part of the kernel $Ker(Ap)$ detected in periodic K-theory by real projective spaces and lens spaces. The orders of these elements in periodic K-theory are detected by the eta invariant.\\
Atiyah, Patodi and Singer showed that there is a formula for the index of the Dirac operator $D$ for a manifold $W$ with boundary $M$, analogous to the usual index formula, but with a correction term $\eta(D(M))$, known as the eta invariant, depending only on the boundary, see \cite{aps1},\cite{aps2}, and \cite{aps3}. Here is a brief outline of the theory:\\
For a compact spin manifold $M$ the Dirac operator $D(M):\Gamma(S) \rightarrow \Gamma(S)$ is a first order differential operator acting on the space of sections of the spinor bundle $S$ of $M$ \cite{11}. The index of an operator $D$ is defined as follows, if it exists:
$$Index(D)=Dim(Ker(D))-Dim(Cokernel(D))$$
More generally, we can twist by a vector bundle $E$ over $M$ equipped with a unitary connection to get a first order elliptic operator $D(M)\otimes E:\Gamma(S\otimes E) \rightarrow \Gamma(S\otimes E)$. Further, if the dimension of $M$ is even, then we have a decomposition $S=S^{+} \oplus S^{-}$ and $D=D^{+}+D^{-}$ where $D^{\pm}(M)\otimes E:\Gamma(S^{\pm}\otimes E) \rightarrow \Gamma(S^{\pm}\otimes E)$.\\
Now if $M$ is the boundary of a manifold $W$ over which $E$ extends, then under certain suitable global boundary conditions, the index of $D^{+}(W \otimes E)$ is well defined and we may apply the Index Theorem for manifolds with boundary:
$$Index(D^{+}(W \otimes E))=\int_{W} \hat{A}(W)ch(E)-\eta(D(M)\otimes E)$$
This is like the usual Index Theorem, with the integrand a polynomial in the Pontrjagin forms of $W$ and the Chern character of $E$, but with a correction term $\eta(D(M)\otimes E)$, depending only on the boundary, called the \emph{eta invariant}, which is an indicator of the asymmetry of the spectrum of the operator $D(M) \otimes E$ with respect to the origin. Because the eta invariant of a boundaryless manifold is an integer, the eta invariant is independent of the choice of $W$ modulo the integers, and we write $\eta(M)$ for the value in $\R/\Z$ thus obtained.\\

Given a representation $\rho$ of a discrete group $\pi$ and a map from a manifold $f:M \to B\pi$, we can form a vector bundle $V_{\rho}:\widetilde{M}\times_{\pi} \rho \to M$, where $\widetilde{M}$ is the $\pi$ cover of $M$ classified by $f$.\\
We now consider the Dirac operator $D(M,f,\rho)$, which is the Dirac operator of $M$ twisted by $V_{\rho}$, and its eta invariant $\eta(M,f)(\rho)$. Since the eta invariant is additive with respect to direct sums of representations, we can extend this definition to include virtual representations.\\
We recall from \cite{9} that there is a geometric description of periodic K-theory of a space $X$ as follows, where the isomorphism is induced by $pD$:
$$\Omega_{*}^{spin}(X)/T_{*}(X)[B^{-1}] \cong KO_{*}(X)$$
where, analogous to the simply connected case, $T_{*}(X)$ is the subgroup generated by quaternion projective bundles with structure group $PSp(3)$ with maps into $X$, and $B=B^8$ is a Bott manifold, which is any simply connected manifold with $\hat{A}(B)=1$.
Using this description of the periodic K-theory, we have the following result \cite{2}:
\begin{thm}
Let $\rho$ be a virtual representation of $\pi$ of virtual dimension zero. Then for a spin manifold $M$ with a map $f:M \rightarrow B\pi$, the map $(M,f) \mapsto \eta(M,f)(\rho) \in \R$ gives rise to a well defined homomorphism
$$\eta(\rho):\Omega_{n}^{spin}(B\pi)\rightarrow \R/\Z;\eta(\rho):KO_{n}(B\pi)\rightarrow \R/\Z$$
which sends $[f:M \rightarrow B\pi]$ to $\eta(M,f)(\rho)$ reduced modulo $\Z$.\\
 Further if $\rho$ is real and $n \equiv 3 \mod 8$, or if $\rho$ is quaternion and $n \equiv 7 \mod 8$, then we can replace the range of $\eta(\rho)$ by $\R/2\Z$, meaning the map which sends $[f:M \rightarrow B\pi]$ to $\eta(M,f)(\rho)$ reduced modulo $2\Z$ is still well defined.
\end{thm}
The following Theorem in \cite{2}, is deduced from work of Donnelly in \cite{5}, and is the first main tool for actually computing some eta invariants.
\begin{thm}
Let $\rho$, $\pi$ be as above, and $\tau:\pi \rightarrow U(m)$, be a fixed point free representation. Assume there exists a representation $det(\tau)^{1/2}$ of $\pi$ whose tensor square is $det(\tau)$. Then letting $M=S^{2n-1}/\tau(\pi)$ with the inherited structures, we have
$$\eta(M)(\rho)=\mid \pi \mid^{-1} \sum_{1 \neq g \in \pi} \frac{Trace(\rho(g))det(\tau(g))^{1/2}} {(det(I-\tau(g)))}$$
\end{thm}

 This is the main Theorem for computing eta invariants of the manifolds with periodic fundamental groups we are interested in, namely cyclic and quaternion lens spaces.\\
 The first step for us in verifying the GLR conjecture for an arbitrary finite group $G$ will be to choose some periodic subgroup $H$, choose an appropriate lens space with fundamental group $H$ and classifying map $u$ for the universal cover, and then understand the order of $i \circ D([L,u]) \in ko_*(BG)$, where $i$ is induced by the inclusion $ H\hookrightarrow G$. To do this we will use that the eta invariant is natural with respect to inclusions, see \cite{aps1}, \cite{2} and \cite{5}. This means that if we have an inclusion of groups $ f:H\hookrightarrow G$, and a class represented by a manifold $[M] \in \Omega_{n}^{spin}(BH)$, then for a zero-dimensional virtual representation $\rho$ of $G$ the pull-back bundle over $M$ is determined by restricting the representation to the subgroup. It follows that we then have $\eta([M])(f^*(\rho))=\eta(f_*([M]))(\rho)$, as should be clear from the following pull-back diagram:\\
$$
\xymatrix{
i^*(EH \times_H f^*(\rho)) \ar[r] \ar [d] & EH \times_H f^*(\rho) \ar[r] \ar [d] & EG \times_G \rho \ar[d] \\
M \ar[r]^{i} & BH \ar[r]^{f} & BG \\
}
$$
Note also that if $\rho_1, \cdots ,\rho_j$ are virtual representations of a finite group $G$ of virtual dimension zero, and $M_1, \cdots ,M_k$ are $n-$dimensional manifolds equipped with maps to $BG$, then we can define $k$ different $j-$tuples as follows:
$$\overrightarrow{\eta}(M_i)=(\eta(M_i)(\rho_1), \cdots , \eta(M_i)(\rho_j))$$
where each entry lies in $\R/\Z$ or $\R/2\Z$ as appropriate. We can then directly try and calculate the order of the subgroup spanned by these vectors, to see if it is sufficiently large, which means we compare the order of this group with the order of $Ker(Ap) \subset ko_n(BG)$ . For example, we could perform elementary row or column operations on the resulting $k \times j$ matrix.\\
In particular, if $j=k$, then the inverse of the determinant of the resulting $j \times j$ matrix gives a lower bound on the order of the group spanned. This follows since we can perform row operations to get an upper or lower triangular matrix, so that multiplying the orders of the diagonal entries immediately gives a lower bound.\\
The rest of the chapter will outline the proof of the GLR conjecture for cyclic and quaternion groups, and in subsequent chapters we will make similar calculations by considering inclusions from such subgroups, in order to detect the part of $Ker(Ap)$ that is detected in periodic K-theory.
\section{Cyclic groups}
Botvinnik, Gilkey and Stolz \cite{2}, proved the GLR conjecture for groups with periodic cohomology by making explicit eta invariant calculations, by using Theorem $2.0.3$ and Theorem $2.1.1$ below.\\
We first wish to consider the cyclic groups of order $l=2^k$, which we identify with the subgroup of $S^1$ consisting of $l$-th roots of unity:
$$C_l=\{\lambda \in S^1 | \lambda^l=1\}$$
For an integer $a$ we let $\rho_a$ be the representation of $S^1$ where $\lambda \in S^1$ acts by multiplication by $\lambda^a$. For a tuple of integers $\overrightarrow{a}=(a_1, \cdots ,a_t)$, the representation $\lambda^{a_1 }\oplus \cdots \oplus \lambda^{a_t}$ restricts to a free $C_l$ action on $S^{2t-1}$ if and only if all the $a_j$ are odd. Let $t=2i$ be even, and define the quotient manifold
$$X^{4i-1}(l,\overrightarrow{a})=S^{2t-1}/(\rho_{a_1} \oplus \cdots \rho_{a_t})(C_l)$$
This is a lens space, and inherits a natural spin structure in dimensions $3 \mod 4$ \cite{2}. This follows because the second Stiefel-Whitney class of $X$ is the$\mod 2$ reduction of the first Chern class of the determinant line bundle. Now $c_1(det(\tau))$ is even if and only if we can take the square root of the determinant line bundle, which in turn is possible if and only if we can take the square root of the representation $det(\tau)$. For cyclic groups of even order this possible if and only if we are working in dimensions $3 \mod 4$, see \cite{2} for details.\\
Note then that in dimensions $4k+1$ the corresponding construction does not yield a spin manifold. So, consider the vector bundle $H \otimes H \oplus (2k-1)\C \rightarrow S^2$, where $H$ is the Hopf line bundle over $S^2$. As above, for a tuple of integers $\overrightarrow{a}=(a_1, \cdots ,a_t)$, the representation $\lambda^{a_1 }\oplus \cdots \oplus \lambda^{a_t}$ restricts to a free $C_l$ action on the sphere bundle $S(H \otimes H \oplus (2k-1) \C \rightarrow S^2)$ if and only if all the $a_j$ are odd. Again let $t=2i$ be even. Then the quotient
$$X^{4i+1}(l,\overrightarrow{a})=S(H \otimes H \oplus (2k-1)\C \rightarrow S^2)/C_l$$
is a spin manifold of dimension $4k+1$ which is a lens space bundle over the two-sphere. A general formula for the eta invariants of manifolds of the form $S(H_1 \oplus \cdots \oplus H_k \rightarrow S^2)/C_l$, where the $H_i$ are complex line bundles and the $C_l$ action is specified by a tuple of integers $\overrightarrow{a}$ as above, is given in \cite{2}:
\begin{thm}
Let $\rho \in R_0(\pi)$ and $M=S(H_1 \oplus \cdots \oplus H_k \rightarrow S^2)/C_l(\overrightarrow{a})$ as above, with $\overrightarrow{a}=(a_1, \cdots ,a_t)$ and $t=2i$ even. Then
$$\eta(M)(\rho)=l^{-1}\sum_{1 \neq \lambda \in C_l} Trace(\rho(\lambda) \frac{\lambda^{(a_1+ \cdots a_{2i})/2}}{(1-\lambda^{a_1})\cdots (1-\lambda^{a_{2i}})}. \sum_j \frac{1}{2} c_1(H_j)[\CP^1] \frac{1+\lambda^{a_j}}{1-\lambda^{a_j}}$$
\end{thm}

 It is then clear that $\pi_1(X^{4i \pm 1}(l,\overrightarrow{a}))=C_l$, giving these manifolds natural $C_l$ structures, meaning that using the classifying map $u$ for the universal covers, we can view the $[X^{4i\pm 1}(l,\overrightarrow{a}),u]$ as elements of $\Omega_{4i\pm 1}^{spin}(BC_l)$. In order to apply Theorems $2.0.4$ and $2.1.1$, we can simply define $det(\rho_{a_1} \oplus \cdots \oplus \rho_{a_t})^{1/2} :=\rho_{(a_1 + \cdots +a_{2i})/2}$. We say $L^{n}=X^{n}-X_{0}^{n} \in \Omega_{n}^{spin}(BC_l)$, where $X_{0}^{n}$ is the same manifold with trivial $C_l$ structure, meaning that we use the constant map to give a spin bordism class. We can then use the following formula: \\

 For $\lambda \in C_l$ we define $$f_{4i-1}(\overrightarrow{a_{2i}})(\lambda):=\lambda^{(a_1 + \cdots a_{2i})/2}(1-\lambda^{a_1})^{-1}\cdots(1-\lambda^{a_{2i}})^{-1}$$ and $f_{4i+1}(\overrightarrow{a_{2i}})(\lambda):=f_{4i-1}(\overrightarrow{a_{2i}})(\lambda)(1+\lambda^{a_1})(1-\lambda^{a_1})^{-1}$. Thus $\lambda^{(a_1 + \cdots a_{2i})/2}$ acts as the square root of the determinant of $\lambda^{a_1} \oplus \cdots \oplus \lambda^{a_{2i}}$, and we then have the following formula from the above two Theorems:
\begin{lem}
let $n=4k \pm 1$. Then we have\\
$\eta(L^{n}(l,\overrightarrow{b_{2i}}))(\rho)=l^{-1}\sum_{1 \neq \lambda \in C_l}f_{4i \pm 1}(\overrightarrow{b_{2i}})(\lambda)Trace(\rho(\lambda))$
\end{lem}
We now make the following observations, all from \cite{2}; see Proposition 5.1, Lemma 5.2 and Lemma 5.3:
\begin{lem}
For an even dimensional spin manifold $N$, and an odd dimensional manifold $M$ with a map $f:M \rightarrow B\pi$, with $\pi, \rho$ as above, we have $\eta(M \times N)(\rho)=\eta(M)(\rho)\hat{A}(N)$
\end{lem}
\begin{prop}
i) We have that $|ko_{8k+3}(BC_l)|=2(2l)^{2k+1};|ko_{8k+7}(BC_l)|=(2l)^{2k+2}$, and these groups are spanned by the collection of the images of the fundamental classes of the lens spaces $L^{4i-1}(l,\overrightarrow{a}))$.\\
ii) Further $|ko_{8k+1}(BC_l)|=2(l/2)^{2k+1};|ko_{8k+5}(BC_l)|=(l/2)^{2k+2}$ and the latter groups are spanned by the collection of the images of the fundamental classes of the lens spaces $L^{4i+1}(l,\overrightarrow{a})$. In dimension $n=8k+1$, we have that $|Ker(Ap)|=(l/2)^{2k+1}$, and for $n \geq 9$ the lens space bundles $L^{8k+1}(l,\overrightarrow{a})$ span a subspace of at least this order.
\end{prop}
We note here that in \cite{2}, only the orders of the groups $ko_{*}(BC_l)$ are calculated, by using the Atiyah-Hirzebruch spectral sequence. The extension problems for $ko_{*}(BC_l)$ are not resolved.\\
The Gromov-Lawson-Rosenberg conjecture for cyclic groups then follows immediately from Proposition $2.1.4$, and this is proved in \cite{2} by making explicit calculations of eta invariants of lens spaces. We will give a very brief sketch of these calculations.
\begin{lem}
Let $n=4i\pm 1,\sigma=\rho_{-3}(\rho_{0}-\rho_3)^{2}$, and $\rho \in R_0(C_l)$. Then we have\\
a) If $n \geq 7$ then $\eta(L^{n}(l,(\overrightarrow{b_{2i-2}},3,3)))(\sigma \rho)=\eta(L^{n-4}(l,\overrightarrow{b_{2i-2}}))(\rho)$\\
b) If $n\geq 3$ then $\eta(L^{n}(l,(\overrightarrow{b_{2i-1}},1))-3L^{n}(l,(\overrightarrow{b_{2i-1}},3)))(\rho)=
\eta(L^{n}(l,(\overrightarrow{b_{2i-1}},3)))(\rho(\rho_1+\rho_{-1}-2\rho_0))$ and\\
$\eta(L^{n}(l,(\overrightarrow{b_{2i-1}},1))-5L^{n}(l,(\overrightarrow{b_{2i-1}},5)))(\rho)=
\eta(L^{n}(l,(\overrightarrow{b_{2i-1}},5)))(\rho(\rho_1+\rho_{-1}+\rho_2+\rho_{-2}-4\rho_0))$
\end{lem}
We recall again that $B^8$ is a Bott manifold, which is a simply connected manifold with $\hat{A}(B^8)=1$. By Lemma $2.1.3$ multiplying by $B$ does not change eta invariants. We now define, with $\sigma$ as in the Lemma:
$$Y^3=L^3(l;1,1)-3L^3(l;1,3)$$
$$Y^{8j+3}=Y^3 \times B^j$$
$$Z^3=L^3(l;1,1);Z^7=L^7(l;1,1,1,1)-3L^7(l;1,1,1,3)$$
$$Z^5=L^5(l;1,1)-3L^5(l;1,3);Z^9=L^9(l,1,1,1,1)-3L^9(l,1,1,1,3)$$
$$Z^{n}=Z^{n-8} \times B$$

$$\overrightarrow{\eta}(M)=(\eta(M)(\rho_1-\rho_0), \cdots ,\eta(M)(\rho_{l-1}-\rho_0))$$
$$\delta(M)=(\eta(M)(\sigma(\rho_1-\rho_0)), \cdots ,\eta(M)(\sigma(\rho_{l-1}-\rho_0)))$$
Note that if $q=l/2$ and $n \equiv 3 \mod 8$, then $\eta(M)(\rho_q-\rho_0)=\eta_q(M)$ lies in $\R/2\Z$. We now have the following Lemma from \cite{2}:
\begin{lem}
a)$\overrightarrow{\eta}(Y^{8j+3})=0$ and $\eta_q(Y^{8j+3})=\pm 1$\\
b)With $i \geq 1$, and $l=2^k$, we have that $\overrightarrow{\eta}(Z^{4i-1})$ has order at least $2^{k+1}$ in $(\R/\Z)^{l-1}$\\
c)With $i \geq 1$, and $l=2^k$, we have that $\overrightarrow{\eta}(Z^{4i+1})$ has order at least $2^{k-1}$ in $(\R/\Z)^{l-1}$.\\
d) If $n \geq 3$, then $\delta(Z^{n})=0$.\\
e) $\delta(L^5(l;3,3))$ has order at least $2^{k-1}$ in $(\R/\Z)^{l-1}$.
\end{lem}
It now follows that the lens spaces and lens space bundles defined do indeed span all of the kernel $Ker(Ap)$ in the case of cyclic groups; full details are in \cite{2}. We define $${\L}_n(BC_l):=span\{\overrightarrow{\eta}(L^n(l,\overrightarrow{a}))\} \subset (\R/\Z)^{l-1}$$
Thus ${\L}_n(BC_l)$ is defined simply as the span of the images of \emph{all} the lens spaces $L^n(l,\overrightarrow{a})$ under the map $\overrightarrow{\eta}$.\\
Then by Lemmas $2.1.5a),2.1.6$, we deduce $\delta$ induces a surjective map from ${\L}_n(BC_l)$ to ${\L}_{n-4}(BC_l)$, with kernels of order at least $l/2$ and $2l$ for $n \equiv 1,3 \mod 4$ respectively.\\
In the former case, Lemma $2.1.6$d),e) tell us that $|{\L}_5(BC_l)| \geq (l/2)^2$ and thus by induction $|{\L}_{4k+1}(BC_l)| \geq (l/2)^{k+1}|$, from which the second part of Proposition $2.1.4$ follows.\\

In the latter case, since $Z^3$ and $Z^7$ both have order at least $2l$ we can then use multiplication by the Bott manifold to deduce immediately by induction that $|{\L}_{8m+7}(BC_l)| \geq (2l)^{2m+2}$ as required, while $|{\L}_{8m+7}(BC_l)| \geq (2l)^{2m+1}$, and we get an extra factor of $2$ immediately from Lemma $2.1.6(a)$, and Proposition $2.1.4$ follows.\\

\section{Explicit calculations for $C_2$ and $C_4$}
To illustrate this method, it is useful to do some explicit calculations for groups of small order, starting with $C_2=\Z_2$, the cyclic group of order $2$.\\
Indeed, then we have $B\Z_2=\RP^{\infty}$ and $ko_*(\RP^{\infty})$ is well known, see \cite{3} or \cite{14} for example.\\
\begin{center}
\begin{tabular}{|c|c|}
\hline

$n$&$ko_n(\RP^{\infty})$\\
\hline
$8m+0$ & $\Z$\\
\hline
$8m+1$ & $2^2$\\
\hline
$8m+2$ & $2^2$\\
\hline
$8m+3$ & $[2^{4m+3}]$\\
\hline
$8m+4$ & $\Z$\\
\hline
$8m+5$ & $0$\\
\hline
$8m+6$ & $0$\\
\hline
$8m+7$ & $[2^{4m+4}]$\\
\hline
\end{tabular}
\end{center}
Here $[a]$ means cyclic of order $a$, and $2^N$ means the elementary abelian group of order $2^N$. Further $Ker(Ap)$ is trivial except in dimensions $3\mod 4$, where it is the entire group.\\
Stolz and Rosenberg \cite{14} proved the conjecture in this case by observing that the Adams Spectral sequence for $ko_*(\RP^{\infty})$ goes down all the way to the zero line in dimensions $4k+3$, so that all of $ko_{4k+3}(\RP^{\infty})$ must generated by the fundamental class $[\RP^{4k+3}]$, because $[\RP^{4k+3}]$ is detected in the ordinary $\Z_2$ homology of $\RP^{\infty}$. We can alternatively directly compute some eta invariants with respect to the unique irreducible non-trivial representation $\rho_1$(which is real), using Lemma $2.1.2$. Note that in dimensions $3 \mod 8$ the eta invariants calculated are in $\R/2\Z$, since all representations of the cyclic group of order two are real.
$$\eta(\RP^{8m+3})(1-\rho_1)=1/2(-2/2^{4m+2})=-2^{-4m-2} \in \R/2\Z$$
which has order $2^{4m+3}$ as required. Similarly in $8m+7$ we get:
$$\eta(\RP^{8m+7})(1-\rho_1)=1/2(2/2^{4m+4})=2^{-4m-4} \in \R/\Z$$
with order $2^{4m+4}$ as required. Thus all of the kernel $Ker(Ap)$ is spanned by the $ko-$ fundamental class of $\RP^{4k+3}$, which proves the conjecture.\\

We now move on to the cyclic group of order $4$. This involves a little more work, and we start by displaying $ko_*(BC_4)$, calculated using the local cohomology spectral sequence in \cite{3}. Note that $ko_1(BC_4)=[4] \oplus [2]$ by the Atiyah-Hirzebruch spectral sequence. This is a non-split extension on the $E_{\infty}$ page of the local cohomology spectral sequence, but in higher dimensions the extension is indeed split \cite{3}.\\
\begin{center}
\begin{tabular}{|c|c|c|}
\hline

$n$&$ko_n(BC_4)$ & $Ker(Ap)$\\
\hline
$8m+0$ & $\Z$ & $0$\\
\hline
$8m+1 \geq 9$ & $[2^{2m+1}] \oplus 2^2$ & $[2^{2m+1}]$\\
\hline
$8m+2$ & $2^2$ & $0$\\
\hline
$8m+3$ & $[2^{4m+3}] \oplus [2^{2m+1}]$ & $[2^{4m+3}] \oplus [2^{2m+1}]$\\
\hline
$8m+4$ & $\Z$ & $0$\\
\hline
$8m+5$ & $[2^{2m+2}]$ & $[2^{2m+2}]$\\
\hline
$8m+6$ & $0$ & $0$\\
\hline
$8m+7$ & $[2^{4m+5}] \oplus [2^{2m+1}]$ & $[2^{4m+5}] \oplus [2^{2m+1}]$\\
\hline
\end{tabular}
\end{center}

We will span all of the kernel $Ker(Ap)$ by $ko-$fundamental classes of lens spaces and lens space bundles.
\begin{lem}
The $ko-$fundamental classes of the lens spaces $L^{4k-1}(4;1, \cdots,1,1)$ and $L^{4k-1}(4;1, \cdots,1,3)$ span all of $ko_{4k-1}(BC_4)$. The fundamental class of the lens space bundle $L^{4k+1}(4;1, \cdots,1,1)$ spans all of $Ker(Ap) \subset ko_{4k+1}(BC_4)$.
\end{lem}
\begin{proof}
This is done by making some explicit eta invariant calculations using Lemma $2.1.2$, for the natural representation $\rho_1$ of $C_4=<i>$ which sends $i \mapsto i$, together with the real representation $\rho_2$ sending $i \mapsto -1$. This gives:
$$\eta(L^{4k-1}(4;1, \cdots,1,1))(1-\rho_1)=\frac{1}{4}(\frac{i^k(1-i)}{(1-i)^{2k}}+\frac{(-i)^k(1+i)}{(1+i)^{2k}}+\frac{2(-1)^k}{2^{2k}})$$
Since $(1 \pm i)^2=\pm 2i$ this simplifies:
$$=\frac{1}{2}(-1/2)^k+\frac{(-1)^k}{2^{2k+1}}$$
Similarly for $L^{4k-1}(4;1, \cdots,1,3)$, noting that $i^{-1}=-i$ Lemma $2.1.2$ gives:
$$\eta(L^{4k-1}(4;1, \cdots,1,3))(1-\rho_1)=\frac{1}{4}(\frac{i^{k+1}(1-i)}{(1-i)^{2k-1}(1+i)}+\frac{(-i)^{k+1}(1+i)}{(1+i)^{2k-1}(1-i)}+\frac{2(-1)^{k+1}}{2^{2k}})$$
$$=\frac{1}{4}(\frac{i^k(1-i)}{(1-i)^{2k}}+\frac{(-i)^k(1+i)}{(1+i)^{2k}}+\frac{2(-1)^{k+1}}{2^{2k}})$$
$$=\frac{1}{2}(-1/2)^k+\frac{(-1)^{k+1}}{2^{2k+1}}$$
Thus both of these have order $2^{2k+1} \in \R/\Z$. We now calculate with respect to the real representation $\rho_2$:
$$\eta(L^{4k-1}(4;1, \cdots,1,1))(1-\rho_2)=\frac{1}{4}(\frac{2(i)^k}{(1-i)^{2k}}+\frac{2(-i)^k}{(1+i)^{2k}})$$
$$=(-1/2)^k;$$
$$\eta(L^{4k-1}(4;1, \cdots,1,3))(1-\rho_2)=\frac{1}{4}(\frac{2(i)^{k+1}}{(1-i)^{2k-1}(1+i)}+\frac{2(-i)^{k+1}}{(1+i)^{2k-1}(1-i)})$$
$$=\frac{1}{4}(\frac{2(i)^k}{(1-i)^{2k}}+\frac{2(-i)^k}{(1+i)^{2k}})=(-1/2)^k;$$
Thus we must consider the subgroup of $(\R/\Z)^2$ spanned by the two vectors:
$$(\eta(L^{4k-1}(4;1, \cdots,1,1))(1-\rho_1),\eta(L^{4k-1}(4;1, \cdots,1,3))(1-\rho_1))$$
$$=(\frac{1}{2}(-1/2)^k+\frac{(-1)^{k}}{2^{2k+1}},\frac{1}{2}(-1/2)^k+\frac{(-1)^{k+1}}{2^{2k+1}});$$
$$(\eta(L^{4k-1}(4;1, \cdots,1,1))(1-\rho_2),\eta(L^{4k-1}(4;1, \cdots,1,3))(1-\rho_2))=((-1/2)^k,(-1/2)^k)$$
Thus we have two elements in $(\R/\Z)^2$ and we need to calculate the order of the subgroup they span. This is easily done by writing the two elements as a $2 \times 2$ matrix:
\begin{displaymath}
\left(\begin{array}{ccc}
\eta(L^{4k-1}(4;1, \cdots,1,1))(1-\rho_1) & \eta(L^{4k-1}(4;1, \cdots,1,3))(1-\rho_1) \\
\eta(L^{4k-1}(4;1, \cdots,1,1))(1-\rho_2) & \eta(L^{4k-1}(4;1, \cdots,1,3))(1-\rho_2) \\
\end{array} \right)
\end{displaymath}
\begin{displaymath}
=\left(\begin{array}{ccc}
\frac{1}{2}(-1/2)^k+\frac{(-1)^{k}}{2^{2k+1}} & \frac{1}{2}(-1/2)^k+\frac{(-1)^{k+1}}{2^{2k+1}} \\
(-1/2)^k & (-1/2)^k \\
\end{array} \right)
\end{displaymath}

The determinant of the above matrix is then $(-1/2)^k(\pm 2/2^{2k+1})=\pm 1/2^{3k}$, so that the order of the subspace spanned is thus at least $2^{3k}$.\\
So in $n=8m+7=4k-1$ we have spanned a subspace of order $2^{3(2m+2)}=2^{6m+6}$ which is the same as the order of $ko_{8m+7}(BC_4)$, thus proving the GLR conjecture in these dimensions, while in $n=8m+3=4k-1$ we have spanned a subspace of order at least $2^{3(2m+1)}=2^{6m+3}$, but since $\rho_2$ is real, we see that $\eta(L^{8m+3}(1-\rho_2))=(-1/2)^k \in \R/2\Z$ has order $2^{k+1}$ for both the lens spaces, which means that we have in fact spanned a subgroup of twice the order we previously calculated, giving us the required order of $2^{6m+4}$. It thus follows that $ko_{4k-1}(BC_4)$ is spanned by the $ko-$fundamental classes of the lens spaces  $L^{4k-1}(4;1, \cdots,1,1)$ and $L^{4k-1}(4;1, \cdots,1,3)$, and as generators we can choose the fundamental class of one, together with a suitable difference of lens spaces. Explicitly we get
$$ko_{4k-1}(BC_4)=<[L^{4k-1}(4;1, \cdots,1,1)]> \oplus <a[L^{4k-1}(4;1, \cdots,1,1)]-b[L^{4k-1}(4;1, \cdots,1,3)]>$$
with $a,b \in \Z$ odd, where $[M]$ here means $ko-$fundamental class (Compare with Lemmas $2.1.5(b),2.1.6(a)$).\\
Finally in dimensions $4k+1$ we need to span a subspace of order $2^{k+1}$. We claim that the $ko-$fundamental class of the lens space bundle $L^{4k+1}(4;1, \cdots,1,1)$ has the right order of $2^{k+1}$. Again we calculate the eta invariant, this time using Theorem $2.1.1$ and Lemma $2.1.2$:
$$\eta(L^{4k+1}(4;1, \cdots,1,1))(1-\rho_1)=\frac{1}{4}(\frac{i^k(1-i)(1+i)}{(1-i)^{2k+1}}+\frac{(-i)^k(1+i)(1-i)}{(1+i)^{2k+1}}+\frac{2(-1)^k(1-1)}{2^{2k}})$$
$$=\frac{1}{4}(\frac{i^k(1-i)}{(1-i)^{2k}}+\frac{(-i)^k(1+i)}{(1+i)^{2k}}+0)$$
$$=1/2(-1/2)^k$$
which has order $2^{k+1} \in \R/\Z$ as required.
\end{proof}

Note that by the following trivial algebraic lemma, the extensions in dimensions $3 \mod 4$ for $ko_*(BC_4)$ can be resolved by directly using the eta invariant calculations.
\begin{lem}
Let $G$ be an abelian p-group of rank 2, and $A \geq 1$ a natural number such that $A.G=0$. Then if $x\in G$ has order $A$, then $G \cong [A] \oplus [|G|/A]$, where $[N]$ again denoted the cyclic group of order $N$.
\end{lem}
Since we know that the eta invariant detects all of $ko_{4k-1}(BC_4)$, the number $A$ will simply be the largest order of any eta invariant $\eta(M)(\rho)$, with $M$ a lens space and $\rho$ a virtual representation of virtual dimension zero, and the above calculations show that $A=2^{2k+1}$, which determines the extensions.
\section{Quaternion groups}
We give here an outline of the proof in \cite{2} of the Gromov-Lawson-Rosenberg conjecture for the quaternion groups $Q_l$ of order $l=2^q$, generated by $\xi=e^{i\pi/2^{q-2}},j \in \qt$. This gives a natural representation $\tau: Q_l \rightarrow SU(2)$.\\
The calculations in \cite{2} can be summarised as follows:
\begin{prop}
i) We have that $|ko_{8m+3}(BQ_l)|=2^{4m+4}l^{2m+1}$ while $|ko_{8m+7}(BQ_l)|=2^{4m+4}l^{2m+2}$, and the kernel of the map $Ap$ is the entire group in these dimensions. \\
ii) Further, the kernel of the map $Ap$ is trivial in dimensions $n \equiv 2 \mod 4$, and is simply given by the $\eta, \eta^2$ multiples of $ko_{8m+3}(BQ_l)$ in $n \equiv 0,1 \mod 4$ respectively.
\end{prop}
Thus, in order to verify the GLR conjecture, it suffices to span all of $ko_{4m+3}(BQ_l)$ by $ko-$fundamental classes of positive scalar curvature spin manifolds, since multiplying by $\eta$ preserves positive scalar curvature.\\
We have three cyclic subgroups $H_t \subset Q_l$ of order $4$, with $t=1,2,3$ generated by $i,j,\xi j$ respectively, and thus by viewing $S^{4m-1}$ inside $\qt^m$ we can consider the manifolds $M_{t}^{4m-1}=S^{4m-1}/H_t$ together with $M_{Q}^{4m-1}=S^{4m-1}/Q_l$, and give them natural positive scalar curvature metrics induced by taking the quotient of the sphere by a discrete isometry group. We have $3$ non-trivial irreducible real representations $\kappa_t, t=1,2,3$ (see the character table for $Q_8$ in section 2.4) characterised by:
$$\kappa_1(j)=1,\kappa_1(\xi)=-1; $$
$$\kappa_2(j)=-1,\kappa_1(\xi)=1; $$
$$\kappa_3(j)=-1,\kappa_1(\xi)=-1; $$
Now with $\rho_0$ the trivial representation, we define virtual representations:\\

$$\epsilon_2 :=\begin{cases}
$$\rho_0-\kappa_1 \mbox{ if } q=3,$$\\
$$\kappa_3-\rho_0 \mbox{ if } q>3$$\\
\end{cases}$$

$$\epsilon_3 :=\begin{cases}
$$\epsilon_3 :=\rho_0-\kappa_3 \mbox{ if } q=3,$$\\
$$\kappa_1-\rho_0 \mbox{ if } q>3$$\\
\end{cases}$$

Note that the case $q=3$ is exceptional since this means $\xi=i$ so that the $\kappa_t$ are all non-trivial on $H_1$. We define tuples of eta invariants as follows, where the value of $r$ will be given later:
$$ \overrightarrow{\eta}(M)=(\eta(M)(1-\epsilon_2),\eta(M)(1-\epsilon_3),\eta(M)(2-\tau), \cdots , \eta(M)(2-\tau)^r)$$
Since the representations $1-\epsilon_2, 1- \epsilon_3$ as well as $(2-\tau)^{2a}$ are real and the representations $(2-\tau)^{2a+1}$ are quaternion, Theorem 2.0.3 tells us that if $n \equiv 3 \mod 8$ the eta invariants with respect to real representations lie in $\R/2\Z$ so that
$$\overrightarrow{\eta}(M^n) \in \R/2\Z \oplus \R/2\Z \oplus \R/\Z \oplus \R/2\Z \cdots$$
while in $n \equiv 7 \mod 8$, the eta invariants with respect to quaternion representations lie in $\R/2\Z$ so that
$$\overrightarrow{\eta}(M^n) \in \R/\Z \oplus \R/\Z \oplus \R/2\Z \oplus \R/\Z \cdots$$

We then have the following calculations \cite{2} by applying Theorem $2.0.3$, where $B^8$ is a Bott manifold, a simply connected spin manifold with $\hat{A}(B^8)=1$, and $K^4$ is the Kummer surface, a simply connected spin manifold with $\hat{A}(K^4)=2$, and $d=l-1$. Let $n=8m+3$ and $r=2m+1$. Then:\\

$\overrightarrow{\eta}(M_{1}^n-M_{2}^n)=(2^{-2m-1},0, \cdots ,0)$\\
$\overrightarrow{\eta}(M_{1}^n-M_{3}^n)=(*,2^{-2m-1},0, \cdots ,0)$\\
$\overrightarrow{\eta}(M_{Q}^3 \times (B^8)^m)=(*,*,d/l,0, \cdots ,0)$\\
$\overrightarrow{\eta}(M_{Q}^7 \times K^4 \times (B^8)^{m-1})=(*,*,*,2d/l,0, \cdots ,0)$\\
$\overrightarrow{\eta}(M_{Q}^{11} \times (B^8)^{m-1})=(*,*,*,*,d/l,0, \cdots ,0)$\\
$\overrightarrow{\eta}(M_{Q}^{15} \times K^4 \times (B^8)^{m-2})=(*,*,*,*,*,2d/l,0, \cdots ,0)$\\
$ \cdots \cdots $\\
$\overrightarrow{\eta}(M_{Q}^{8m+3})=(*,*,*, \cdots ,0,0,d/l);$\\

Here * is a term we're not interested in, and since the above matrix is lower triangular, taking the products of the orders of the diagonal entries in $\R/\Z, \R/2\Z$ respectively tells us that the $ko-$ fundamental classes of our collection of lens spaces span a subgroup of order at least $2^{4m+4}l^{2m+1}=|ko_{8m+3}(BQ_l)|$ as required.\\
Analogously, in $n=8m+7$, we let $r=2m+2$, and then we have the following calculations:\\

$\overrightarrow{\eta}(M_{1}^n-M_{2}^n)=(2^{-2m-2},0, \cdots ,0)$\\
$\overrightarrow{\eta}(M_{1}^n-M_{3}^n)=(*,2^{-2m-2},0, \cdots ,0)$\\
$\overrightarrow{\eta}(M_{Q}^3 \times K^4 \times (B^8)^m)=(*,*,2d/l,0, \cdots ,0)$\\
$\overrightarrow{\eta}(M_{Q}^7 \times (B^8)^{m-1})=(*,*,*,d/l,0, \cdots ,0)$\\
$\overrightarrow{\eta}(M_{Q}^{11} \times K^4 \times (B^8)^{m-1})=(*,*,*,*,2d/l,0, \cdots ,0)$\\
$\overrightarrow{\eta}(M_{Q}^{15} \times (B^8)^{m-2})=(*,*,*,*,*,d/l,0, \cdots ,0)$\\
$ \cdots \cdots $\\
$\overrightarrow{\eta}(M_{Q}^{8m+7})=(*,*,*, \cdots ,0,0,d/l);$\\

Hence we can deduce that the $ko-$ fundamental classes of our collection of lens spaces span a subgroup of order at least $2^{4m+4}l^{2m+2}=|ko_{8m+7}(BQ_l)|$ as required.

\section{$Q_8$ in more detail}
While the above calculations prove the conjecture for quaternion groups, they do not give us the exact structure of the groups $ko_n(BQ_l)$. It greatly simplifies subsequent calculations for the semi-dihedral group $SD_{16}$ in chapter $5$ to know $ko_n(BQ_8)$ explicitly in terms of generators.\\

This is done in \cite{3} using the local cohomology spectral sequence, and in the thesis of Bayen \cite{bay} by the Adams spectral sequence. The latter calculations are more enlightening from the geometric point of view.\\
There is a 2-local stable decomposition of $BQ_8$, and higher groups, into indecomposable summands as follows, see \cite{bay} and \cite{mp}:
$$BQ_l=BSL_2(q) \vee 2\Sigma^{-1}BS^3/BN$$
where $N$ is the normalizer of a maximal torus in $S^3$, and $l$ is the largest power of $2$ dividing $q^2-1$, for $q$ an odd prime power. Note that $q=3$ for $l=8$. The method then is to make the calculation for each summand. It then turns out that for $K=3$ and $K=7$ we get the following $2-$local results:
$$ko_{8m+K}(\Sigma^{-1}BS^3/BN)=\Z_{2^{2m+2}}$$
$$ko_{8m+3}(BSL_2(3))=\Z_{2^{4m+3}}\oplus \Z_{2^{2m}}=<x_{8m+3}> \oplus <\beta x_{8m-5}-16x_{8m+3}>$$
$$ko_{8m+7}(BSL_2(3))=\Z_{2^{4m+6}}\oplus \Z_{2^{2m}}=<x_{8m+7}> \oplus <\beta x_{8m-1}-16x_{8m+7}>$$
$$ko_*(BQ_8)=ko_*(BSL_2(3))\oplus 2 ko_*(\Sigma^{-1}BS^3/BN)$$.\\
We shall now verify this geometrically, by describing the generators as images of fundamental classes of manifolds. For reference, here is the character table of $Q_8$ \\

\begin{center}
\begin{tabular}{|c|c|c|c|c|c|}
\hline
$ $&1 & 1 & $2$ & $2$ & $2$ \\
$\rho$&$1$ & $-1$ & $[i]$ & $[j]$ & $[k]$\\
\hline
$1=\rho_0$ & $1$ & $1$ & $1$ & $1$ & $1$\\
$\kappa_1$ & $1$ & $1$ & $-1$ & $1$ & $-1$\\
$\kappa_2$ & $1$ & $1$ & $1$ & $-1$ & $-1$\\
$\kappa_3$ & $1$ & $1$ & $-1$ & $-1$ & $1$\\
$\tau$ & $2$ & $-2$&$0$ &$0$& $0$\\
\hline
\end{tabular}
\end{center}

Firstly, $2 ko_*(\Sigma^{-1}BS^3/BN)$ term is independent of $l$, and comparing with the eta invariant calculations given in the previous section, where we see that $\eta(M_{1}^n-M_{2}^n)(2- \tau)^a=0$ for every $a$, and similarly for $M_{1}^n-M_{3}^n$, it follows that the $ko$-fundamental classes of the manifolds $M_{1}^n-M_{2}^n, M_{1}^n-M_{3}^n$ generate all of $2 ko_n(\Sigma^{-1}BS^3/BN)$, with $ n \equiv 3\mod 4$, 2-locally.\\

It thus remains to realize the $BSL_2(3)$ part, and to do this we will consider the manifolds $M_Q^{4k+3}=S^{4k+3}/Q_l$.
\begin{prop}
Localised at $2$, we have that $ko_{4k+3}(BSL_2(3))$, which is the same as the quotient $ko_{4k+3}(BQ_l)/2 ko_{4k+3}(\Sigma^{-1}BS^3/BN)$, can be realized by the images of the fundamental classes of the manifolds $M_Q^{4k+3}$ and $B^8 \times M_Q^{4k-5}$.
\end{prop}
\begin{proof}
We do this by explicit eta calculations, using only representations $(2-\tau)^a$, which suffices since we have already observed $\eta(M_{1}^n-M_{2}^n)(2- \tau)^a=0=\eta(M_{1}^n-M_{3}^n)$ for every $a$.\\
We start by calculating the order of $[M_Q^{4k+3}]$. Note that $M_Q^{4k+3}=S^{4k+3}/(k+1)\tau$, where $(k+1)\tau$ is the $(k+1)$-fold sum $\tau \oplus \cdots \oplus \tau$. Now $\tau$ is a unitary representation which may be characterized by \\
\begin{displaymath}
\tau(i)=\left(\begin{array}{ccc}
i & 0 \\
0 & -i \\
\end{array} \right);
\tau(j)=\left(\begin{array}{ccc}
0 & \omega^3 \\
\omega & 0 \\
\end{array} \right)
\end{displaymath}
where $\omega=(1+i)/\sqrt{2}$. Thus for $x \neq 1$ we get $det(1-\tau)(x)=4$ if $x=-1$, and $2$ else. So we can apply Theorem $2.0.3$ to see:

$$\eta(M_Q^{4k+3})(2- \tau)=\frac{1}{8}\sum_{1 \neq g \in Q_8} \frac{Trace((2-\tau)(g))det((k+1)\tau(g))^{1/2}} {(det(I-(k+1)\tau(g)))}$$
$$=1/8(4/4^{k+1}+6(2/2^{k+1}))=1/2^{2k+3}+3/2^{k+2}$$
Thus in dimensions $8m+3$ the fundamental class has of order at least $2^{4m+3}$, and of course the calculations in \cite{2} imply that it has exactly this order. The representation $\tau$ is quaternion, so in dimensions $8m+7$ the order is $2^{4m+6}$, since the eta invariant extends to $\R/2\Z$.\\
We wish to see how large a subspace is spanned by $M_Q^{4k+3}$ and $B \times M_Q^{4k-5}$, where $B=B^8$ is a Bott manifold. Multiplying by the Bott element is a monomorphism in dimensions $ 3 \mod 4$, so $B \times M_Q^{4k-5}$ has the same order as $M_Q^{4k-5}$. Then, just as above, we can make the following calculations, remembering $(2-\tau)^a$ is real for $a$ even, and quaternion for $a$ odd:
$$\eta(M_Q^{4k+3})(2- \tau)^2=\frac{1}{8}\sum_{1 \neq g \in Q_8} \frac{Trace((2-\tau)(g))^2det((k+1)\tau(g))^{1/2}} {(det(I-(k+1)\tau(g)))}$$
$$=1/8(16/4^{k+1}+6(4/2^{k+1}))=2/4^{k+1}+3/2^{k+1} \in \R/2\Z$$

Thus using these we can deduce:
$$(\eta(M_Q^{8m+3})(2- \tau),\eta(M_Q^{8m-5} \times B^8)(2- \tau))=(1/2^{4m+3}+3/2^{2m+2},1/2^{4m-1}+3/2^{2m})$$
while for $(2-\tau)^2$, we know we have a real representation so that $\eta(M_Q^{8m+3}(2-\tau)^2 \in \R/2\Z$, so that we can divide by $2$ to get:
$$(\eta(M_Q^{8m+3})(2- \tau)^2,\eta(M_Q^{8m-5} \times B^8)(2- \tau)^2)=(1/2^{4m+2}+3/2^{2m+2},1/2^{4m-2}+3/2^{2m})$$
So, the order of the subgroup spanned may be bounded below the order $\in \R/\Z$ of the determinant of the following matrix:\\
\begin{displaymath}
X=\left(\begin{array}{ccc}
\eta(M_Q^{4k+3})(2- \tau) & \eta(M_Q^{4k-5} \times B^8)(2- \tau) \\
\eta(M_Q^{4k+3})(2- \tau)^2 & \eta(M_Q^{4k-5} \times B^8)(2- \tau)^2 \\
\end{array} \right)
\end{displaymath}
 From the above calculations, the entries can be read off immediately, so that if $4k+3=8m+3$ we get:\\
\begin{displaymath}
X=\left(\begin{array}{ccc}
1/2^{4m+3}+3/2^{2m+2} & 1/2^{4m-1}+3/2^{2m} \\
1/2^{4m+2}+3/2^{2m+2} & 1/2^{4m-2}+3/2^{2m}\\
\end{array} \right)
\end{displaymath}
Then the determinant of $X$ is $$1/2^{8m+1}+3/2^{6m+3}+3/2^{6m}+9/2^{4m+2}-1/2^{8m+1}-3/2^{6m+2}-3/2^{6m+1}-9/2^{4m+2}$$
$$=3/2^{6m+3} + \cdots$$
The dots mean terms which have strictly lower order in $\R/\Z$, so this has order $2^{6m+3} \in \R/\Z$, which is the same as the 2-local order of $ko_{8m+3}(BSL_2(3))$, just as required.\\
Analogously in dimensions $8m+7=4k+3$, this time by dividing by $2$ for the quaternion representation $2-\tau$ we have:
$$(\eta(M_Q^{8m+7})(2- \tau),\eta(M_Q^{8m-1} \times B^8)(2- \tau))=(1/2^{4m+6}+3/2^{2m+4},1/2^{4m+2}+3/2^{2m+2})$$
while for $(2-\tau)^2$ we get:
$$(\eta(M_Q^{8m+7})(2- \tau)^2,\eta(M_Q^{8m-1} \times B^8)(2- \tau)^2)=(1/2^{4m+3}+3/2^{2m+2},1/2^{4m-1}+3/2^{2m})$$
The analogous determinant then has order $2^{6m+6} \in \R/\Z$ , which is again the same as the 2-local order of $ko_{8m+7}(BSL_2(3))$ as required.
\end{proof}

Thus $2-$locally, all of $ko_{4k+3}(BQ_8)/2 ko_{4k+3}(\Sigma^{-1}BS^3/BN)$ is realized by the manifolds $M_Q^{4k+3}$ and $M_Q^{4k-5} \times B^8$, and may be detected by the eta invariant, using only the virtual representations $2-\tau$ and $(2-\tau)^2$. The extension can again be directly determined from Lemma $2.2.2$.

\chapter{Elementary abelian groups}
In this chapter we prove the first two of our main Theorems. We start by summarizing the known calculations for $ko_*(BV(n))$. These may be found in \cite{3} and \cite{cyy}, using the local cohomology and Bockstein, and Adams spectral sequences respectively.\\
We then proceed to construct some spin manifolds. It turns out that suitably chosen iterated real projective bundles over real projective spaces will suffice. The periodic part of the kernel $Ker(Ap)$ is detected in periodic $K-$theory using the eta invariant, and this is all spanned purely by inclusion from cyclic subgroups. We shall prove this for \emph{arbitrary} rank.\\
The remaining Bott torsion classes in the kernel are detected explicitly in the ordinary $\Z_2$ homology, by constructing manifolds and calculating the images of their fundamental classes.\\
Cohomology and homology will always be with $\Z_2$ coefficients in this and subsequent chapters, where $\Z_2$ is the field with two elements. We recall that $H^*(BV(n))= \Z_2[x_1, \cdots,x_n]$, and we denote by $\xi_{(a_1, \cdots, a_n)}$ the element in the homology $H_*(BV(n))$ dual to $x_1^{a_1} \cdots x_n^{a_n}$ with respect to the usual monomial basis.\\
Further, we know that $H^*(\RP^n)=\Z_2[x]/x^{n+1}$, and by \cite{micc}, the total Stiefel-Whitney class is given by $w(\RP^n)=(1+x)^{n+1}$. Thus the real projective spaces $\RP^n$ are spin exactly when $n \equiv 3 \mod 4$ or $n=1$.
\section{$ko_*(BV(n))$ calculations}
For arbitrary rank, these are given in \cite{3}, chapter 10, using the Bockstein spectral sequence, and in \cite{cyy}. We first summarize the results.
\begin{thm}
The $ko_*$ homology of elementary abelian groups $ko_*(BV(n))$ of rank $n$ may be tabulated as follows:\\
\begin{center}
\begin{tabular}{|c|c|}
\hline
$N$&$ko_N(BV(n))$\\

\hline

$8m+0 $ & $\Z \oplus 2^{h_{8k}}$\\

$8m+1$ & $2 \oplus 2^{{n \choose 1} + \cdots + {n \choose 4m+1}} \oplus 2^{h_{8m+1}}$\\

$8m+2$ & $2 \oplus 2^{{n \choose 1} + \cdots + {n \choose 4m+1}} \oplus 2^{h_{8m+2}}$\\

$8m+3$ & $[2^{4m+3}]^{n \choose {n-1}} \oplus [2^{4m+2}]^{n \choose {n-2}} \oplus \cdots \oplus [2^{4m-n+5}]^{n \choose {1}} \oplus [2^{4m-n+4}]^{n \choose {0}} \oplus 2^{h_{8m+3}}\oplus 2^{h_{8m+3}-{n \choose 4m+3}}$\\

$8m+4$ & $\Z \oplus 2^{h_{8m+4}}$\\

$8m+5$ & $2^{h_{8m+5}}$\\

$8m+6$ & $2^{h_{8m+6}}$\\
$8m+7$ & $[2^{4m+4}]^{n \choose {n-1}} \oplus [2^{4m+3}]^{n \choose {n-2}} \oplus \cdots \oplus [2^{4m-n+6}]^{n \choose {1}} \oplus [2^{4m-n+5}]^{n \choose {0}} \oplus 2^{h_{8m+7}}$\\
\hline
\end{tabular}
\end{center}
\end{thm}
Here $h_i=dim_{\Z_2}(Start(2)(TU/Sq^2)^{\vee})_{i}$, where $TU$ is the complex Bott torsion part of the complex connective cohomology  (\cite{3}, chapter 10), $Start(r)$ means we take suspensions so that the first non-zero entry is in degree $r$, and $[2^a]$ means a cyclic group of order $2^a$, and an unbracketed $2^a$ is an elementary abelian group of rank $a$.\\
From \cite{3}, it follows that the $h_i$ part is exactly the Bott torsion part we need to detect in ordinary homology, and the cyclic summands in dimensions $3 \mod 4$ should be detected in periodic K-theory by the eta invariant.\\
In low ranks, these calculations have also been carried out using the local cohomology spectral sequence (see \cite{3}, chapter 12), and indeed, this is summarized for $V(2)$ in the first chapter. For completeness, and use in later sections, we also display the $E_{\infty}$ page for $V(3)$ here.

\setlength{\unitlength}{1cm}
\begin{picture}(20,19)
\multiput(0.2,4)(0,0.5){28}%
{\line(1,0){12.8}}
\put(0.2,4){\line(0,1){13.5}}
\put(4,4){\line(0,1){13.5}}
\put(6,4){\line(0,1){13.5}}
\put(10.3,4){\line(0,1){13.5}}
\put(13,4){\line(0,1){13.5}}
\multiput(11.6,6.6)(0,0.5){3}%
{$0$}
\multiput(11.6,10.6)(0,0.5){3}%
{$0$}
\multiput(11.6,14.6)(0,0.5){3}%
{$0$}

\multiput(11.6,5.6)(0,4){3}%
{$0$}
\put(11.6,4.6){$2^4$}
\put(11.6,5.1){$2^7$}
\multiput(11.6,8.6)(0,0.5){2}%
{$2^{8}$}
\multiput(11.6,12.6)(0,0.5){2}%
{$2^{8}$}
\multiput(11.6,16.6)(0,0.5){2}%
{$2^{8}$}
\multiput(11.6,4.1)(0,2){7}%
{$\mathbb{Z}$}

\multiput(8,9.6)(0,1){3}%
{$0$}
\multiput(8,13.6)(0,1){3}%
{$0$}
\put(6.7,6.1){$[8]^3\oplus [4]^3\oplus[2]$}
\put(6.7,8.1){$[16]^3\oplus [8]^3\oplus[4]$}
\put(6.7,10.1){$[2^7]^3\oplus [2^6]^3\oplus[2^5]$}
\put(6.7,12.1){$[2^8]^3\oplus [2^7]^3\oplus[2^6]$}
\put(6.7,14.1){$[2^{11}]^3\oplus[2^{10}]^3\oplus[2^{9}]$}
\put(6.7,16.1){$[2^{12}]^3\oplus[2^{11}]^3\oplus[2^{10}]$}
\put(8,4.1){$0$}
\put(8,4.6){$0$}
\put(8,5.1){$0$}
\put(8,5.6){$0$}
\put(8,6.6){$0$}
\multiput(5,4.1)(0,0.5){27}%
{$0$}
\multiput(8,7.1)(0,2){6}%
{$0$}
\multiput(8,7.6)(0,2){5}%
{$0$}
\multiput(8,8.6)(0,4){3}%
{$0$}
\put(8,17.6){$\vdots$}
\put(5,17.6){$\vdots$}
\put(2,17.6){$\vdots$}
\put(11.8,17.6){$\vdots$}
\put(2,4.1){$0$}
\put(2,4.6){$0$}
\put(2,5.1){$0$}
\put(2,5.6){$0$}
\put(2,6.1){$0$}
\put(2,6.6){$0$}
\put(2,7.1){$0$}
\put(2,7.6){$2^8$}
\put(2,8.6){$2^{6}$}
\put(2,9.6){$2^{15}$}
\put(2,10.6){$2^{14}$}
\put(2,11.6){$2^{24}$}
\put(2,12.6){$2^{24}$}
\put(2,13.6){$2^{35}$}
\put(2,14.6){$2^{36}$}
\put(2,15.6){$2^{48}$}
\put(2,16.6){$2^{50}$}
\put(2,8.1){$2^3$}
\put(2,10.1){$2^7$}
\put(2,12.1){$2^{13}$}
\put(2,14.1){$2^{21}$}
\put(2,16.1){$2^{31}$}
\put(2,9.1){$2^3$}
\put(2,11.1){$2^8$}
\put(2,13.1){$2^{15}$}
\put(2,15.1){$2^{24}$}
\put(2,17.1){$2^{35}$}

\put(13.3,4.1){0}
\put(13.3,4.6){1}
\put(13.3,5.1){2}
\put(13.3,5.6){3}
\put(13.3,6.1){4}
\put(13.3,6.6){5}
\put(13.3,7.1){6}
\put(13.3,7.6){7}
\put(13.3,8.1){8}
\put(13.3,8.6){9}
\put(13.2,9.1){10}
\put(13.2,9.6){11}
\put(13.2,10.1){12}
\put(13.2,10.6){13}
\put(13.2,11.1){14}
\put(13.2,11.6){15}
\put(13.2,12.1){16}
\put(13.2,12.6){17}
\put(13.2,13.1){18}
\put(13.2,13.6){19}
\put(13.2,14.1){20}
\put(13.2,14.6){21}
\put(13.2,15.1){22}
\put(13.2,15.6){23}
\put(13.2,16.1){24}
\put(13.2,16.6){25}
\put(13.2,17.1){26}

\put(13,18.2){degree(t)}

\put(6.7,3.2){$H^{1}_{J}(\overline{QO})$}
\put(10.5,3.2){$H^{0}_{J}(\tau) \oplus H^{0}_{I}(\overline{QO})$}
\put(3.8,3.2){$H^{2}_{J}(TO)$}
\put(1.5,3.2){$H^{3}_{J}(TO)$}

\put(0.2,1.7){where$[n]:=$ cyclic group of order $n$, $2^{r}$:= elementary abelian group of order $r$.}
\put(2,0.7){The $E_{\infty}$-page for $ko_{*}(BV(3))$}
\end{picture}
It is easy to now see the orders $h_i$ of the Bott torsion classes we must realize. Explicitly, for $k \geq 1$ we have (see \cite{3}, section 12.4.G):\\
$h_{4k+1}=k^2+k+1$; $h_{4k+3}=k^2+2k;$\\
$h_{4k+2}=k^2+5k$; $h_{4k}=k^2+4k+3$.

\section{The periodic part}
In this section we realize all of the one-column of the $E_{\infty}$ page of the local cohomology spectral sequence for $ko_*(BV(n))$. To give a clearer idea of the general method, we give explicit calculations for $V(2),V(3)$.\\
We start by calculating the determinant of the character table of $V(n)$, something that proves useful later.
\begin{lem}
The determinant $D$ of the character table of $V(n)$ equals $\pm (2^n)^{2^{n-1}}$
\end{lem}
\begin{proof}
Multiplying any character matrix of an abelian group by its complex conjugate matrix gives us a diagonal matrix, with the diagonal entries equalling precisely the order of the group.\\
Since $V(n)$ is abelian of order $2^n$ with only real representations, and real entries (namely $\pm 1$) in its character table, we immediately see that $D^2=(2^n)^{2^n}$, so that the conclusion is immediate upon taking square roots.
\end{proof}

Let $\rho_1$ denote the non-trivial representation of $\Z_2$. Then from section $2.2$ of the previous chapter we have  $$\eta(\RP^{4j+3})(\rho_0-\rho_1)=1/2(2/2^{2j+2})=2^{-2j-2}$$
so that the $ko-$fundamental class $[\RP^{8k+7}]$ has order at least $2^{4k+4}$, while $[\RP^{8k+3}]$ has order at least $2^{4k+3}$, because since $\rho_1$ is real we get $\eta(\RP^{8k+3})(\rho_0-\rho_1)=2^{-4k-2} \in \R/2\Z$.\\
Let $V(2)=\{e,x,y,z\}$. We now denote by $[\RP^{n}_{x}]$ (respectively $y,z$), the image of $[\RP^{n}]$ via the map $\RP^n \rightarrow B<x> \hookrightarrow BV(2)$, where the first map is the classifying map for the universal cover of $\RP^n$, and the second is induced by the inclusion $x \hookrightarrow V(2)$. Further, we denote by $\hat{x}$ the irreducible representation of $V(2)$ with kernel $<x>$. From the previous section and chapter 1, we have that $ko_{8k+3}(BV(2))=[2^{4k+3}]^2 \oplus [2^{4k+2}]$ while $ko_{8k+7}(BV(2))=[2^{4k+4}]^2 \oplus [2^{4k+3}]$.\\
\begin{prop}
If $n=3,7 \mod 8$, then the $ko-$fundamental classes $\{[\RP^{n}_{x}],[\RP^{n}_{y}],[\RP^{n}_{z}]\}$ span all of $ko_n(BV(2))$.
\end{prop}
\begin{proof}
Let $n=8k+3$. Since all the representations of $V(r)$ are real, the eta invariant will take values in $\R/2\Z$. By restricting representations we deduce immediately that the triple $$(\eta(\RP^{n}_{x})(\hat{x}-\rho_0),\eta(\RP^{n}_{y})(\hat{x}-\rho_0),\eta(\RP^{n}_{z})(\hat{x}-\rho_0))=(0,2^{-4k-2},2^{-4k-2}) \in (\R/2\Z)^3$$
has order $2^{4k+3} \in \R/2\Z$. By symmetry, we then get $$(\eta(\RP^{n}_{x})(\hat{y}-\rho_0),\eta(\RP^{n}_{y})(\hat{y}-\rho_0),\eta(\RP^{n}_{z})(\hat{y}-\rho_0))=(2^{-4k-2},0,2^{-4k-2}) \in (\R/2\Z)^3$$ $$(\eta(\RP^{n}_{x})(\hat{z}-\rho_0),\eta(\RP^{n}_{y})(\hat{z}-\rho_0),\eta(\RP^{n}_{z})(\hat{z}-\rho_0))=(2^{-4k-2},2^{-4k-2},0) \in (\R/2\Z)^3$$
These triples are pairwise independent, and adding all three gives us $(2^{-4k-1},2^{-4k-1},2^{-4k-1})$ which is of order $2^{4k+2}$ in $(\R/2\Z)^3$, so in all we've spanned a subgroup of order $(2^{4k+3})^{2} 2^{4k+2}$, just as required. In a bit more detail, we have three elements of order $2^{4k+3}$, and the only relation is that their sum has order $2^{4k+2}$, so that the order of the subgroup they span must be $(2^{4k+3})^{2} 2^{4k+2}$.\\
The argument in dimensions $n=8k+7$ is exactly the same, giving us three triples
$$(\eta(\RP^{n}_{x})(\hat{x}-\rho_0),\eta(\RP^{n}_{y})(\hat{x}-\rho_0),\eta(\RP^{n}_{z})(\hat{x}-\rho_0))=(0,2^{-4k-4},2^{-4k-4}) \in (\R/\Z)^3$$
$$(\eta(\RP^{n}_{x})(\hat{y}-\rho_0),\eta(\RP^{n}_{y})(\hat{y}-\rho_0),\eta(\RP^{n}_{z})(\hat{y}-\rho_0))=(2^{-4k-4},0,2^{-4k-4}) \in (\R/\Z)^3$$ $$(\eta(\RP^{n}_{x})(\hat{z}-\rho_0),\eta(\RP^{n}_{y})(\hat{z}-\rho_0),\eta(\RP^{n}_{z})(\hat{z}-\rho_0))=(2^{-4k-4},2^{-4k-4},0) \in (\R/\Z)^3$$
which span a subgroup of order $(2^{4k+4})^{2} 2^{4k+3}$ as required.\\
Alternatively, to illustrate the method for higher ranks, we calculate the determinants of the resulting $3 \times 3$ matrices. For simplicity, we henceforth denote $2^{-4k-2} \in \R/2\Z$ by $2^{-4k-3}$. Looking at the above calculation, the matrix we get in dimensions $8k+3$ is \\
\begin{displaymath}
X=\left(\begin{array}{ccc}
0 & 2^{-4k-3} & 2^{-4k-3} \\
2^{-4k-3} & 0 & 2^{-4k-3} \\
2^{-4k-3} & 2^{-4k-3} & 0 \\
\end{array} \right)
\end{displaymath}
Now the character table table of $V(2)$ is :\\
\begin{center}
\begin{tabular}{|c|c|c|c|c|}
\hline
$\rho$& $e$ & $x$ & $y$ & z\\
\hline
$1=\rho_0$ & $1$ & $1$ & $1$ & $1$\\
$\hat{x}$ & 1 & $1$ & $-1$ & $-1$\\
$\hat{y}$ & $1$ & $-1$ & $1$ & $-1$\\
$\hat{z}$ & $1$ & $-1$ & $-1$ & $1$\\
\hline
\end{tabular}
\end{center}
After subtracting off the first row, it is immediate that the determinant of the character table is the same as the determinant of the matrix
\begin{displaymath}
Y=\left(\begin{array}{ccc}
0 & 2 & 2 \\
2 & 0 & 2 \\
2 & 2 & 0 \\
\end{array} \right)
\end{displaymath}

which has determinant $-16$. So we can deduce immediately that
$$|Det(X)|=(2^{-4k-4})^3 |Det(Y)|=2^{-12k-8}$$
meaning a subgroup of order at least $2^{12k+8}=(2^{4k+3})^2 2^{4k+2}$ is spanned, as required. The calculation goes through completely analogously in $8k+7$, giving a group of order at least $(2^{4k+4})^2 2^{4k+3}$ as required.

\end{proof}
We proceed in a similar manner for the rank three case. Let $V(3)=\{e,x,y,z,xy,xz,yz,xyz\}$ and just as before $[\RP^{n}_{x}]$, the image of $[\RP^{n}]$ via the map $\RP^n \rightarrow B<x> \hookrightarrow BV(3)$ (and similarly for $y$,$z$ etc.). For a non-trivial irreducible representation $\rho$ of $V(3)$ we can define $7-$tuples of eta invariants:
$$\overrightarrow{\eta}(\rho-1)=(\eta(\RP^{n}_{x})(\rho-1), \cdots,\eta(\RP^{n}_{xyz})(\rho-1))$$
where $1$ is simply the trivial irreducible representation. We can now prove:
\begin{prop}
If $n=8m+3, 8m+7$, then the $ko-$fundamental classes $\{[\RP^{n}_{x}],\cdots,[\RP^{n}_{xyz}]\}$ span  subgroups of order
$$(2^{4m+3})^{3} (2^{4m+2})^{3}2^{4m+1},$$
$$(2^{4m+4})^{3} (2^{4m+3})^{3}2^{4m+2}$$
 respectively, inside $ko_n(BV(3))$.
\end{prop}
\begin{proof}
Each $\overrightarrow{\eta}(\rho-1)$ is immediately determined by restriction to cyclic subgroups, and so in $n=8m+3$ we get the following $7 \times 7$ matrix of row vectors:\\

\begin{displaymath}
\left(\begin{array}{ccccccc}
2^{-4m-3} & 0 & 0 & 2^{-4m-3} & 2^{-4m-3} & 0 & 2^{-4m-3} \\
0 &2^{-4m-3} & 0 & 2^{-4m-3} & 0 & 2^{-4m-3} & 2^{-4m-3}\\
0 & 0 & 2^{-4m-3} & 0 & 2^{-4m-3} & 2^{-4m-3} & 2^{-4m-3}\\
2^{-4m-3} & 2^{-4m-3} & 0 & 0 & 2^{-4m-3} & 2^{-4m-3} & 0 \\
2^{-4m-3} & 0 & 2^{-4m-3} & 2^{-4m-3} & 0 & 2^{-4m-3} & 0 \\
0 & 2^{-4m-3} & 2^{-4m-3} & 2^{-4m-3} & 2^{-4m-3} & 0 & 0 \\
2^{-4m-3} & 2^{-4m-3} & 2^{-4m-3} & 0 & 0 & 0 & 2^{-4m-3} \\

\end{array} \right)
\end{displaymath}
Note that to obtain this matrix, we can proceed like we did in Proposition $3.2.2$. Thus all we do is take the character table of $V(3)$, subtract off the first row (that is the trivial representation), from all the other rows, and then multiply each of the other seven rows by $2^{-4m-4}$. The determinant of the above matrix is thus $(2^{-4m-4})^7$ times the determinant of the character table of $V(3)$, and thus equals, by Lemma $3.2.1$:
$$ (2^3)^{2^2}.2^{-28m-28}=2^{-28m-16}$$
$$=2^{-(3(4m+3)+3(4m+2)+4m+1)}$$
just as required. Analogously in $n=8m+7$ a we span subspace of order at least $2^{28m+20}$ which is the same as $2^{(3(4m+4)+3(4m+3)+4m+2)}$ as required.
\end{proof}
We now consider the case of arbitrary rank. The argument is similar, but combinatorially much more intricate. Let $V(n)=<x_1, \cdots, x_n>$, and let $[\RP^k_{x}]$ etc. be defined as before. Then it is known \cite{3}, Theorem $10.2.1$, using the Bockstein spectral sequence that
$$ko_{8m+3}(BV(n))= [2^{4m+3}]^{n \choose {n-1}} \oplus [2^{4m+2}]^{n \choose {n-2}} \oplus \cdots \oplus [2^{4m-n+5}]^{n \choose {1}} \oplus [2^{4m-n+4}]^{n \choose {0}} \oplus 2^{h_{8m+3}-{n \choose {4m+3}}}$$
where the $2^{h_{8m+3}}$ is $2-$torsion that is detected in the $n-$th filtration of the local cohomology spectral sequence, and vanishes upon inverting the Bott element. Similarly,
$$ko_{8m+7}(BV(n))= [2^{4m+4}]^{n \choose {n-1}} \oplus [2^{4m+3}]^{n \choose {n-2}} \oplus \cdots \oplus [2^{4m-n+6}]^{n \choose {1}} \oplus [2^{4m-n+5}]^{n \choose {0}} \oplus 2^{h_{8m+7}}$$
We will show that the $[\RP^k_{x}]$ span all of these groups, except for the $2^{h_{k}}$ term. Since the $2^{h_{k}}$ term vanishes upon inverting $\beta$, it suffices to check that the $ko-$fundamental classes $[\RP^k_{x}]$ span a subgroup of the right order in periodic K-theory, which we now do by using eta invariant calculations.
\begin{thm}
In the above situation, with $N=8m+3,8m+7$, the set of all $[\RP^N_{x}]$ induced by inclusion from cyclic $\Z_2$ subgroups span subgroups of order
$$(2^{4m+3})^{n \choose {n-1}}(2^{4m+2})^{n \choose {n-2}}\cdots(2^{4m-n+5})^{n \choose {1}}(2^{4m-n+4})^{n \choose {0}},$$
$$(2^{4m+4})^{n \choose {n-1}}(2^{4m+3})^{n \choose {n-2}}\cdots(2^{4m-n+6})^{n \choose {1}}(2^{4m-n+5})^{n \choose {0}},$$
respectively, inside $ko_N(BV(n))$.
\end{thm}

\begin{proof}
We proceed just as we did for $V(3)$ in Proposition $3.2.3$. For a non-trivial irreducible representation $\rho$ of $V(n)$ we can define $(2^{n}-1)-$tuples of eta invariants:
$$\overrightarrow{\eta}(1-\rho)=(\eta(\RP^{N}_{x_1})(1-\rho), \cdots,\eta(\RP^{N}_{x_1x_2 \cdots x_n})(1-\rho))$$
We can do this for each of the $2^n-1$ non-trivial irreducible representations. The order of the determinant $D^{'} \in \R/\Z$ of the resulting $(2^n-1) \times (2^n-1)$ square matrix gives a lower bound for the order of the group spanned in $(\R/\Z)^{2^n-1}$. This may be calculated by taking the determinant of the character table, and then dividing by $(2R)^{2^n-1}$, where $R$ denotes the order of $[\RP^N_x]$.\\
This follows because for any $y \in V(n)$, we know $\rho(y)= \pm 1$, so subtracting off the first row in the character table matrix (which is just the trivial representation), we get a new row $1-\rho$ and $(1-\rho(y))=2I_{\rho(y)=-1}$, with value $2$ if $\rho(y)=-1$ and zero else. Subtracting off the first row of the character table matrix from each other row gives a new matrix with the same determinant $D$ as the character table (compare Proposition 3.2.3). Now in order to calculate $D'$, we simply divide each non-zero entry by $2R$, where $R$ is the order of $[\RP^N]$, which is equivalent to dividing $D$ by $(2R)^{2^n-1}$\\
Now, in degree $N=8m+3$, we have $R=2^{4m+3}$ so that
$$ \eta(\RP^{N}_{y})(1-\rho)=2^{-4m-3}I_{\rho(y)=-1}$$
because the eta invariant will be either $0$, if $\rho$ restricts trivially to $<y>$, or $2^{-4m-3}$, if it doesn't. So we can apply this to each row to deduce that in $N=8m+3$ we get $D'=D(2^{-4m-4})^{2^n-1}$, and similarly for $N=8m+7$ we get $D'=D(2^{-4m-5})^{2^n-1}$.\\
By Lemma $3.2.1$ we know, up to sign, $D=(2^n)^{2^{n-1}}$, and we now claim
$$(2^n)^{2^{n-1}}=2^{2^n-1}2^{{n \choose 2} +2 {n \choose 3}+ \cdots + (n-1) {n \choose n}}$$
Equating exponents, it suffices to show
$$n2^{n-1}=2^n-1+{n \choose 2}+2 {n \choose 3}+ \cdots + (n-1) {n \choose n}$$
We argue by induction. For $n=1$, the left hand side is $1.2^0=1$, while the right hand side is $2^1-1+0=1$ as required.\\
Now suppose the claim is true for $n=k$, and consider the case $n=k+1$. We wish to show
$$(k+1)2^{k}=2^{k+1}-1+{k+1 \choose 2}+2 {k+1 \choose 3}+ \cdots + k {k+1 \choose k+1}$$
The left hand side is then:
$$(k+1)2^{k}=2(k2^{k-1}+2^{k-1})$$
$$=2(2^k-1+{k \choose 2}+2 {k \choose 3}+ \cdots + (k-1) {k \choose k}+2^{k-1})$$
by the induction hypothesis. Separating terms out this gives
$$=(2^k-1+2^k)+(2^k-1)+2({k \choose 2}+2 {k \choose 3}+ \cdots + (k-1) {k \choose k})$$
$$=2^{k+1}-1+({k \choose 1}+ {k \choose 2}+ \cdots +{k \choose k})+2({k \choose 2}+2 {k \choose 3}+ \cdots + (k-1) {k \choose k})$$
$$=2^{k+1}-1+({k \choose 1}+ 3{k \choose 2}+5{k \choose 3}+ \cdots +(2k-1){k \choose k})$$
$$=2^{k+1}-1+({k \choose 1}+ {k \choose 2})+2({k \choose 2}+{k \choose 3})+3({k \choose 3}+{k \choose 4})+ \cdots +(k-1)({k \choose k-1}+{k \choose k})+k({k \choose k}+{k \choose k+1})$$
$$=2^{k+1}-1+{k+1 \choose 2}+2 {k+1 \choose 3}+ \cdots + k {k+1 \choose k+1}$$
as required.\\
Thus for $N=8m+3$, we deduce that
$$D'=(2^{-4m-3})^{2^n-1}2^{{n \choose 2} +2 {n \choose 3}+ \cdots + (n-1) {n \choose n}}$$
The order of the group spanned is thus bounded below by the order of $D' \in \R/\Z$, which is just the reciprocal
$$1/D'=\frac{(2^{4m+3})^{2^n-1}}{2^{{n \choose 2} +2 {n \choose 3}+ \cdots + (n-1) {n \choose n}}}$$
Now of course
$$2^n-1={n \choose n-1} +{n \choose n-2} + \cdots {n \choose 1}+{n \choose 0}$$
so that
$$1/D'=\frac{(2^{4m+3})^{n \choose n-1}(2^{4m+3})^{n \choose n-2}\cdots(2^{4m+3})^{n \choose 0}}
{2^{{n \choose n-2}}2^{2 {n \choose n-3}}\cdots 2^{(n-1) {n \choose 0}}} $$
$$=(2^{4m+3})^{n \choose {n-1}}(2^{4m+2})^{n \choose {n-2}}\cdots(2^{4m-n+5})^{n \choose {1}}(2^{4m-n+4})^{n \choose {0}}$$
just as required. The situation is the same in $N=8m+7$, where we start with $2^{-4m-4}$ to get
$$1/D^{'}=(2^{4m+4})^{n \choose {n-1}}(2^{4m+3})^{n \choose {n-2}}\cdots(2^{4m-n+6})^{n \choose {1}}(2^{4m-n+5})^{n \choose {0}}$$
as required.

\end{proof}

We wish to realize all of $Ker(Ap) \subset ko_*(BV(n))$ by positive scalar curvature manifolds, where $A$ is the assembly map, and $p$ is the periodicity map inverting the Bott element. Thus from above it follows immediately that:
\begin{cor}
$Ker(A) \cap Im (p) \subset KO_*(BV(n))$ is realized entirely by $KO-$fundamental classes of positive scalar curvature manifolds.
\end{cor}
Thus to prove the Gromov-Lawson-Rosenberg conjecture for $V(n)$, we need only to realize all those classes lying in higher filtration, which vanish upon inverting the Bott element. This is what we do for $n=2,3$ in the next two sections.

\section{Projective bundles}
Classes which lie in higher local cohomology filtration for $ko_*(BV(n))$ are detected in ordinary $\Z_2$ homology (\cite{3}, chapters 10 and 12), and our method will be to realize these classes by using projective bundles of vector bundles over projective spaces. By Lemma $1.0.2$ such bundles will automatically carry positive scalar curvature metrics, so long as the dimension of the fibre is at least $2$. Furthermore, the following result enables us to calculate the cohomology and Stiefel-Whitney classes of such bundles, implying in particular that suitable choices of bundles will give spin manifolds. This and some of the resulting geometric constructions may also be found in \cite{cylin}.
\begin{thm}
Let $\pi:E\mapsto B$ be a real vector bundle of dimension $n$ with $\RP(\pi)$ the associated projective bundle. Then $H^{*}(\RP(\pi))=H^{*}(B)[t]/(t^{n} + t^{n-1}w_{1}(\pi) + \cdots + w_{n}(\pi))$, and \\
$w_{1}(\RP(\pi))=w_{1}(\pi)+w_{1}(B)+nt$\\
$w_{2}(\RP(\pi))=w_{2}(B)+w_{1}(B)(nt+w_{1}(\pi))+\frac{n(n-1)}{2}t^{2} + (n-1)w_{1}(\pi)t+w_{2}(\pi)$\\
where $t$ is the first Stiefel Whitney class of the canonical line bundle over $\RP(\pi)$.
\end{thm}
\begin{proof}
Note that there is a canonical line bundle $L$ over $\RP(\pi)$, where the fibre at a point $(x,b) \in \RP^{n-1} \times B$ is just $x$, and indeed $t=w_1(L)$. Denote the projection map of $\RP(\pi)$ by $p$. We can consider the pullback bundle $p^{*}(E)$ over $\RP(\pi)$. Clearly $L$ is a subbundle, and thus by choosing an orthogonal metric, we can assume $p^{*}(E)=L \oplus L^{\perp}$. From \cite{micc}, we know that the tangent bundle of $\RP^{n-1}$ is isomorphic to $Hom(\gamma_{1},\gamma_{1}^{\perp})$, where $\gamma_1$ is the canonical line bundle over $\RP^{n-1}$. It follows that the tangent bundle $T\RP(\pi)$ of $\RP(\pi)$ is isomorphic to $p^*(TB) \oplus Hom(L,L^{\perp})$, where $TB$ is the tangent bundle of the base $B$. Now $Hom(L,L)$ is trivial since it has a nowhere vanishing section, and so we deduce the stable isomorphism:
$$T\RP(\pi) \oplus \varepsilon \cong p^*(TB) \oplus Hom(L,L^{\perp}) \oplus Hom(L,L)$$
$$ \cong p^*(TB) \oplus Hom(L,p^*(\pi))$$
Now by the splitting principle, it suffices to assume that $\pi= l_1 \oplus \cdots \oplus l_n$ is a sum of line bundles, and that $p^*$ is injective.\\
Further, since $Hom(L,p^*(\pi))$ has a nowhere vanishing section, via the subbundle $L$, we must have $w_n(Hom(L,p^*(\pi)))=0$. Thus letting $x_i=w_i(L_i)$ we deduce
$$0=w_n(Hom(L,p^*(\pi)))=w_n(L \otimes p^*(\pi))$$
$$=w_n(L \otimes l_1 \oplus \cdots \oplus L \otimes l_n)$$
$$=(t+x_1)\cdots (t+x_n)$$
$$=t^n+t^{n-1}w_1(\pi)+\cdots +w_n(\pi)$$
So that the first part of the Theorem follows from the Leray-Hirsch Theorem, since restricting $t$ freely generates the cohomology of the fibre \cite{micc}. For the second part we simply note that by the Whitney product formula, we have the total Stiefel-Whitney class
$$w(\RP(\pi))=w(B)(1+t+x_1) \cdots (1+t+x_n)$$
and then multiply out the degree $1$ and $2$ parts.
\end{proof}
We now illustrate how to use this Theorem with some examples, which become relevant in the next section.
\begin{ex}Spin bundles over orientable projective spaces
\end{ex}
The idea is to start with an orientable non-spin projective space $\RP^{4i+1}$ with $i \geq 1$ and then consider the vector bundle $\pi:2 \gamma_1 \oplus A\epsilon \rightarrow \RP^{4i+1}$, where $\gamma_1$ is the canonical line bundle over $\RP^{4i+1}$, $\epsilon$ is the trivial bundle, and $A \equiv 2 \mod 4$ a natural number. Thus applying the Theorem with $H^*(\RP^{4i+1})=\Z_2[x]/x^{4i+2}$ we deduce $$H^*(\RP(\pi))=\Z_2[x,t]/x^{4i+2},t^{A+2}+x^2t^A$$
Further, $w_1(\RP(\pi))=0+0+2t=0$ and $w_2(\RP(\pi))=x^2+x^2+{A+2 \choose 2}t^2=0$,
since $A+2 \equiv 0 \mod 4$.\\
This construction along with ordinary products will produce sufficiently many classes to prove the GLR conjecture for $V(2)$.

\begin{ex}Iterated bundles over non-orientable projective spaces
\end{ex}
Starting over an even dimensional projective space requires taking two bundles in order to get a spin manifold, one to make $w_1$ vanish, and a second to then make $w_2$ vanish. So for instance in $n=4k$ we can start with the bundle $X=\RP(\gamma_{4i}^{1} \oplus (4j+1)\varepsilon \mapsto \RP^{4i})$ and then have $Y=\RP(L_{1} \otimes L_{0} \oplus  L_{1} \oplus  L_{0} \oplus (A) \varepsilon \mapsto X)$ with $A \equiv 1 \mod 4$. Here $L_1=L$ is the canonical line bundle over $X$ and $t=w_1(L)$, and $L_0$ is the pullback of the canonical line bundle over $\RP^{4i}$, with $x=w_1(L_0)$. Now we can calculate the total Stiefel-Whitney classes $w(\gamma_{4i}^{1} \oplus (4j+1)\varepsilon \mapsto \RP^{4i})=1+x$, and $w(L_{1} \otimes L_{0} \oplus  L_{1} \oplus  L_{0} \oplus (A) \varepsilon \mapsto N)=1+x^2+t^2+xt+x^2t+xt^2$ so that using Theorem $3.3.1$ we will then have $$H^{*}(Y)=\Z_{2}[x,t,u]/(x^{4i+1},t^{4j+2}+xt^{4j+1},u^{A+3}+x^{2}u^{A+1}+t^{2}u^{A+1}+xtu^{A+1}+x^{2}tu^{A}+xt^{2}u^{A})$$\\
It is easy to then apply Theorem $3.3.1$ to see that $w_1(X)=x+x+2t=0$, while $w_1(Y)=0+0+2u=0$ and $w_2(Y)=2(x^2+xt)+{A+3 \choose 2}u^2=0$.\\

For notational simplicity, we introduce the following notation:
\begin{df}
Consider triples of natural numbers $(a,b,c)$ with $c \equiv 3 \mod 4$, $b>1$ odd, and $a$ even. Suppose further that $b \equiv 1 \mod 4$ or $c \neq 3$.\\
By $M(a,b,c)$ we will mean the fundamental homology class in $H_n^{+}(BV(3))$ of a positive scalar curvature spin manifold of \emph{even} dimension $n=a+b+c$, constructed as an $\RP^c$ bundle over a manifold $X$, which in turn will be an $\RP^b$ bundle over $\RP^a$.
\end{df}
The manifolds representing these classes, along with classes of the form $M(a,4B+3,3)$, which are not covered under the above notation, will be defined in section $3.5$.\\

Thus in Example $3.3.3$, we write $M(4i, 4j+1, n-4i-4j-1)$ for the image of the fundamental class of $Y^n$ in $H_n^{+}(BV(3))$. Note that we could choose $j=0$ in Example $3.3.3$, but we choose not to for combinatorial reasons.\\
Analogous constructions may be performed to get $(4k+2)-$dimensional bundles, and bundles over $\RP^i$, with $i \equiv 2 \mod 4$. These will produce the necessary classes for $V(3)$, and details are in Proposition $3.5.2$.\\

We also remind ourselves that by Lemma $1.0.2$ the manifolds obtained by these and subsequent constructions all automatically admit positive scalar curvature metrics, and so we need only check that they are spin and that they span sufficiently large subgroups of $H_*(BV(n))$.

\section{The 2-column for V(2)}

The classes on the two column on the $E_{\infty}$ page of the local cohomology spectral sequence for $ko_*(BV(2))$ are all in even dimensions. The classes are detected in ordinary $\Z_2$ homology and span $\Z_2$ vector spaces of orders $2^k$ and $2^{k+1}$ in $n=4k+2$ and $n=4k$ respectively. Since there are no classes in first local cohomology in these dimensions, the following Proposition implies the conjecture for $V(2)$ immediately.
\begin{prop}
Spin Projective bundles over projective space span a subspace of dimension $2^{k}$ for $n=4k+2$ and $2^{k+1}$ for $n=4k$, in $H_{*}^{+}(BV(2))$.
\end{prop}
\begin{proof}
If $n=4k+2$, the manifolds $\RP^{a} \times \RP^{b}$, with $a,b\equiv 3 \mod 4$ , and $a+b=n$, which are trivial projective bundles, are spin, and give the classes $\xi_{(a,b)} \in H_{*}(BV(2))$.\\
Since $a+b=4k+2$ counting the possible values of $a$ gives us $k$ classes, as required.\\
If $n=4k$, the only product which is spin is $\RP^{4k-1} \times \RP^{1}$, which gives two classes, namely $\xi_{(4k-1,1)}$ and $\xi_{(1,4k-1)}$. The remaining classes can be generated by choosing projective bundles like in example $3.3.2$. We define
$$M_{(a,b)} := \RP(2L_{a} \oplus (n-1-a)\varepsilon \rightarrow \RP^{a})$$
 with $ 5 \leq a \equiv 1 \mod 4$ and $b=n-a+2$. Here $L_{a}$ is the canonical line bundle over $\RP^{a}$, and $\varepsilon$ is the trivial line bundle. \\
The reason for the $M_{(a,b)}$ notation used is that, when choosing generators of $H_{*}^{+}(BV(2))$, geometrically it is perhaps easier to visualize the submanifold dual to $w_{2}$ inside $\RP^{a} \times \RP^{b}$, where $b=n-a+2 \equiv 1 \mod 4$.\\
Indeed, if we consider $H^*(\RP^{a} \times \RP^{b})=\Z_2[x,y]/(x^{a+1},y^{b+1})$, then $w_2(\RP^{a} \times \RP^{b})=a^2+b^2$, so that the only classes in $H^{a+b-2}(\RP^{a} \times \RP^{b})$ that have non-trivial product with $w_2$ are $x^ay^{b-2}$ and $x^{a-2}y^b$. Thus it follows that if $M$ is the spin manifold dual to $w_2$ inside $\RP^{a} \times \RP^{b}$, then when we dualise in homology we get $[M] \rightarrow \xi_{(a,b-2)}+\xi_{(a-2,b)} \in H_{a+b-2}(BV(2))$. Thus we can view $H_{*}^{+}(BV(2))$ as being generated by the homology classes of products of projective spaces, along with manifolds dual to $w_2$ inside $\RP^{a} \times \RP^{b}$, so long as we know that these latter classes do indeed lie in $H_{*}^{+}(BV(2))$. The projective bundle construction we have given implies this, as we shall see.\\
We know $H^{*}(\RP^{a})=\Z_{2}[x]/x^{a+1}$ so that $w_{1}((2L_{a} \oplus (n-1-a)\varepsilon \rightarrow \RP^{a})=x+x=0$, and $w_{2}=x^2$, with the higher $w_{i}$ all zero . We thus have $$H^{*}(M_{(a,b)})=\Z_{2}[x,y]/(x^{a+1},y^{n-a+1}+x^{2}y^{n-a-1})$$
Using our formulae $w_{1}(M_{(a,b)})=w_{1}(\RP^{a})+w_{1}(2L_{a} \oplus (n-1-a)\varepsilon \rightarrow \RP^{a})+ny=0$, since $n$ is even and $a$ is odd. Similarly $w_{2}(M_{(a,b)})=x^2+n(n-1)/{2}y^{2} + x^2=0$, since $n\equiv 0 \mod 4$(all $w_{1}$ terms are zero), so this is indeed a spin manifold.\\
The unique non-zero top dimensional cohomology class (i.e. in dimension $n$) in $H^{*}(M_{(a,b)})=\Z_{2}[x,y]/(x^{a+1},y^{n-a+1}+x^{2}y^{n-a-1})$ is given by: $$x^{a}y^{n-a}=x^{a-2}y^{n-a+2}= \cdots =xy^{n-1}$$.\\
So looking at the map $H^*(BV(2))=\Z_2[X,Y] \rightarrow H^{*}(M_{(a,b)})$, we see that $X^kY^{n-k} \rightarrow xy^{n-1}$ as long as $k \leq a$ is odd. The dual map in homology then must send the dual class $(xy^{n-1})^{\vee}=[M_{(a,b)}] \in H_n(M_{(a,b)})$ to the sum $\sum \xi_{(i,j)}$, where $\xi_{(i,j)}$ is dual to $X^iY^j$ in the monomial basis, and the sum is over all $i,j$ with $ X^iY^j \rightarrow xy^{n-1}$ in cohomology. Thus we deduce from above that
 $$[M_{(a,b)}] \rightarrow \sum \xi_{(k,n-k)}\in H_{*}(BV(2))$$
 where the sum is over all $k \leq a$ is odd. Since we can choose bundles over all $\RP^a$ with $a \equiv 1 \mod 4$, and $5\leq a \leq n-3=4k-3$, there are $k-1$ classes independent classes obtained using this construction, so we get a total of $k+1$ homology classes as required.
\end{proof}
We now justify the choice of manifolds dual to $w_2$ inside $\RP^a \times \RP^b$ as generators, and thus the use of the $M_{(a,b)}$ notation.\\
As an example $[M_{(5,5)}] \rightarrow \xi_{(5,3)} +\xi_{(3,5)}+\xi_{(1,7)}$, but observe that the $\xi_{(1,7)}$ is the fundamental class of $\RP^{1} \times \RP^{7}$, so that we can choose our three generators for $H_{8}^{+}(BV(2))$ as $\xi_{(5,3)} +\xi_{(3,5)}, \xi_{(1,7)}$ and $\xi_{(7,1)}$. More generally we have the following:\\
\begin{lem}
$H_{4k}^{+}(BV(2))$ is spanned by the independent classes $\xi_{(a,n-a)} +\xi_{(a-2,n-a+2)}$, with $n=4k$ and $5 \leq a \equiv 1 \mod 4$, together with $\xi_{(4k-1,1)},\xi_{(1,4k-1)}$.
\end{lem}
\begin{proof}
We have seen that $[M_{(a,b)}] \rightarrow \sum \xi_{(k,n-k)}\in H_{*}(BV(2))$ where $k \leq a$ is odd. We now simply observe that the terms in the sum with $k \leq a-4$ are realized by smaller values of $a$. Explicitly we have that
$$f_{*}([M_{a,b}]) = \sum \xi_{(k,n-k)}\in H_{*}(BV(2))$$
$$=\xi_{(a,b-2)}+\xi_{(a-2,b)}+f_{*}([M_{a-4,b+4}])$$
where the sum is over all $k \leq a$ odd, and $f$ is the classifying map $M_{(a,b)} \rightarrow BV(2)$. For simplicity we say $M_{(1,b)}=\RP^{1} \times \RP^{b-2}$ so that $f_*(M_{(1,b)})=\xi_{(1,b-2)}$, and the claim follows.
\end{proof}
Before we move on to the rank three case, we return to the classes in first local cohomology briefly. As we have already seen, these are in dimensions $n=4k-1$, and detected using the eta invariant, by $ko-$fundamental classes of real projective spaces $\RP^{4k-1}$, via inclusions from cyclic subgroups. However, these classes are also detected in ordinary homology, and it will be useful in the next section to know their images explicitly.
\begin{lem}
The three inclusions $V(1) \hookrightarrow V(2)$ map the image of the fundamental class of $\RP^{4k-1}$ to the classes $\xi_{(4k-1,0)}, \xi_{(0,4k-1)}$ and $\sum \xi_{(a,b)}$, where the sum is over all non-negative integers $a,b$ with $a+b=4k-1$, respectively, in $H_{*}^{+}(BV(2))$
\end{lem}
\begin{proof}
It is easy to see that the three restriction maps $H^*(BV(2))=\Z_2[X,Y] \rightarrow H^*(BV(1))=\Z_2[x]$ induced in cohomology by the three inclusions are given by $X \mapsto x, y \mapsto 0$; $Y \mapsto x, X \mapsto 0$; and the diagonal $X,Y \mapsto x$. The image of $[\RP^{4k-1}]$ in $H_*(BV(1))$ is the class $\xi_{4k-1}$ dual to $x^{4k-1}$, so that the result follows immediately by applying the dual maps in homology.
\end{proof}

\section{The 3-column for V(3)}
Here there are Bott torsion classes to realize in both even and odd dimensions, and we need to first consider how many classes we get by including from the seven Klein 4 subgroups.\\
We note that given a class $\xi_{(a,b,c)} \in H_n^+(BV(3))$, the classes obtained by permuting the subscripts $a,b,c$ (so $\xi_{(b,a,c)}, \xi_{(a,c,b)}$ etc.) are also in $H_n^+(BV(3))$. This is immediate since any map $f:X \rightarrow Y$ between spaces sends $H_*^+(X)$ to $H_*^+(Y)$ by definition. Any permutation determines a map $P:BV(n) \rightarrow BV(n)$, which is in fact a homeomorphism. Thus for a class $x \in H_n^+(BV(3))$, we will often say $x$, \emph{along with its permutations}, by which we always mean $x$ along with all the classes obtained by permuting the subcripts.

\begin{prop}
Including fundamental classes of manifolds from $H_n^{+}(BV(2))$ into $H_n^{+}(BV(3))$ yields at least $6K+2$ independent classes if $n=4K$ and  $6K$ if $n=4K+2$.
\end{prop}
\begin{proof}
We say that $V(2)=<a_1,a_2>$, and $V(3)=<b_1,b_2,b_3>$. Then there are seven inclusions $V(2) \hookrightarrow V(3)$, and each of them may be viewed as linear maps and thus represented as matrices:
\begin{displaymath}
A=\left(\begin{array}{ccc}
1 & 0 \\
0 & 1 \\
0 & 0 \\
\end{array} \right),
B=\left(\begin{array}{ccc}
1 & 0 \\
1 & 0 \\
0 & 1 \\
\end{array} \right),
C=\left(\begin{array}{ccc}
1 & 1 \\
1 & 0 \\
0 & 1 \\
\end{array} \right)
\end{displaymath}
These give 3 inclusions, and the remaining 4 are obtained simply by permuting the rows of $A$ and $B$. Explicitly,
\begin{displaymath}
A'=\left(\begin{array}{ccc}
1 & 0 \\
0 & 0 \\
0 & 1 \\
\end{array} \right),
A''=\left(\begin{array}{ccc}
0 & 0 \\
1 & 0 \\
0 & 1 \\
\end{array} \right),
B'=\left(\begin{array}{ccc}
1 & 0 \\
0 & 1 \\
1 & 0 \\
\end{array} \right),
B''=\left(\begin{array}{ccc}
0 & 1 \\
1 & 0 \\
1 & 0 \\
\end{array} \right)
\end{displaymath}
Since $H^*(BV(2))=\Z_2[x_1,x_2]$ and $H^*(BV(3))=\Z_2[y_1,y_2,y_3]$, the induced maps in cohomology are determined just by taking the dual maps, which can be given by the transposes of the above matrices.\\
Thus we get $A^*(y_1)=x_1, A^*(y_2)=x_2, A^*(y_3)=0$, and by permuting the rows we obtain $A'$ and $A''$.\\
Further $B^*(y_1)=x_1, B^*(y_2)=x_1, B^*(y_3)=x_2$, and similarly for $B',B''$.\\
Finally $C^*(y_1)=x_1+x_2, C^*(y_2)=x_1, C^*(y_3)=x_2$.\\

Dualising now gives the map in homology. Letting $\xi_{(a_1,a_2)} \in H_*(BV(2))$ denote the element dual to $x_1^{a_1}x_2^{a_2}$ and $\xi_{(b_1,b_2,b_3)} \in H_*(BV(3))$ the element dual to $y_1^{b_1}y_2^{b_2}y_3^{b_3}$, and $X$ one of the inclusions given above, it follows that the dual map in homology is given by
$$X_*(\xi_{(a_1,a_2)})= \sum \xi_{(b_1,b_2,b_3)}$$
where the sum is over all triples $(b_1,b_2,b_3)$ such that $X^*(y_1^{b_1}y_2^{b_2}y_3^{b_3})=x_1^{a_1}x_2^{a_2}+ \cdots$, which simply means that $x_1^{a_1}x_2^{a_2}$ is a non-zero summand in the monomial basis decomposition of $X^*(y_1^{b_1}y_2^{b_2}y_3^{b_3})$.\\

So we now see explicitly that $A_*(\xi_{(a_1,a_2)})=\xi_{(a_1,a_2,0)}$, and similarly $A'_*(\xi_{(a_1,a_2)})=\xi_{(a_1,0,a_2)}$ and $A''_*(\xi_{(a_1,a_2)})=\xi_{(0,a_1,a_2)}$.\\

Since $B^*(y_1^{b_1}y_2^{b_2}y_3^{b_3})=x_1^{b_1+b_2}x_2^{b_3}$, dualising gives that
$$B_*(\xi_{(a_1,a_2)})=\sum_{i=0}^{a_1}\xi_{(i,a_1-i,a_2)}$$
and similarly for $B',B''$, where the position of the subscript with the fixed co-ordinate is permuted.\\
Finally $C$ can be dealt with by explicit calculations using the general formula, and indeed we do get extra classes this way, but combinatorially it is easier to realize these differently.\\

Consider now the case $n=4K$. We saw in the previous section that $H_n^+(BV(2))$ is of order $2^{K+1}$, and is generated by $\xi_{(a,n-a)}+\xi_{(a-2,n-a+2)}$ with $5 \leq a \equiv 1 \mod 4$, along with $\xi_{(n-1,1)}, \xi_{(1,n-1)}$.\\
Applying $A_*$ means we directly get the classes $\xi_{(a,n-a,0)}+\xi_{(a-2,n-a+2,0)}$ with $5 \leq a \equiv 1 \mod 4$, along with $\xi_{(n-1,1,0)}$ and the classes with the subscripts permuted in all possible ways, which is a total of $6+3(K-1)=3K+3$, all of which are clearly independent.\\
Further $B_*(\xi_{(n-1,1)})=\sum_{i=0}^{n-1}\xi_{(i,n-1-i,1)}$, while
$B_*(\xi_{(a,n-a)}+\xi_{(a-2,n-a+2)})=\sum_{i=0}^{a}\xi_{(i,a-i,n-a)}+\sum_{i=0}^{a-2}\xi_{(i,a-2-i,n-a+2)}$.\\
Taking permutations again means we get $3K$ classes.\\

Now consider any class $\xi_{(a,b,c)}$, with $n=4K$, $a,b,c \geq 1$ and $c$ odd, $a+b+c=n$ (so that $\xi_{(a,b,c)}, \xi_{(b,a,c)}$ are both summands in $\sum_{i=0}^{n-c}\xi_{(i,n-c-i,c)}$) and without loss of generality $a$ odd. Then the class $\xi_{(a,b,c)}$ (and $\xi_{(a,c,b)})$ occurs as a summand in $\sum_{i=0}^{n-a}\xi_{(a,i,n-a-i)}$. Further, if $d,e \leq n$ are odd and $d \neq c, e \neq a$ then $\xi_{(a,b,c)}$ \emph{can not} occur as a summand in $\sum_{i=0}^{n-d}\xi_{(i,n-d-i,d)}$ or $\sum_{i=0}^{n-e}\xi_{(e,i,n-e-i)}$. Thus we deduce that $\xi_{(a,b,c)}$ occurs as a summand only of $\sum_{i=0}^{n-c}\xi_{(i,n-c-i,c)}$ and $\sum_{i=0}^{n-a}\xi_{(a,i,n-a-i)}$, and any permutation of $\xi_{(a,b,c)}$ occurs exactly as a summand of a permutation of one of these two sums. Further, $\xi_{(a,b,c)}$ is the \emph{only} term that can occur as a summand in both these sums, because if $\xi_{(a',b',c')}$ is another such term, then we must have $a'=a$ and $c'=c$ and thus $b'=b$. It thus follows that if we add up all the classes obtained by permuting the subscripts of $\sum_{i=0}^{n-c}\xi_{(i,n-c-i,c)}$ and $\sum_{i=0}^{n-a}\xi_{(a,i,n-a-i)}$, then all the permutations of $\xi_{(a,b,c)}$ cancel out, and thus adding up all the $3K$ classes induced by $B_{*},B_{*}^{'}$ and $B_{*}^{''}$ gives zero.  This follows since we have seen that any $\xi_{(a,b,c)}$ occurs in exactly two of our summands, and that it is impossible for two different terms to occur in the same two summands. Since $c$ can be any odd number less than $n$ here this means we have one, and only one relation, giving the required total of $3K+3+3K-1=6K+2$. Note that we have not considered the case where one of $a,b$ or $c$ is zero, but all such summands can be directly cancelled by the classes induced by inclusion from $A_{*},A_{*}^{'}$ and $A_{*}^{''}$.\\

We now move on to $n=4K+2$. We know that here $H_n^+(BV(2))$ is of order $2^{K}$, with generators given by product classes $\xi_{(a,n-a)}$ with $3 \leq a \leq n-3, a \equiv 3 \mod 4$. Applying $A_*$ straightaway gives the three classes $\xi_{(a,n-a,0)},\xi_{(a,0,n-a)},\xi_{(0,a,n-a)}$ for each $a$, and here $B_*$ simply gives the three permutations of $\sum_{i=0}^{n-a}\xi_{(a,i,n-a-i)}$ for each $a$ and thus we have already obtained $3K+3K=6K$ independent classes.\\
\end{proof}

The number of generators obtained via inclusion grows linearly, and we know the kernel dimensions grow quadratically. The way to generate the remaining manifolds is by using iterated projective bundles of a vector bundle over a projective space, for example the manifold $Y$ of example $3.3.3$. As we've seen, suitable choices of such bundles yield spin manifolds.

\begin{prop}
Spin projective bundles can be constructed to give homology classes of the form $M(2i,4j+3,4k+3) \in H_*^+(BV(3))$ with $i,k \geq 1$ and $j \geq 0$, and of the form $M(2i,4j+1,4k+3)$ with $i,j \geq 1$ and $k \geq 0$. Further we can construct spin projective bundles to give us classes of the form $M(2,3,3)$ and $M(4,3,3)$.
\end{prop}
\begin{proof}
In dimensions $n=4K$ we can take the manifolds $Y=Y^n$ of Example $3.3.3$ with fundamental classes $M(4i, 4j+1, n-4i-4j-1)$, with $i,j \geq 1$.\\
As a slight aside, we remark that we choose not to take $j=0$ here. Indeed, calculation shows that taking $j=0$ produces homology classes in the span of those induced by inclusion via the permutations of $B_*$ in Proposition $3.5.1$. Thus combinatorially it is useful to think of the classes induced by inclusion via $B_*$ as permutations of $M(4i,1,n-4i-1)$.\\

We can also take the following iterated bundle over $\RP^{4i+2}$.\\

Start with the bundle $N=\RP(\gamma_{4i+2}^{1} \oplus (4j+3)\varepsilon \mapsto \RP^{4i+2})$, and then take $M=\RP(3L_{1} \otimes L_{0} \oplus  L_{1} \oplus  L_{0} \oplus (4K-4i-4j-3) \varepsilon \mapsto N)$, where again $L_1$ is the canonical line bundle over $N$, with $t=w_1(L_1)$ and $L_0$ is the pullback of the canonical line bundle over $\RP^{4i+2}$, with $x=w_1(L_0)$. Note that the fibre $\RP^{l-1}$ here must have dimension at least $7$, where $l=n-4i-4j-4$.\\
Applying Theorem $3.3.1$ we see that $w_1(N)=0, w_2(N)=xt$, and thus $w_2(M)=xt+xt+{4l \choose 2}u^2=0$ so that $M$ is indeed spin. Note that $w(\gamma_{4i+2}^{1} \oplus (4j+3)\varepsilon \mapsto \RP^{4i+2})=1+x$, and $w(3L_{1} \otimes L_{0} \oplus  L_{1} \oplus  L_{0} \oplus (4K-4i-4j-3) \varepsilon \mapsto N)=1+x^2+xt+xt^2+x^4+t^4+x^3t+xt^3+x^4t+xt^4+x^3t^2+x^2t^3$ so that by Theorem $3.3.1$ the cohomology ring of $M$ is\\

$H^{*}(M)=\Z_{2}[x,t,u]/(x^{4i+3},t^{4j+4}+xt^{4j+3},u^{4K-4i-4j-4}+xtu^{4K-4i-4j-6}+x^{2}tu^{4K-4i-4j-7}+xt^{2}u^{4K-4i-4j-7}+(x^{4}+t^{4}+x^{3}t+xt^{3})u^{4K-4i-4j-8}+(x^{4}t+xt^{4}+x^{3}t^{2}+x^{2}t^{3})u^{4K-4i-4j-9})$.\\

We denote the image of the homology fundamental class in $H_n^{+}(BV(3))$ by $M(4i+2, 4j+3, n-4i-4j-5)$.\\

We make analogous constructions in dimensions $n=4K+2$. Start with a bundle of the form $N=\RP(3\gamma_{4i+2}^{1} \oplus (4j+3)\varepsilon \mapsto \RP^{4i+2})$. Then $N$ is orientable, and we can consider the bundle $M=\RP(L_{1} \otimes L_{0} \oplus L_{1} \oplus L_{0} \oplus (4K-4i-4j-5)) \varepsilon \mapsto N)$. Much like above, $w_2(N)=x^2+t^2+xt=w_2(L_{1} \otimes L_{0} \oplus L_{1} \oplus L_{0} \oplus (4k-4i-4j-7))$ so that $M$ is indeed spin. We can apply Theorem $3.3.1$ again to see that such a bundle will have cohomology ring\\

$H^{*}(M)=\Z_{2}[x,t,u]/(x^{4i+3},t^{4j+6}+xt^{4j+5}+x^{2}t^{4j+4}+x^{3}t^{4j+3},u^{4K-4i-4j-4}+x^{2}u^{4K-4i-4j-6}+t^{2}u^{4K-4i-4j-6}+xtu^{4K-4i-4j-6}+x^{2}tu^{4K-4i-4j-7}+xt^{2}u^{4K-4i-4j-7})$,\\

and we again denote the image in homology by $M(4i+2, 4j+5, n-4i-4j-5)$.\\

Further, we can consider the bundle $N=\RP(3\gamma_{4i}^{1} \oplus (4j+1)\varepsilon \mapsto \RP^{4i})$, and then take $M=\RP(L_{1} \otimes L_{0} \oplus 3 L_{1} \oplus 3 L_{0} \oplus (4K-4i-4j-5) \varepsilon \mapsto N)$. Note again that the fibre $\RP^{l-1}$ here must have dimension at least $7$, where $l=n-4i-4j-2$.\\
Applying Theorem $3.3.1$ we see that $w_1(N)=0, w_2(N)=xt$, and thus $w_2(M)=xt+xt+{4l \choose 2}u^2=0$. Thus $M$ is indeed spin, and its cohomology ring is given by\\

$H^{*}(M)=\Z_{2}[x,t,u]/(x^{4i+3},t^{4j+4}+xt^{4j+3}+x^2t^{4j+2}+x^3t^{4j+1},u^{4K-4i-4j-4}+xtu^{4K-4i-4j-6}+x^{2}tu^{4K-4i-4j-7}+xt^{2}u^{4K-4i-4j-7}+(x^{4}+t^{4}+x^{3}t+xt^{3})u^{4K-4i-4j-8}+(x^{4}t+xt^{4}+x^{3}t^{2}+x^{2}t^{3})u^{4K-4i-4j-9})$,\\
and we denote the image of the homology fundamental class in $H_n^{+}(BV(3))$ by $M(4i,4j+3, n-4i-4j-3)$.\\

Again as a slight aside, we observe that the circle bundle $M=\RP(\gamma_{4i+2}^{1} \oplus \varepsilon \mapsto \RP^{4i+2})$ is spin. This follows since $H^{*}(M)=\Z_{2}[x,t]/(x^{4i+3},t^{2}+xt)$ so that $w_{2}(M)=t^{2}+xt=0$. So $M \times \RP^{n-4i-3}$ is a positive scalar curvature manifold generating three permutations of $\sum_{j=1}^{4i+2} \xi_{(j,4i+3-j,4K+2-4i-3)}$, but by Proposition $3.5.1$ these classes are already induced by inclusion. Combinatorially, it is helpful to think of these as giving three permutations of $M(4i+2,1,n-4i-3)$ classes.\\

Finally we wish to construct triples of the form $M(4i+2, 4j+3,3)$, and $M(4i,4j+3,3)$, which the above constructions fail to give. For the former case, consider first the case $(2,3,3)$. Let $\tau_n$ denote the tangent bundle of $\RP^n$ and consider the bundle
$$N^5=\RP(\tau_2 \oplus 2 \varepsilon \mapsto \RP^2)$$
Then $N$ is orientable since $w_1(N)=2w_1(\RP^2)+2t=0$, and $H^*(N)=\Z_2[x,t]/x^3,t^4+xt^3+x^2t^2$. Further, by construction, the tangent bundle of $TN$ of $N$ has a nowhere zero section since $w_5(TN)=0$, so that we have a splitting $TN=T \oplus \varepsilon$ where $T$ is a four dimensional bundle over $N$. Thus we can simply consider $M=\RP(T \mapsto N)$, which is an eight dimensional spin manifold since $w_2(M)=w_2(N)+w_2(T)+{4 \choose 2}u^2=2w_2(N)=0$. Here $u$ is the first Stiefel-Whitney class of the canonical line bundle over $M$. Thus $M$ is a spin iterated $\RP^3$ bundle over $N$, which is in turn an $\RP^3$ bundle over $\RP^2$, thus giving a homology class of the form $M(2,3,3)$ as required.\\

Similarly, for $M(4,3,3)$, we can start with the bundle $N^7=\RP(\gamma_4^1 \oplus 3\varepsilon \mapsto \RP^4)$, and similarly calculate that $w_5(N)=w_6(N)=w_7(N)=0$ so that $TN=T \oplus 3\varepsilon$ has three sections and then take $M=\RP(T \mapsto N)$. Thus $M$ will now be a spin manifold, which is an $\RP^3$ bundle over $N$, which is in turn an $\RP^3$ bundle over $\RP^4$, giving a homology class of the form $M(4,3,3)$ as required.
\end{proof}

Notice that the classes $M(a,b,c)$ are realized as fundamental classes of manifolds $M^n$ with $H^*(M^n)=\Z_2[x,t,u]/I$ for a suitably defined ideal $I$. The following lemma is then obvious by dualising.
\begin{lem}
If the class $M(a,b,c)$ is realized as the fundamental class of a manifold $M^n$ as constructed above, then $M(a,b,c)=\sum \xi_{(A,B,C)}$, where the $(A,B,C)$ range over all triples such that $H^n(M)=\Z_2<x^At^Bu^C>$. In particular $(a,b,c)$ is one such triple, and any such triple must have $A \leq a$ and $A+B \leq a+b$, since $M^n$ is an $\RP^c$ bundle over some manifold $N^{n-c}$, which is in turn an $\RP^b$ bundle over $\RP^a$.
\end{lem}
For the cases $M(4i+2, 4j+3,3)$, and $M(4i,4j+3,3)$ with $i,j>0$ we need an algebraic argument. Here the $M(a,b,c)$ classes will not be constructed directly as fundamental classes of manifolds, but rather by permuting the subscripts of the classes we already have.
\begin{lem}
i)$\xi_{(a-1,b+1,c)}$ occurs as a summand of $M(a,b,c)$, where $M(a,b,c)$ is as above.\\
ii)$\xi_{(n-8,3,5)}$ occurs as a summand of $M(n-8,5,3)$, with $n \geq 10$ even.
\end{lem}
\begin{proof}
This is immediate from the cohomology rings of the manifolds we construct. For example in dimensions $n=4k$, we saw that the manifold $Y$ of Example $3.3.3$, which represented classes of the form $M(4i,4j+1,n-4i-4j-1)$ has cohomology ring
$$H^{*}(Y)=\Z_{2}[x,t,u]/(x^{4i+1},t^{4j+2}+xt^{4j+1},u^{A+3}+x^{2}u^{A+1}+t^{2}u^{A+1}+xtu^{A+1}+x^{2}tu^{A}+xt^{2}u^{A})$$
from which we see that $x^{4i}y^{4i+1}=x^{4i-1}y^{4i+2}$ so that $H^n(Y)=\Z_2<x^{4i}t^{4j+1}u^{n-4i-4j-1}>=\Z_2<x^{4i-1}t^{4j+2}u^{n-4i-4j-1}>$. Dualising thus gives that $M(4i,4j+1,n-4i-4j-1)=\xi_{(4i,4j+1,n-4i-4j-1)}+\xi_{(4i-1,4j+2,n-4i-4j-1)}+ \sum \xi_{(A,B,C)}$, as required, where the sum is over other triples $(A,B,C)$ with $A+B+C=n$ such that $x^At^Bu^C=x^{4i}t^{4j+1}u^{n-4i-4j-1}$ in $H^*(Y)$. The other constructions giving us our $M(a,b,c)$ classes can be dealt with in exactly the same way, since we will always have $x^at^b=x^{a-1}t^{b+1}$ by construction, which proves the first part.\\

The second part is similar. If $X^n$ is the manifold whose fundamental class is $M(n-8,5,3)$, then we know $H^*(X)=\Z_2[x,t,u]/I$, where $I$ can be read off from Example $3.3.3$ and Proposition $3.5.2$. In particular we always have:
$$u^4+x^2u^2+t^2u^2+xtu^2 +x^2tu+xt^2u=0$$
Now $H^n(X)=\Z_2<x^{n-8}y^5u^3>$. Since $X$ is an iterated bundle over $\RP^{n-8}$, we must have $x^{n-7}=0$. Thus we can use the above formula to deduce $x^{n-8}y^3u^5=x^{n-8}y^5u^3$, so that the claim follows.
\end{proof}
Now, in order to produce a class of the form $M(2i,4j+3,3)$ with $j \geq 1$, we consider the class $M(2i,3,4j+3)$. From Proposition $3.5.2$ we know we can represent this class by a spin manifold which is an iterated projective bundle, and we simply define $M(2i,4j+3,3)$ as the class obtained by switching the last two subscript factors. Namely, if we have $M(2i,3,4j+3)= \sum \xi_{(a,b,c)}$ for suitable triples $(a,b,c)$ then $M(2i,4j+3,3) := \sum \xi_{(a,c,b)}$. Any continuous map $f:X \rightarrow Y$ between two spaces sends $H_*^{+}(X)$ to $H_*^{+}(Y)$, so $M(2i,4j+3,3) \in H_*^+(BV(3))$.\\
To obtain $M(4i+2,3,3)$ classes with $i \geq 1$, note that we can construct a manifold to represent the class $M(4,4i+1,3)$ which we know by Lemma $3.5.3$ is of the form $\xi_{(4,4i+1,3)}+\xi_{(3,4i+2,3)}+\sum \xi_{(a,b,c)}$ for suitable other triples $(a,b,c)$. Thus we can simply define $M(4i+2,3,3)$ as the class obtained from $M(4,4i+1,3)$ by permuting the first two factors, so that $M(4i+2,3,3)=\xi_{(4i+2,3,3)}+\xi_{(4i+1,4,3)}+\sum \xi_{(b,a,c)}$.\\

We would like to argue similarly to define $M(4i,3,3)$ for $i \geq 2$, namely by saying $M(4,4i-1,3)=\xi_{(4,4i-1,3)}+\xi_{(3,4i,3)}+\sum \xi_{(a,b,c)}$ and then taking the permutation swapping the first two subscript factors. However, $M(4,4i-1,3)$ is itself defined by permuting the last two subscript factors of $M(4,3,4i-1)$ and so this does not work. Instead we simply consider $M(2,4i+1,3)=\xi_{(2,4i+1,3)}+\sum \xi_{(a,b,c)}$. The manifold representing this class is an iterated $\RP^3$ bundle over a manifold which is an $\RP^{4i+1}$ bundle over $\RP^2$, and thus it follows that $\xi_{(3,4i,3)}$ is \emph{not} a summand of $M(2,4i+1,3)$. Nevertheless, we define $M(4i,3,3)$ as the class obtained by permuting the first two subscripts of $M(2,4i+1,3)$. Thus $M(4i,3,3)=\xi_{(4i+1,2,3)}+\sum \xi_{(b,a,c)}$ for suitable triples $(b,a,c)$. Thus $M(4i,3,3)$ gives us the class $\xi_{(4i+1,2,3)}$ as a summand, which does not appear as a summand  in the other $M(a,b,c)$ classes (see Lemma 3.5.6 below).\\

We will show that the four sets of $M(a,b,c)$ generators will suffice together with inclusion from the $V(2)$ subgroups. The $M(a,b,c)$ are parametrized by four sets of triples of natural numbers $(a,b,c)$, and the explicit constructions of these manifolds tell us directly which triples we can count.\\
The plan will be to first argue that the $M(a,b,c)$ classes are all independent. The key to the combinatorics is to remember that we are thinking of $M(a,b,c)$ as being the fundamental class of a manifold $Y^n$ which is an $\RP^c$ bundle over some manifold which is in turn an $\RP^b$ bundle over $\RP^a$.\\
We always have $H^*(Y)=\Z_2[x,t,u]/I$ for a suitable ideal $I$, and thus we have $M(a,b,c)= \sum \xi_{(A,B,C)}$, where the $(A,B,C)$ are all triples such that $H^n(Y)=\Z_2<x^At^Bu^C>$. Further, since $Y^n$ is an $\RP^c$ bundle over some manifold $X$ which is in turn an $\RP^b$ bundle over $\RP^a$, we must always have $A \leq a$ (since $x^{A+1}=0$) and $A+B \leq n-c$ (since $X$ has dimension $a+b=n-c$). Thus if we have $M(a,b,c)$ and $M(a',b',c')$ with, for example $a>a'$, then $\xi_{(a,b,c)}$ can not occur as a summand of $M(a',b',c')$ because this is the fundamental class of a manifold which is an iterated bundle over $\RP^{a'}$, so that any summand $\xi_{(A,B,C)}$ of $M(a',b',c')$ must have $A' \leq a'<a$. Thus in order to cancel $\xi_{(a,b,c)}$, we must have to consider some other $M(d,e,f)$ class with $d \geq a$, but then we can argue the same way for the pairs $(a,b,c)$ and $(d,e,f)$, and eventually we will run out of triples.\\
The other cases, including those $M(a,b,c)$ defined as permutations, are dealt with similarly, see Lemma $3.5.7$. We can then use similar reasoning to see independence from classes induced from $V(2)$.\\
Before giving a general argument, we work through a couple of example dimensions to give an idea of the general pattern.
\begin{ex}
$n=8$:
\end{ex}
The third local cohomology here is a $\Z_2$ vector space of order $2^{15}$, so we will produce 15 independent classes in the ordinary homology of $BV(3)$, represented by positive scalar curvature spin manifolds.\\
Firstly we can include from $V(2)$ subgroups, where we had the classes $\xi_{(7,1)},\xi_{(1,7)}$, both represented by $\RP^7 \times \RP^1$, and $\xi_{(3,5)}+\xi_{(5,3)}$ represented by the manifold $M_{(5,5)}$ described in the previous section. In the notation of Proposition $3.5.1$, it follows immediately that applying $A_*$ to these classes gives the classes $\xi_{(7,1,0)},\xi_{(1,7,0)},\xi_{(3,5,0)}+\xi_{(5,3,0)}$ in $H_8^{+}(BV(3))$, along with all their permutations, namely $\xi_{(7,0,1)},\xi_{(1,0,7)},\xi_{(0,7,1)},\xi_{(0,1,7)},\xi_{(3,0,5)}+\xi_{(5,0,3)}$ and $\xi_{(0,3,5)}+\xi_{(0,5,3)}$, which is a total of nine classes.\\
We can also apply $B_*$ to $\xi_{(7,1)}$ to obtain the classes $\sum_{i=1}^{7}\xi_{(1,i,7-i)}$, and applying $B_{*}^{'}$ and $B_*^{''}$ yields the permutations $\sum_{i=1}^{7}\xi_{(i,1,7-i)}$ and $\sum_{i=1}^{7}\xi_{(i,7-i,1)}$. Similarly, for $\xi_{(3,5)}+\xi_{(5,3)}$ we obtain the classes $\sum_{i=1}^{3}\xi_{(5,3-i,i)}+\sum_{i=1}^{5}\xi_{(3,5-i,i)}$ and its permutations $\sum_{i=1}^{3}\xi_{(3-i,5,i)}+\sum_{i=1}^{5}\xi_{(5-i,3,i)}$ and $\sum_{i=1}^{3}\xi_{(3-i,i,5)}+\sum_{i=1}^{5}\xi_{(5-i,i,3)}$.\\
It is easy to then check (Proposition 3.5.1) that there is precisely one relation between these six classes, namely that they add up to zero, so that we have 14 of the 15 classes needed. The last class is obtained by the $M(2,3,3)$ construction given above. We let $\tau_n$ denote the tangent bundle of $\RP^n$ and consider the bundle
$$N^5=\RP(\tau_2 \oplus 2 \varepsilon \mapsto \RP^2)$$
Then $N$ is orientable, and $H^*(N)=\Z_2[x,t]/x^3,t^4+xt^3+x^2t^2$, and further, by construction, the tangent bundle of $TN$ of $N$ has a nowhere zero section, so that we have a splitting $TN=T \oplus \varepsilon$ where $T$ is a four dimensional bundle over $N$. Then take $M=\RP(T \mapsto N)$, which is an eight dimensional spin manifold, with
$$H^*(M)=\Z_2[x,t,u]/x^3,t^4+xt^3+x^2t^2,u^4+x^2u^2+xtu^2+x^2tu$$
Dualising then gives the image of the fundamental class $[M] \rightarrow M(2,3,3)=\xi_{(2,3,3)}+\xi_{(1,4,3)}+\xi_{(1,2,5)}$ which is independent of those above, giving the required extra class.
\begin{ex}
$n=10$:
\end{ex}
The third local cohomology here is a $\Z_2$ vector space of order $2^{14}$, and we will produce 14 independent classes in the ordinary homology. The generators for $V(2)$ are $\xi_{(7,3)}$ and $\xi_{(3,7)}$, both coming from $\RP^7 \times \RP^3$. Applying $A_*$ again gives six permutations of $\xi_{(7,0,3)}$, and three each of $\sum_{i=1}^{7}\xi_{(3,i,7-i)}$ and $\sum_{i=1}^{3}\xi_{(7,i,3-i)}$. This gives twelve independent classes, and the remaining two may be chosen to be $M(2,5,3)=\xi_{(2,5,3)}+ \cdots$ and $M(4,3,3)=\xi_{(4,3,3)}+ \cdots$, which we constructed in Proposition $3.5.2$. These are independent of each other because if $\xi_{(a,b,c)}$ is a summand of $M(2,5,3)$ then $a \leq 2$ so that the $\xi_{(4,3,3)}$ summand of $M(4,3,3)$ can not be cancelled by any summand of $M(2,5,3)$. Similarly, the six classes $\xi_{(5,3,2)}, \xi_{(3,5,2)},\xi_{(5,2,3)},\xi_{(7,1,2)},\xi_{(7,2,1)}$ and $\xi_{(1,7,2)}$ are summands of each of the permutations of $\sum_{i=1}^{7}\xi_{(3,i,7-i)}$ and $\sum_{i=1}^{3}\xi_{(7,i,3-i)}$. By Lemma $3.5.2$ these can not be cancelled by any summand of $M(2,5,3)$ or $M(4,3,3)$, from which it follows that these classes are independent of those induced from $V(2)$, giving us the required $14$ independent classes.\\

We now move on to the general even dimensional case, arguing along the same lines. Note that in even dimensions there are no first local cohomology classes detected in periodic K-theory, so as above, we directly produce sufficiently many classes in ordinary homology. The first step is to convince ourselves that the $M(a,b,c)$ classes we have constructed form an indpendent set. Note that in dimension $n$, we have such a class for any triple of natural numbers $(a,b,c)$, with $a$ even, $3 \leq b$ odd, $c \equiv 3 \mod 4$, and $a+b+c=n$.
\begin{lem}
The $M(a,b,c)$ classes form a linearly independent set in $H_n^+(BV(3))$.
\end{lem}
\begin{proof}
First suppose that either $b \equiv 1 \mod 4$ or $c \neq 3$. Then we have seen that we can construct $M(a,b,c)$ directly as the fundamental class of some manifold $Y^n$, with $n=a+b+c$ and $H^*(Y)= \Z_2[x,t,u]/I$, where I is an ideal. By dualising it follows that $M(a,b,c)=\xi_{(a,b,c)}+ \sum \xi_{(A,B,C)}$, where the sum is over all other triples $(A,B,C)$ such that $H^n(Y)=\Z_2<x^At^Bu^C>$. Since $Y$ was constructed as an $\RP^c$ bundle over a manifold $X$ which was in turn an $\RP^b$ bundle over $\RP^a$, we must have $A \leq a$ and $A+B \leq a+b$. This also applies to the classes $M(2,3,3)$ and $M(4,3,3)$ which we represented explicitly as fundamental classes of manifolds.\\
If $7 \leq b \equiv 3 \mod 4$ and $c=3$ then we defined $M(a,b,c)$ by taking a suitable permutation of the subscripts of the class $M(a,c,b)$ which we directly constructed as the fundamental class of a manifold $Y^n$, with $H^*(Y)= \Z_2[x,t,u]/I$ again. We chose the permutation so that $M(a,b,c)=\xi_{(a,b,c)}+ \sum \xi_{(A,B,C)}$, where the sum is over all other triples $(A,B,C)$ such that $H^{n}(Y)=\Z_2<x^At^Cu^B>$ .\\
Finally if we have $a>4$ and $b=c=3$ then if $ a \equiv 2 \mod 4$ then $M(a,3,3)= \xi_{(a,3,3)}+ \sum \xi_{(A,B,C)}$, for suitable triples $(A,B,C)$ having in particular $A \leq a$ and $A+C \leq a+3$. On the other hand if $a \equiv 0 \mod 4$, then we had $M(a,3,3)= \xi_{(a+1,2,3)}+ \sum \xi_{(A,B,C)}$.\\
Now if we have classes $M(a,b,c)$ and $M(a',b',c')$ coming from two distinct triples with $a+b+c=n=a'+b'+c'$, then without loss of generality we can assume $a>a'$ or $a=a'$ and $b>b'$. In either case, by construction, $\xi_{(a,b,c)}$ is a summand of $M(a,b,c)$ but not of $M(a',b',c')$, except if $a=n-6$ and $n \equiv 2 \mod 4$, in which case $(a,b,c)=(n-6,3,3)$ and in this case $\xi_{(n-5,2,3)}$ is a summand of $M(n-6,3,3)$ but not of $M(a',b',c')$.\\
Thus any linear relation must involve a third class, say $M(d,e,f)$, and we can now apply the same argument to the triples $(a,b,c)$ and $(d,e,f)$ to deduce a relation must involve yet another triple. Since we must have $d \geq a$, we either have $d>a$, or $d=a$ and $b<e$. Since either the first subscript or the sum of the first two subscripts must increase, we deduce by induction that this argument applies to all the $M(a,b,c)$ classes, so that there are no linear relations.
\end{proof}
We will now to combine our $M(a,b,c)$ classes with those induced by inclusion to check that we have sufficiently many classes. We will show that there are no relations between the $M(a,b,c)$ classes and those induced by inclusion. The method is to rule out the possibility of a relation with specific homology classes induced by inclusion. Like in the proof of Lemma $3.5.7$, this may be done by singling out particular classes that can not be cancelled. This enables us to deduce recursively that any linear relation must include more classes than we initially started with, and we can then repeat the argument until we have no classes left. The GLR conjecture for $V(3)$ in even dimensions follows immediately from the following proposition.
\begin{prop}
Iterated projective bundles realize all the Bott torsion in $ko_*(BV(3))$ in even dimensions.
\end{prop}
\begin{proof}
Consider first the case $n=4k+2$. The kernel dimension is $k^{2}+5k$. We have seen that inclusion from $V(2)$ gives us at least $6k$ classes, leaving us needing $k^2-k$ more.\\
We count the number of $M(a,b,c)$ generators described above. If $a \equiv 2 \mod 4$, then we can choose any $b \equiv 1 \mod 4, 5 \leq b \leq n-3-a$ to get an $M(a,b,n-a-b)$ class represented by a spin manifold. There are $k-1$ choices of $b$ for $a=2$, $k-2$ for $a=6$ and so on, giving a total of $\sum_{i=1}^{k}(k-i)=\sum_{i=1}^{k-1}i=k(k-1)/2$.\\
On the other hand if $a \equiv 0 \mod 4$, we can again choose any $b \equiv 3 \mod 4, 3 \leq b \leq n-3-a$, so that the total here is $\sum_{i=1}^{k}(k-i)=\sum_{i=1}^{k-2}i=k(k-1)/2$.\\
Adding these up gives $k^2-k$ just as required.\\
We know these $M(a,b,c)$ classes are independent of each other, and we claim that they are independent of the classes induced from $V(2)$ also. Recall that these are explicitly given by the permuting the subscripts of $\xi_{(a,n-a,0)}$ and $S_{a}^1:=\sum_{i=1}^{n-a-1} \xi_{(a,n-a-i,i)}$, with $a \equiv 3 \mod 4$ and $3 \leq a \leq n-3$. Here $S_{a}^i$ means the class obtained by permuting the subscripts of $S_{a}^1$ to have the $a$ in the i-th position. So $S_{a}^2:=\sum_{i=1}^{n-a-1} \xi_{(n-a-i,a,i)}$ and similarly $S_{a}^3:=\sum_{i=1}^{n-a-1} \xi_{(n-a-i,i,a)}$. So for $S_{a}^i$ we just put the fixed value $a$ in the $i-$th position, and take the sum over all such classes in a fixed dimension $n$. Thus the subscript denotes the value of the fixed co-ordinate, and the superscript its position . To get a feel for the combinatorics, it is useful to think of $S_{a}^3$ as giving us a class of the form $M(n-a-1,1,a)$, and the $S_{a}^i$ as the class obtained by suitably permuting the subscripts (compare Proposition $3.5.2$).\\
No combination of the $M(a,b,c)$ classes are in the span of any of the $S_{n-3}^i$. This is immediate because $a \geq 2, b \geq 3$ and $\xi_{(a,b,c)}$ is a summand of $M(a,b,c)$ ($\xi_{(a+1,2,3)}$ in the case of $M(n-6,3,3)$), and $S_{n-3}^1=\xi_{(n-3,1,2)}+\xi_{(n-3,2,1)}$, and similarly for $S_{n-3}^2,S_{n-3}^3$. Further, the same is true for $S_{3}^i$, because these contain $\xi_{(n-7,4,3)},\xi_{(n-7,3,4)},\xi_{(3,n-5,2)}$ as summands for $i=3,2,1$ respectively. These classes can not be cancelled by any of the $M(a,b,c)$ classes constructed directly as fundamental clases of manifolds, because we always have $a \leq n-5$ and $a+b \leq n-3$. It follows that the first two can only be cancelled by $M(n-6,3,3)$, but this is obtained from $M(2,n-5,3)$ by swapping the first and second subscripts, which means if $\xi_{(A,B,C)}$ is a summand of $M(n-6,3,3)$ then $B \leq 2$. The last class $\xi_{(3,n-5,2)}$ can only be cancelled by some class $M(4A,4B+3,3)$ with $A \geq 1$, but then we would have to cancel the $\xi_{(4A,4B+3,3)}$ term that is a summand of $M(4A,4B+3,3)$. This would have to be done by a class $M(D,E,F)$ with $D>4A$, and thus induction and Lemma $3.5.7$ tell us there can be no linear relation. Note that these observations already suffice for the case $n=10$.\\
Thus we are left to consider if some combination of our $M(a,b,c)$ classes is in the span of the $S_{a}^i$, with $7 \leq a \leq n-7$.\\
So consider the class $M(4A,4B+3,4C+3)$, with $A\geq 1$ and $B,C \geq 0$. If $C=0$, then by the above remarks, if some combination  $M(4A,4B+3,4C+3)+\sum M(a,b,c)$ for suitable triples $(a,b,c)$ is the same as some class $x$ induced by inclusion (meaning if there is some linear dependence), then we can deduce $S_{4B+3}^2$ is a summand of $x$. However $\xi_{(4A+1,4B+3,2)}$ is a summand of $S_{4B+3}^2$, and it is clear that this class cannot be cancelled by any of the other $S_{a}^i$.\\
Since $4A+1+4B+3=n-2$ it also can not be cancelled by any $M(a,b,c)$ that we defined as the fundamental class of a manifold which was an $\RP^N$ bundle with $N \geq 3$. Thus the only possibilities are the permutation classes $M(n-6,3,3)$ and $M(n-4k-2,4l+3,3)$ with $k \geq 2$ and $l \geq 1$. Now $M(n-6,3,3)$ was defined by permuting the first two subscripts of $M(2,n-5,3)$, but any summand $\xi_{(a,b,c)}$ of this class must have $a+b \leq n-3$, and thus the same is true for $M(n-6,3,3)$.\\
This leaves the possibility of $M(n-4k-2,4l+3,3)$ with $k \geq 2$ and $l \geq 1$, which was defined by permuting the last two subscripts of the class $M(n-4k-2,3,4l+3)$. If $\xi_{(4A+1,4B+3,2)}$ were to cancel, then we must have $n-4k-2-4A+1 \geq 3$ and thus $l<B$. This leaves us now needing to cancel $\xi_{(n-4k-2,4l+3,3)}$, but repeating the above argument means we reach the $B=0$ case by induction, which we have already dealt with.\\
If $C \geq 1$, then proceeding as above we would deduce that either $S_{4B+3}^2$ or $S_{4C+3}^2$ is a summand of $x$. In the former case, just like above we observe that $\xi_{(n-4B-5,4B+3,2)}$ is a summand of $S_{4B+3}^2$ and apply the same argument. In the latter case, we similarly note that $\xi_{(n-4C-5,2,4C+3)}$ is a summand of $S_{4C+3}^3$. This can only occur as a summand of some $M(a,b,4D+3)$ with $a>n-4C-5, b \geq 3$ and thus $D<C$. Proceeding inductively we reach the $C=0$ case again.\\
Now take the case $M(4A+2,4B+5,4C+3)$ with $A,B,C \geq 0$. The $C=0$ case is dealt with the same way, and so like above we deduce there is a linearly dependent class $x$ with a summand $S_{4C+3}^3$. Again $\xi_{(n-4C-5,2,4C+3)}$ is a summand of $S_{4C+3}^3$, and exactly as above, this can only occur as a summand of some $M(a,b,4D+3)$ with $a>n-4C-5, b \geq 3$ and thus $D<C$, so that induction reduces the problem to the $C=0$ case.\\
Thus the $M(a,b,c)$ classes must be independent of the $S_a^i$ classes. Independence from the $\xi_{(a,n-a,0)}$ classes and those obtained by permuting subscripts is immediate, and thus we have realized sufficiently many classes.\\

Now consider the case $n=4k$. Here the kernel dimension is $k^{2}+4k+3$ of which we have shown $6k+2$ is induced from $V(2)$, leaving $k^{2}-2k+1$. Just as above, we can count. Classes $M(a,b,c)$ with $ a \equiv 0 \mod 4$ may be counted for any $b$ with $\leq b \equiv 1 \mod 4, 5 \leq b \leq n-3-a$ again giving $\sum_{i=1}^{k}(k-i-1)=(k-2)(k-1)/2$.\\
On the other hand, for $a \equiv 2 \mod 4$ we can choose any $b$ with $b \equiv 3 \mod 4, 3 \leq b \leq n-3-a$, which is again $k(k-1)/2$, so that adding up gives $(k-1)^{2}$ as required.\\
The independence argument is similar. The classes induced by inclusion here are explicitly given by the $\xi_{(1,n-1,0)}$, permuting the subscripts of which gives six classes, the three classes obtained by permuting the subscripts of $S_{1}^i$ of the sum $S_{1}^1:=\sum_{i=1}^{n-1} \xi_{(1,n-1-i,i)}$ and the three permutations of $S_{a}^1:=\sum_{i=1}^{n-a-1} \xi_{(a,n-a-i,i)}+\sum_{i=1}^{n-a-3} \xi_{(a+2,n-a-2-i,i)}$, with $3 \leq a \equiv 3 \mod 4$. We have exactly one relation, which is that adding up all of the $S_{a}^i$ and the $S_{1}^i$ gives zero. Again, it is useful to think of $S_{a}^3$ as giving us a class of the form $M(n-a-1,1,a)$, and the $S_{a}^i$ as the class obtained by suitably permuting the subscripts.\\
Again, we first rule out relations with the $S_{3}^i$. Observe that the $S_{3}^2$ contains $\xi_{(n-5,3,2)}$ as a summands. These can not be cancelled by any of our $M(a,b,c)$ classes because the only one that could cancel it is $M(n-6,3,3)=\sum \xi_{(a,b,c)}$, but here any admissible triple $(a,b,c)$ must always have $a+b \leq n-3$, since the class is defined by permuting the first and second subscripts of $M(4,n-5,3)$.\\
For $S_{3}^3$, we observe that $\xi_{(n-8,5,3)}$ is a summand which can only be cancelled by (a summand of) $M(n-8,5,3)$, and Lemma $3.5.4$ shows that $\xi_{(n-8,3,5)}$ is a summand of $M(n-8,5,3)$, and this can not be cancelled by a summand of $S_{3}^3$ or any of the other $M(a,b,c)$ classes. Indeed, the only class that could cancel $\xi_{(n-8,3,5)}$ is $M(n-6,3,3)$, which is defined by permuting the first two subscripts of each summand of $M(4,n-7,3)$. This is the fundamental class of a manifold $Y$ as in Example $3.3.3$, and from there we can easily see that $x^3y^{n-8}z^5=0$ in $H^n(Y)$. Thus $\xi_{(3,n-8,5)}$ is not a summand of $M(4,n-7,3)$ so that $\xi_{(n-8,3,5)}$ is not a summand of $M(n-6,3,3)$.\\
Thus the only possibility is that some collection of the $M(a,b,c)$ classes sum up to $S_{3}^1$, but this is impossible since summing up any collection of the $M(a,b,c)$ classes gives us a sum $\sum \xi_{(A,B,C)}$ for suitable triples $(A,B,C)$, and by their construction at least one such admissible triple must have the first subscript $A$ even.\\
Now start with the class $M(4A+2,4B+3,4C+3)$, with $A,B,C \geq 0$, and suppose $M(4A+2,4B+3,4C+3)+\sum M(D,E,F)$ is a linearly dependent class for some collection of triples $(D,E,F)$. By the above remark we only need to consider the cases $B \geq 1$ and $C \geq 1$. Much like before, the former case implies a relation with a class $x=S_{4B+3}^2+ \cdots$. Now $\xi_{(n-4B-5,4B+3,2)}$ is a summand of $S_{4B+3}^2$, and it is clear that this class cannot be cancelled by any of the $M(a,b,c)$. Thus $S_{n-4B-5}^1$ must also be a summand of $x$. Now $\xi_{(n-4B-5,4B+4,1)}$ is a summand of $S_{n-4B-5}^1$, and thus $S_{1}^3$ must also be a summand of $x$, and so on. Repeating the argument inductively, we can check that all the $S_{a}^i$ must occur as summands of $x$.\\
However, adding up \emph{all} of the $S_{a}^i$ gives zero, and thus there than be no linear relation, since otherwise the $M(a,b,c)$ classes would be linearly dependent, contradicting Lemma $3.5.7$.\\
On the other hand if $x=S_{4C+3}^3+ \cdots$, then we know $\xi_{(n-4B-5,2,4C+3)}$ is a summand of $S_{4C+3}^3$. If this is realised by some class $M(a,b,4D+3)$ we must have $D<C$, so that inductively we reach $C=0$ again. Alternately $S_{n-4B-5}^1$ must also be a summand of $x$, and the argument proceeds as above.\\
Finally take the case of $M(4A,4B+1,4C+3)$, with $A,B \geq 1, C \geq 0$. Any possible relations with $S_{4C+3}^3$ are dealt with as above, and so it suffices to consider relations with classes of the form $x=S_{4B+1}^2+ \cdots$. However $\xi_{(n-4B-2,4B+1,1)}$ is a summand which can not be cancelled out by an $M(a,b,c)$. Again $S_{1}^3$ must also be a summand of $x$, and repeating the argument inductively again implies that all the $S_{a}^i$ must occur as summands of $x$.\\
\end{proof}

The odd dimensional case is a little easier, but if $n=4k-1$ we have to be careful to ensure that the classes we generate are distinct from those detected in periodic K-theory. This can again be done by explicitly calculating the images in $\Z_2$ homology. The classes detected in periodic K-theory are the $ko-$ fundamental classes of real projective spaces $\RP^{4k-1}_x$, induced by the maps $\RP^{11}_x \rightarrow B<x> \hookrightarrow BV(3)$, where $<x>$ is any cyclic subgroup of $V(3)$. We will see that these classes also map non-trivially to the ordinary homology, and their images are in fact independent of those classes obtained by our projective bundle constructions. By the following trivial algebraic lemma, it will then follow that all of $ko_{4k-1}(BV(3))$ is spanned by $ko-$fundamental classes of positive scalar curvature spin manifolds.
\begin{lem}
Let $A$ be a finite abelian $2-$group and $a_1, \cdots,a_r \in A$. If the classes $b_1, \cdots,b_r \in A/2A$, with $b_i=a_i$ reduced modulo $2$, span $A/2A$, then $a_1, \cdots,a_r$ span $A$.
\end{lem}
We again start by giving a couple of example dimensions.
\begin{ex}
$n=11$:
\end{ex}
The kernel dimension in third local cohomology is $8$, and we start by taking $\RP^{1} \times \RP^{3}\times \RP^{7}$, whose image in homology will be $\xi_{(1,3,7)}$, and it is clear that the six permutations give us six generators for our kernel.\\
Next we take $\RP^{3}\times M_{(5,5)}$. This gives us $\xi_{(3,3,5)} + \xi_{(3,5,3)}$. The reverse $M_{(5,5)} \times \RP^{3}$ gives $\xi_{(5,3,3)} + \xi_{(3,5,3)}$, while the \emph{middle} permutation $M_{5,3,5}$ gives $\xi_{(3,3,5)} + \xi_{(5,3,3)}$, and it is thus immediate that $[\RP^{3}\times M_{(5,5)}]+[M_{(5,5)} \times \RP^{3}]=[M_{5,3,5}]$, giving us a kernel dimension of $6+3-1=8$ as required.\\

As we have seen, the classes in first local cohomology are detected by the eta invariant, and represented by the $ko-$fundamental classes of real projective spaces $\RP^{11}_x$, induced by the maps $\RP^{11}_x \rightarrow B<x> \hookrightarrow BV(3)$, where $<x>$ is any cyclic subgroup of $V(3)$. Using Proposition $3.5.1$ and Lemma $3.4.3$, it is easy to calculate the images of these seven classes in $H_*(BV(3);\Z_2)$ explicitly. Indeed, in the notation of the Proposition, inclusion via $A_*$ gives the three classes $\xi_{(11,0,0)},\xi_{(0,11,0)}$ and $\xi_{(0,0,11)}$. Inclusion using $B_*$ gives the three permutations of the sum $\Sigma_{a+b=11} \xi_{(a,b,0)}$, while the diagonal inclusion $C_*$ simply gives $\Sigma_{a+b+c=11} \xi_{(a,b,c)}$. The presence of summands $\xi_{(a,b,c)}$ with two out of $a,b,c$ even immediately tells us that these classes are all independent of the $8$ we gave above, so that by Lemma $3.5.9$, we have spanned all of $ko_{11}(BV(3))$.

\begin{ex}
$n=13$:
\end{ex}
The kernel dimension is $13$, and we first consider $\RP^{1} \times \RP^{1}\times \RP^{11}$ which maps to $\xi_{(1,1,11)}$. Similarly $\RP^{3} \times \RP^{3}\times \RP^{7}$ maps to $\xi_{(3,3,7)}$, so permuting clearly gives us $6$ independent elements.\\
Next we can take $\RP^{1}\times M_{(5,9)}$ which maps to $\xi_{(1,5,7)} + \xi_{(1,3,9)}$. Permutations of this are independent, so we get $6$ more elements.\\
Finally we have $M_{(5,5,5)}$, which, like in the rank 2 case, can be viewed as dual to $w_2$ inside $\RP^5 \times \RP^5 \times \RP^5$. The fundamental class of this manifold maps to $\xi_{(5,3,5)} + \xi_{(5,5,3)}+\xi_{(3,5,5)}$ which gives us the last of the $13$ generators needed. Much like for $V(2)$, we can realize this homology class by an appropriate projective bundle.\\

We now give the general argument for odd dimensions. Proposition $3.5.8$ together with the following result then imply the GLR conjecture for $V(3)$ immediately.
\begin{prop}
Iterated projective bundles realize all the Bott torsion in $ko_*(BV(3))$ in odd dimensions.
\end{prop}
\begin{proof}
Consider first the case of manifolds with dimension $n=4k+1$:\\
The dimensions of the kernels in the first few dimensions are 3,7,13,21 and so on. More generally, for $n=4k+1$ the kernel dimension is $1+\sum_{i=1}^{k}2i=1+k+k^{2}$ \cite{3}. We shall now explicitly generate this kernel.\\
1. By permuting the subscripts we get 3 generators via $\RP^{4k-1} \times \RP^{1}\times \RP^{1}$. \\
2. We also have generators of the form $\RP^{a} \times \RP^{b}\times \RP^{c}$ where $a,b,c \equiv 3\mod 4$ and $a+b+c=4k+1$. Writing $a=4A+3$ ($A\in \N_{0}$) and so forth we have $4A+3+4B+3+4C+3=4K+1$ which reduces to $A+B+C=K-2$. It is easy to check that this gives us $k(k-1)/2$ manifolds.\\
3. Let $M_{(a,b)}\hookrightarrow  \RP^{a} \times \RP^{b}$ denote the submanifold dual to the second Stiefel-Whitney class, as in the previous section. We can consider manifolds of the form $\RP^{1}\times M_{(a,b)}$ with $a,b\equiv 1 \mod 4, a,b\geq 5$. Fixing $\RP^{1}$ clearly gives $k-1$ of these, so permuting gives a total of $3(k-1)$ manifolds.\\
4. Finally we can consider $M_{(a,b,c)}\hookrightarrow \RP^{a} \times \RP^{b}\times \RP^{c}$ defined as before, dual to the second Stiefel-Whitney class. Here we need $a+b+c=4k+3,a,b,c \equiv 1 \mod 4$ and all of $a,b,c \geq 5$. This reduces to triples $\{(A,B,C)\in (\N)^{3}|A+B+C=k\}$ of which there are $(k-1)(k-2)/2$ .\\
Adding up the number of generators described gives $3+k(k-1)/2 + 3(k-1) +(k-1)(k-2)/2=1+k+k^{2}$
as required. \\
Before we consider independence of these classes, we remark that analogous to the rank 2 case, we can alternatively describe the $M_{(a,b,c)}$ manifolds as projective bundles, this time over products of projective spaces. Explicitly we can write down spin manifolds by
$$\RP(2L_a \oplus 2L_b \oplus (c-5)\varepsilon \rightarrow \RP^a \times \RP^b)$$
Where $L_a$ and $L_b$ denote the pullbacks via projection of the canonical line bundles over $\RP^a$ and $\RP^b$ respectively. Entirely analogous constructions give us a spin  $\RP^a$ bundle over $\RP^b \times \RP^c$ and an $\RP^b$ bundle over $\RP^a \times \RP^c$. A calculation similar to lemma $3.4.2$ confirms that we can view the fundamental class of $M_{(a,b,c)}\hookrightarrow \RP^{a} \times \RP^{b}\times \RP^{c}$ dual to $w_2$ as the appropriate generator.\\

We now return to the linear independence. There are no classes on the $E_{\infty}$ page in first local cohomology, so it will suffice to see that the images of the fundamental classes of these manifolds are all linearly independent in the mod 2 homology of $BV(3)$. To this end we again denote by $\xi_{(a,b,c)} \in H_{a+b+c}(BV(3))$ the homology element dual to $x^{a}y^{b}z^{c} \in H^{a+b+c}(BV(3))$ where $x,y,z$ are the three cohomology generators.\\
 It is then immediate that $\RP^{a} \times \RP^{b}\times \RP^{c}$ maps independently to $\xi_{(a,b,c)}$ ensuring that the generators from 1. and 2. are independent. Further, we have $[M_{a,b}]=\xi_{a-2,b}+\xi_{a,b-2}$ meaning the generators in 3. map to $\xi_{(a,b-2,1)}+\xi_{(a-2,b,1)}$ and its two permutations, namely $\xi_{(a,1,b-2)}+\xi_{(a-2,1,b)}$ and $\xi_{(1,a,b-2)}+\xi_{(1,a-2,b)}$.\\
These classes are independent from those obtained from products of three projective spaces, since $5\leq a,b \equiv 1 \mod 4$, and the classes obtained by products are all obtained from $\xi_{(4k-1,1,1)}$ and $\xi_{(a,b,c)}$ with $a,b,c$ all $3 \mod 4$, via permuting subscripts. Similarly we conclude $[M_{a,b,c}]=\xi_{(a-2,b,c)}+\xi_{(a,b-2,c)} +\xi_{(a,b,c-2)}$ which is again independent since $5\leq a,b,c \equiv 1 \mod 4$, and none of the earlier classes included triples with two terms that were both at least $5$ and both $1 \mod 4$, that is, any terms of the form $\xi_{(a,b,c-2)}$ with $5 \leq a,b,c \equiv 1 \mod 4$. Thus all our generators are independent, which completes the proof in this case.\\

We now consider the case $n=4k+3$. The kernel dimensions here are $3,8,15,24$ and so forth. For general $n=4k+3$ the kernel dimension is $\sum_{i=1}^{k}2i+1=2k+k^{2}$, and we again explicitly generate the kernel:\\
 1. We have products $\RP^{a} \times \RP^{b}\times \RP^{c}$, each spin. and $a+b+c=4k+3$. However this immediately implies that, relabeling if needed, $a=1$ and $b,c \equiv 3 \mod 4$, giving a total of $3k$ generators after counting permutations.\\
 2. Products $\RP^{a}\times M_{(b,c)}$ with $5\leq b,c \equiv 1 \mod 4$ which implies $a \equiv 3 \mod 4$. So we have $4A+3+4B+1+4C+1=4k+5$ and therefore $A+B+C=k$, with $B,C \geq 1$ and $A \geq 0$. Counting gives $(k-1)(k-2)/2 + (k-1)$ possible choices of triples $(A,B,C)$. The fundamental class $[\RP^{a}\times M_{(b,c)}]$ maps to $\xi_{(a,b-2,c)}+\xi_{(a,b,c-2)}$. Taking permutations also gives the classes $\xi_{(b-2,a,c)}+\xi_{(b,a,c-2)}$ and $\xi_{(b-2,c,a)}+\xi_{(b,c-2,a)}$ which means we have a total of $3k(k-1)/2$ classes.\\
 Now our total number of generators is $3k + 3k(k-1)/2$. Subtracting the kernel dimension $2k+k^{2}$ from this gives us $k(k-1)/2$. So it suffices to show that there are exactly these many relations between the images of the fundamental classes of these manifolds in the mod 2 homology of $BV(3)$.\\
 As before, the images of $\RP^{1} \times \RP^{b}\times \RP^{c}$ (and permutations) are mutually independent. Observe that $[\RP^{a}\times M_{(b,c)}]=\xi_{(a,b,c-2)}+\xi_{(a,b-2,c)}$. Since $a,c-2 \equiv 3 \mod 4$, the class $\xi_{(a,b,c-2)}$ can only appear as a summand in the fundamental class of one other manifold, namely $[M_{(a+2,b)} \times \RP^{c-2}]=\xi_{(a,b,c-2)}+\xi_{(a+2,b-2,c-2)}$. Adding these two together gives us $\xi_{(a,b-2,c)}+\xi_{(a+2,b-2,c-2)}$, which is exactly $[M_{a+2,b-2,c}]$ by which we mean the permutation of $M_{(a+2,c)} \times \RP^{b-2}$ obtained by putting $\RP^{b-2}$ in the \emph{middle}.\\
 Explicitly we have
 $$[\RP^{a}\times M_{(b,c)}]=\xi_{(a,b,c-2)}+\xi_{(a,b-2,c)}$$
 $$[M_{(a+2,b)} \times \RP^{c-2}]=\xi_{(a,b,c-2)}+\xi_{(a+2,b-2,c-2)}$$
 $$[M_{(a+2,c)} \times \RP^{b}]=[M_{a+2,b-2,c}]=\xi_{(a,b-2,c)}+\xi_{(a+2,b-2,c-2)}$$
 We started with an arbitrary fundamental class $[\RP^{a}\times M_{b,c}]=\xi_{(a,b,c-2)}+\xi_{(a,b-2,c)}$, and saw that $\xi_{(a,b,c-2)}$ (and by symmetry $\xi_{(a,b-2,c)}$) appears as a summand in the fundamental class of exactly one other manifold, and adding these two classes gave us a class we already had. Thus we deduce there is exactly one relation for every three such manifolds, giving a total of $(3k(k-1)/2)/3=k(k-1)/2$ relations, exactly as required.\\
Finally since $n=4k-1$, there is a term in first local cohomology which is detected by inclusion from cyclic subgroups using the eta invariant, and we must check that these classes are independent from the ones we have just given. However, we can detect these in ordinary homology also, and can proceed just like we did in Example $3.5.10$, by using Proposition $3.5.1$ and Lemma $3.4.3$.\\
In the notation of Proposition $3.5.1$, inclusion via $A_*$ gives the three classes $\xi_{(n,0,0)},\xi_{(0,n,0)}$ and $\xi_{(0,0,n)}$. Inclusion using $B_*$ gives the three permutations of the sum $\Sigma_{a+b=n} \xi_{(a,b,0)}$, while the diagonal inclusion $C_*$ simply gives $\Sigma_{a+b+c=n} \xi_{(a,b,c)}$. The presence of summands $\xi_{(a,b,c)}$ with two out of $a,b,c$ even immediately tells us that these independent classes are all independent of the ones we gave above, so that the conclusion follows from Lemma $3.5.9$.
\end{proof}

\chapter{The Dihedral groups}
\section{Preliminaries}
We start by recalling that the dihedral group $D_{2^{N+2}}$ is generated by a rotation of order $2^{N+1}$ and a reflection, and has presentation $<\omega,s| \omega^{2^{N+1}}=s^2=1, s\omega s=\omega^{-1}>$. There are two maximal dihedral subgroups $D'=<s, \omega^2>$ and $D''=<s\omega, \omega^2>$ generated by even and odd conjugacy classes of reflections, and one maximal cyclic subgroup $<\omega>$ of order $2^{N+1}$. There are three non-trivial one-dimensional representations, with Kernel each of the three maximal subgroups, and the remaining irreducible representations are of degree $2$, and restrict from $O(2)$, implying that all representations are real. Here is the character table:\\

\begin{center}
\begin{tabular}{|c|c|c|c|c|c|}
\hline
$ $&1 & 1 & $2(2^N-1)$ & $2^N$ & $2^N$ \\
$\rho$&$\omega_0$ & $\omega_{2^N}$ & $\omega_j (1 \leq j \leq 2^N-1)$ & $s$ & $\omega s$\\
\hline
$1=\rho_0$ & $1$ & $1$ & $1$ & $1$ & $1$\\
$\hat{C_{N+1}}=\hat{\omega}$ & $1$ & $1$ & $1$ & $-1$ & $-1$\\
$\hat{D'}=\hat{s}$ & $1$ & $1$ & $(-1)^j$ & $1$ & $-1$\\
$\hat{D''}=\hat{\omega s}$ & $1$ & $1$ & $(-1)^j$ & $-1$ & $1$\\
$\sigma_k(1 \leq k \leq 2^N-1)$ & $2$ & $2(-1)^k$&$2cos(\frac{2jk\pi}{2^{N+1}})$ &$0$& $0$\\
\hline
\end{tabular}
\end{center}

The cohomology ring $H^*(D_{2^{N+2}}; \Z_2)$, for $n \geq 1$, is given by $\Z_2[\alpha, \beta, \delta]/\alpha \beta+\beta^2$, with all three generators arising as Stiefel-Whitney classes of representations \cite{dihcoh}. We have $\beta=w_1(\hat{s}), \alpha + \beta=w_1(\hat{\omega s})$, and $\delta=w_2(\sigma_1)$ is $w_2$ of a natural two dimensional representation $\sigma_k$ with $k$ odd. Note that $\beta$ restricts to $0$ on $D'$, and to $\alpha$ on $D''$, meaning  $\beta=\beta_N$ depends on $N$, while $\alpha, \delta$ restrict from $O(2)$ and so do not. Note also that since the cohomology ring is generated by Stiefel-Whitney classes of representations, we can calculate the restriction maps to subgroups just by restricting representations, which is what we need to do in Section $4$. The Steenrod square action \cite{dihcoh} is determined by $Sq^1(\delta)=\alpha\delta$.\\
Further, the integral cohomology \cite{dihcoh} is given by $H^*(D_{2^{N+2}}; \Z)=\Z[a,b,c,d]/2a,2b,2c,2^{N+1}d$, with $|a|=|b|=2,|c|=3,|d|=4$.\\

\section{ko calculations}
In this section we will give a brief summary of the calculations in \cite{3} for $ko_*(BD)$ using the local cohomology spectral sequence, in contrast to \cite{mjam}, where we use the Adams spectral sequence.\\
The methods and calculations sketched in this section may be found in their entirety in \cite{3}, and to prove the Gromov-Lawson-Rosenberg conjecture for dihedral groups, we will use only the results in \cite{3} for $ko_*(BD)$.\\
The method used is similar to that for the Klein $4$-group $V(2)$. Again there are exact sequences
$$0 \rightarrow TO \rightarrow ko^{*}(BD) \rightarrow QO \rightarrow 0$$
$$0 \rightarrow T \rightarrow ko^{*}(BD) \rightarrow \overline{QO} \rightarrow 0$$
where $TO$ is detected in ordinary cohomology, and $QO$ and $\overline{QO}$ are the images in periodic real and complex cohomology respectively. It can be shown that the submodule $\tau$ of $\eta$ multiples maps isomorphically to $T/TO$, so that we have an isomorphism
$$T \cong TO \oplus \tau$$
Now the rank of $D=D_{2^{N+2}}$ is always $2$, independent of $N$, but the number of conjugacy classes keeps increasing. Thus it is much easier to understand the local cohomology of $TO$, which is detected in ordinary cohomology independently of $N$, than it is of $QO$, for which the representation theoretic calculations become increasingly complicated as $N$ increases.\\
Indeed, calculation shows that
$$H_I^*(TO)=H_I^{2}(TO) \cong H_I^{2}(TO_{V(2)})$$
so that, before considering differentials at least, the two-column of the local cohomology spectral sequence for all dihedral groups is the same as for $V(2)$.
The Bockstein spectral sequence implies that $\tau$ is in degrees $1,2 \mod 8$, and that the local cohomology is all in cohomological degrees $0$ and $1$. By using a principal ideal again, the local cohomology may be calculated, see section $8.5$ in \cite{3}, but we omit the details.\\

Finally, for $\overline{QO}$, the first local cohomology is all in degrees $0 \mod 4$, and the \emph{orders} of these groups are not too difficult to calculate, just by viewing a generator of the principal ideal as a map from the representation ring to itself, and considering its determinant, see section $8.5$ in \cite{3} again. The result is as follows:
\begin{center}
$$|H_{J_1}(\overline{QO})_{8k}|=\left\{
  \begin{array}{ll}
    $$2^{(2N+12)k+2} \hbox{if } k \geq 0;$$ \\
    $$1  \hbox{if } k \leq-1;$$
  \end{array}
  \right.$$
  \end{center}
\begin{center}
$$|H_{J_1}(\overline{QO})_{8k+4}|=\left\{
  \begin{array}{ll}
    $$2^{(2N+12)k+n+11} \hbox{if } k \geq 0;$$\\
    $$2 \hbox{if } k=-1;$$\\
    $$1 \hbox{if } k \leq-2;$$
  \end{array}
  \right.$$
  \end{center}

The precise structure is much harder to compute, and indeed \cite{3} only provides an answer for $D_8$.\\
We now display the local cohomology spectral sequence for $ko_*(BD_8)$. This is indicative of the general pattern, which we will describe, along with the differentials.\\

\setlength{\unitlength}{1cm}
\begin{picture}(20,19)
\multiput(3,2)(0,0.5){32}%
{\line(1,0){9.8}}
\put(3,2){\line(0,1){16}}
\put(6,2){\line(0,1){16}}
\put(10,2){\line(0,1){16}}
\put(13,2){\line(0,1){16}}
\multiput(11.6,6.6)(0,0.5){3}%
{$0$}
\multiput(11.6,10.6)(0,0.5){3}%
{$0$}
\multiput(11.6,14.6)(0,0.5){3}%
{$0$}
\multiput(11.6,2.1)(0,0.5){4}%
{$0$}
\multiput(11.6,5.6)(0,4){3}%
{$0$}
\multiput(11.6,4.6)(0,0.5){2}%
{$2^{4}$}
\multiput(11.6,8.6)(0,0.5){2}%
{$2^{5}$}
\multiput(11.6,12.6)(0,0.5){2}%
{$2^{5}$}
\multiput(11.6,16.6)(0,0.5){2}%
{$2^{5}$}
\multiput(11.6,4.1)(0,2){7}%
{$\mathbb{Z}$}

\multiput(8,2.6)(0,1){5}%
{$0$}
\multiput(8,9.6)(0,1){3}%
{$0$}
\multiput(8,13.6)(0,1){3}%
{$0$}
\multiput(8,3.1)(0,1){1}%
{$0$}
\put(8,5.1){$0$}
\put(7.9,2.1){$[2]$}
\put(6.7,4.1){$[2]\oplus [2]\oplus[1]$}
\put(6.7,6.1){$[16]^2\oplus [8]\oplus[2]$}
\put(6.7,8.1){$[32]^2\oplus[32]\oplus[2]$}
\put(6.7,10.1){$[2^8]^2\oplus [2^7]\oplus[8]$}
\put(6.7,12.1){$[2^9]^2\oplus [2^9]\oplus[8]$}
\put(6.7,14.1){$[2^{12}]^2\oplus[2^{11}]\oplus[2^5]$}
\put(6.7,16.1){$[2^{13}]^2\oplus[2^{13}]\oplus[2^{5}]$}

\multiput(8,7.1)(0,4){3}%
{$0$}
\multiput(8,7.6)(0,2){5}%
{$0$}
\multiput(8,8.6)(0,4){3}%
{$2$}
\multiput(8,9.1)(0,4){3}%
{$2$}
\put(8,17.6){$\vdots$}
\put(4.5,17.6){$\vdots$}
\put(11.8,17.6){$\vdots$}
\put(4.5,2.1){$2$}
\put(4.5,4.1){$2^2$}
\put(4.5,5.1){$2$}

\put(4.5,6.1){$2^{3}$}
\put(4.5,7.1){$2^2$}
\put(4.5,8.1){$2^{4}$}
\put(4.5,9.1){$2^{3}$}
\put(4.5,10.1){$2^{5}$}
\put(4.5,11.1){$2^{4}$}
\put(4.5,12.1){$2^{6}$}
\put(4.5,13.1){$2^{5}$}
\put(4.5,14.1){$2^{7}$}
\put(4.5,15.1){$2^{6}$}
\put(4.5,16.1){$2^{8}$}
\put(4.5,17.1){$2^{7}$}
\multiput(4.5,4.6)(0,1){13}%
{$0$}
\put(4.5,2.6){$0$}
\put(4.5,3.1){$0$}
\put(4.5,3.6){$0$}
\multiput(6.1,2.1)(0,2){8}%
{\vector(-1,0){1.0}}

\put(6.2,2.1){$d_1$}
\put(6.2,4.1){$d_1$}
\put(6.2,6.1){$d_1$}
\put(6.2,8.1){$d_1$}
\put(6.2,10.1){$d_1$}
\put(6.2,12.1){$d_1$}
\put(6.2,14.1){$d_1$}
\put(6.2,16.1){$d_1$}

\put(13.3,2.1){-4}
\put(13.3,2.6){-3}
\put(13.3,3.1){-2}
\put(13.3,3.6){-1}
\put(13.3,4.1){0}
\put(13.3,4.6){1}
\put(13.3,5.1){2}
\put(13.3,5.6){3}
\put(13.3,6.1){4}
\put(13.3,6.6){5}
\put(13.3,7.1){6}
\put(13.3,7.6){7}
\put(13.3,8.1){8}
\put(13.3,8.6){9}
\put(13.2,9.1){10}
\put(13.2,9.6){11}
\put(13.2,10.1){12}
\put(13.2,10.6){13}
\put(13.2,11.1){14}
\put(13.2,11.6){15}
\put(13.2,12.1){16}
\put(13.2,12.6){17}
\put(13.2,13.1){18}
\put(13.2,13.6){19}
\put(13.2,14.1){20}
\put(13.2,14.6){21}
\put(13.2,15.1){22}
\put(13.2,15.6){23}
\put(13.2,16.1){24}
\put(13.2,16.6){25}
\put(13.2,17.1){26}
\put(13.2,17.6){27}

\put(12.5,18.2){degree(t)}

\put(6.7,1.2){$H^{1}_{J}(\overline{QO}) \oplus H^{1}_{J}(\tau)$}
\put(10.5,1.2){$H^{0}_{J}(\tau) \oplus H^{0}_{I}(\overline{QO})$}
\put(3.8,1.2){$H^{2}_{J}(TO)$}

\put(0.2,0.2){where$[n]:=$ cyclic group of order $n$, $2^{r}$:= elementary abelian group of order $r$.}
\put(10.2,4.6){\line(-6,1){4.6}}
\put(10.3,4.6){$d_2$}
\linethickness{0.5mm}
\put(10,4){\line(1,0){3}}
\put(10,4.5){\line(-1,0){4}}
\put(6,5){\line(-1,0){3}}
\put(10,4){\line(0,1){0.5}}
\put(6,4.5){\line(0,1){0.5}}
\put(3,5){\line(0,1){0.5}}

\end{picture}\\
Next we consider differentials, and proceed much like we did for $V(2)$. Since $ko_*$ is zero in negative degrees, and $ko_0= \Z$, we have that the differentials $d_1:H_{-4}^{1} \rightarrow H_{-4}^{2}$ and $d_1:H_{0}^{1} \rightarrow H_{0}^{2}$ are isomorphisms, and the long differential $d_2:H_{1}^{0} \rightarrow H_{2}^{2}$ is surjective, and thus has kernel $2^3$. Further,
\begin{lem}
The differential $d_1:H_{4k}^{1} \rightarrow H_{4k}^{2}$ has rank $3$ for $k \geq 1$.
\end{lem}
\begin{proof}
Comparison with the complex case \cite{3} gives an upper bound of $3$. To see that is attained, we need only check that $ko_2(BD) \leq 2^4= H^{0}_{J}(\tau)_1$, but this is immediate from the Atiyah-Hirzebruch spectral sequence.
\end{proof}
We now summarise these results, see \cite{3}, corollary $8.5.9$.
\begin{thm}
The $ko-$homology of $BD_{2^{N+2}}$ is given as follows. Here $[2^a]$ means cyclic of order $2^a$, $2^a$ means elementary abelian of rank $a$, and $<2^a>$ means an undetermined group of order $2^a$.\\
\begin{center}
\begin{tabular}{|c|c|c|}
\hline

$n$&$ko_n(BD_{2^{N+2}})$ & $ $\\
\hline

$8k+0$ & $\Z \oplus 2^{a_k}$ & $2k-2^N-2 \leq a_k \leq 2k+2^N$\\

$8k+1$ & $2^{b_k}$ & $2^N-1 \leq b_k \leq 2+2^{N+1}$ \\

$8k+2$ & $2^{c_k}$ & $k+4 \leq c_k \leq 2k+2^{N}+6$\\

$8k+3$ & $<2^{d_k}>$ & $d_k=2^{(2N+12)k+n+8}$\\

$8k+4$ & $\Z \oplus 2^{e_k}$ & $e_k=2k+2$\\

$8k+5$ & $0$ & $ $\\

$8k+6$ & $2^{f_k}$ & $f_k=2k+1$\\
$8k+7$ & $<2^{g_k}>$ & $g_k=2^{(2N+12)(k+1)-1}$\\
\hline
\end{tabular}
\end{center}
\end{thm}
Thus lots more information is required for the precise structure if the group, some of which we give in \cite{mjam}. However, a consequence of the calculations and the spectral sequence we have displayed is that the second local cohomology $H^2$ is the same for all the groups, and the \emph{positions} of the non-trivial $H^0$ and $H^1$ groups are the same. Further, in Proposition $4.4.2$ we will explicitly construct spin manifolds of positive scalar curvature which are detected in ordinary homology. These manifolds will then have independent $ko-$fundamental classes, and the number of classes we obtain this way combined with taking the $H^1$ and $H^2$ parts and the calculations in \cite{3}, then proves the following Proposition.
\begin{prop}
With the $a_k,d_k$ and $g_k$ as in Theorem $4.2.2$, we have the following description of the kernel $Ker(Ap)$:\\
\begin{center}
\begin{tabular}{|c|c|}
\hline

$n$&$Ker(Ap) \subset ko_n(BD_{2^{N+2}})$\\
\hline

$8k+0 \geq 8$ & $\widetilde{ko}_{8k}(BD_{2^{N+2}})=2^{a_k}=2^{2k+1} \oplus \eta(ko_{8k-1}(BD_{2^{N+2}}))$\\

$8k+1$ & $\eta(\widetilde{ko}_{8k}(BD_{2^{N+2}}))$ \\

$8k+2$ & $2^{2k}$\\

$8k+3$ & $<2^{d_k}>=ko_{8k+3}(BD_{2^{N+2}})$\\

$8k+4$ & $2^{2k+2}=\widetilde{ko}_{8k+4}(BD_{2^{N+2}})$\\

$8k+5$ & $0=ko_{8k+5}(BD_{2^{N+2}})$\\

$8k+6$ & $2^{2k+1}=ko_{8k+6}(BD_{2^{N+2}})$\\
$8k+7$ & $<2^{g_k}>=ko_{8k+7}(BD_{2^{N+2}})$\\
\hline
\end{tabular}
\end{center}
\end{prop}
Thus we see that if we realize all of $ko_{8k-1}(BD_{2^{N+2}})$ and $\widetilde{ko}_{8k}(BD_{2^{N+2}})$ by $ko-$fundamental classes of positive scalar curvature spin manifolds, then we will have also verified the Gromov-Lawson-Rosenberg conjecture in dimensions $8k$ and $ 8k+1$. This follows because multiplying by $\eta$ preserves the positive scalar curvature ideal, so that if $ko_{8k-1}(BD_{2^{N+2}})$ and $\widetilde{ko}_{8k}(BD_{2^{N+2}})$ are in $ko_*^+(BD_{2^{N+2}})$, then so are $\eta(ko_{8k-1}(BD_{2^{N+2}}))$ and $\eta(\widetilde{ko}_{8k}(BD_{2^{N+2}}))$, regardless of their orders.\\
Similarly, since $(Ap)$ is injective on the $H^0$ part, we do not need to know how much of the $H^0$ part survives to the $E_{\infty}$ page for the purposes of the Gromov-Lawson-Rosenberg conjecture. Thus, we can verify the conjecture even though the exact values of $a_k,b_k$ and $c_k$ in Theorem $4.2.2$ are not determined.\\

Further, calculations in \cite{3} show that there are no more $d_1$ differentials either. This follows since $d_1$ is induced by a connecting homomorphism from $\overline{QO}$ to $T$. In high enough dimensions, there can be no more $d_2$ differentials either, since otherwise using Theorem $1.5.2$ there would be too few classes with non-trivial index. This combined with the eta invariant calculations of the next section, imply that we can deduce considerably more information about the groups $ko_*(BD)$. Here $C_l$ denotes the cyclic group of order $l$.
\begin{thm}
The $ko-$homology of $BD_{2^{N+2}}$ is given as follows. Here $[2^a]$ means cyclic of order $2^a$, $2^a$ means elementary abelian of rank $a$, and $<2^a>$ means an undetermined group of order $2^a$. In dimensions $0,1$ and $2 \mod 8$ we assume $k \geq n-1$.\\
\begin{center}
\begin{tabular}{|c|c|c|}
\hline

$n$&$ko_n(BD_{2^{N+2}})$ & $ $\\
\hline

$8k+0$ & $\Z \oplus 2^{a_k}=\Z \oplus \eta(ko_{8k-1}(BD_{2^{N+2}})) \oplus 2^{2k+1}$ & $a_k=2k+2^N$\\

$8k+1$ & $2^{b_k}=2^{2^N+3} \oplus \eta(\widetilde{ko}_{8k}(BD_{2^{N+2}}))$ & $b_k=2+2^{N+1}$ \\

$8k+2$ & $2^{c_k}=2^{2^N+3}\oplus 2^{2k}$ & $c_k=2k+2^{N}+3$\\

$8k+3$ & $<2^{d_k}>=2ko_{8k+3}(\RP^{\infty}) \oplus (ko_{8k+3}(BC_{2^{N+1}})/\Z_2)$ & $d_k=2^{(2N+12)k+n+8}$\\

$8k+4$ & $\Z \oplus 2^{e_k}$ & $e_k=2k+2$\\

$8k+5$ & $0$ & $ $\\

$8k+6$ & $2^{f_k}$ & $f_k=2k+1$\\
$8k+7$ & $<2^{g_k}>==2ko_{8k+7}(\RP^{\infty}) \oplus (ko_{8k+7}(BC_{2^{N+1}})/\Z_2)$ & $g_k=2^{(2N+12)(k+1)-1}$\\
\hline
\end{tabular}
\end{center}
\end{thm}

\section{Lens spaces and the periodic part}
 From the character table in the previous section we can see that the three non-trivial one dimensional irreducible representations are denoted $\hat{\omega},\hat{s\omega}$ and $\hat{s}$, with kernels $<\omega>$, $<s \omega,\omega^2>$ and $<s,\omega^2>$ respectively. The two dimensional representations $\sigma_m$ with $1 \leq m \leq 2^{N}-1$ restrict to $\rho_m \oplus \rho_{2^{N+1}-m}$ on the maximal cyclic subgroup $C_{2^{N+1}}=<\omega>$, where as in chapter $2$ we have $\rho_i(\omega)=\omega^i$. Further, all representations of the dihedral groups are real. We saw in chapter two that the eta invariant completely detects the connective k-theory of cyclic groups, and so now we would like to use the methods of chapter 2 to show that inclusion from cyclic subgroups detects all of $Ker(A) \cap Im(p) \subset KO_*(BD)$. \\

We know that $|ko_{8k+3}(BD_{2^{N+2}})|=2^{(2N+12)k+N+8}$, while $|ko_{8k+7}(BD_{2^{N+2}})|=2^{(2N+12)(k+1)-1}$. We then have
\begin{prop}
The images under inclusion of the fundamental classes of the cyclic lens spaces $X^{4i-1}(l=2^{N+1},\overrightarrow{a})$ defined in chapter 2, together with those of the real projective spaces $\RP^{4i-1}_{s}, \RP^{4i-1}_{\omega s}$, span all of $ko_{4i-1}(BD_{2^{N+2}})$.
\end{prop}
Here $\RP^{4i-1}_{s}$ and $\RP^{4i-1}_{\omega s}$ are the $ko_*$ images of the compositions $\RP^{4i-1} \rightarrow B<s> \rightarrow BD_{2^{N+2}}$ and $\RP^{4i-1} \rightarrow B<\omega s> \rightarrow BD_{2^{N+2}}$ respectively, where the first map is the classifying map, and the second is induced by inclusion. We will first give some useful eta invariant calculations, and the proof of the Proposition will then consist of putting these together. The eta invariant calculations for cyclic groups given here will also appear in \cite{cs}.
\begin{lem}
Let $ 1 \leq m \in \Z$ and $l=2^{N+1}$. Then\\
a) $\eta(L=L^{4j-1}(l,(a_1, \cdots,a_{2j})))(\rho_{2m-1}-\rho_0)$ has order $2^{N+2j} \in \R/\Z$.\\
b) Further, $\eta(L)(\rho_m-\rho_0)=\eta(L)(\rho_{l-m}-\rho_0)$.\\
c) More generally, for $K \in \Z$ and $m$ odd, $2\eta(L)(\rho_m-\rho_0)$ has the same order in $\R/\Z$ as $\eta(L)(\rho_m + \rho_{m+2K}-2\rho_0)$.
\end{lem}

\begin{proof}
We introduce the notation $A=(\sum_{k=1}^{2j} a_j)/2$. From Theorem $2.0.4$, we need to consider the sum
$$\sum_{1 \neq \lambda \in C_l} \frac{\lambda^{A}(1-\lambda^{2m-1})}{(1-\lambda^{a_1})\cdots (1-\lambda^{a_{2j}})}$$

It is useful in general to split these sums up. A simple way to give the idea of the proof is to split up the sum:
$$\sum_{1 \neq \lambda \in C_l} \frac{\lambda^{A}(1-\lambda^{2m-1})}{(1-\lambda^{a_1})\cdots (1-\lambda^{a_{2j}})}$$
$$=(\frac{(-1)^{A}(2)}{(2^{a_1})\cdots (2^{a_{2j}})})+(\frac{i^{A}(1-i^{2m-1})}{(1-i^{a_1})\cdots (1-i^{a_{2j}})}+\frac{(-i)^{A}(1-(-i)^{2m-1})}{(1-(-i)^{a_1})\cdots (1-(-i)^{a_{2j}})})+\cdots $$
by bracketing the $\lambda=-1$ term, the two $\lambda= \pm i$ terms, the four terms from $\lambda=\omega, \omega^3, \omega^5, \omega^7$, where $\omega=(1+i)/\sqrt{2}$ and so on, and then observing that each successive bracket has strictly smaller order $\in \R/\Z$ than the one preceding it.\\
We make an eta invariant calculation to prove a) for $L^{4j-1}(l=2^{N+1},(1, \cdots,1))$ and $\rho_1$, and the general case follows analogously. We use Theorem $2.0.4$ and Lemma $2.1.2$ again:

$$\eta(L^{4j-1}(l=2^{N+1},(1, \cdots,1)))(\rho_{1}-\rho_0)=l^{-1}\sum_{1 \neq \lambda \in C_l}-\frac{\lambda^j(1-\lambda)} {(1-\lambda)^{2j}}$$
$$=l^{-1}\sum_{1 \neq \lambda \in C_l}-\frac{\lambda^j} {(1-\lambda)^{2j-1}}$$
$$=2^{-(N+2j)}+l^{-1}\sum_{\pm 1 \neq \lambda \in C_l}-\frac{\lambda^j(1-\lambda)} {(1-\lambda)^{2j}}$$
But now we claim the sum $l^{-1}\sum_{\pm 1 \neq \lambda \in C_l}-\frac{\lambda^j(1-\lambda)} {(1-\lambda)^{2j}}$ can not have order greater than $2^{N+j}$. This follows since if we write $C_l=<\omega>$, then $$l^{-1}\sum_{\pm 1 \neq \lambda \in C_l}-\frac{\lambda^j(1-\lambda)} {(1-\lambda)^{2j}}=l^{-1}\sum_{\pm 1 \neq \lambda \in \{\omega, \cdots ,\omega^{l/2-1}\}}-2Re(\frac{\lambda^j(1-\lambda)} {(1-\lambda)^{2j}})$$
Where $Re$ denotes the real part. Now argue by induction on $j$, claiming no term in the summand has order $2^{N+2j}$,starting at $j=1$, and say $\lambda=x+iy \in \Comp,|\lambda|=1$. Then we are considering $$\lambda/(1-\lambda)=1/(1-\lambda)-1=-\frac{(1-x)+iy} {(1-x)^2+y^2}-1$$
which has real part $-1/2$ since $x^2+y^2=1$, so that there is no term of order $2^{N+2}$ or more (remember $\lambda=-1$ is not being considered). Thus assuming the conclusion for $j=m$, consider if $j=m+1$; Then $$\lambda^{m+1}/(1-\lambda)^{2m+1}=(\lambda^{m}/(1-\lambda)^{2m-1})(\lambda/(1-\lambda))^2$$
But now the norm of $(\lambda/(1-\lambda))^2$ is the same as $|(1-1/\lambda)^2|$ which is $1/\sqrt{2-2x} \geq 1/2$, so that the total order can be doubled at most, meaning there is again no term of order $2^{N+m+2}$, since there was no term of order $2^{N+m+1}$ by the induction hypothesis. The argument for $\rho_{2m-1}$ and arbitrary
$L^{4j-1}(l=2^{N+1},(a_1, \cdots,a_{2k}))$ goes through the same way, by just separating out what we get for $\lambda=-1$ from the sum.\\

This also proves part c) the Lemma, since we can consider the above sum for the representations $\rho_m-\rho_0$ and $\rho_{m+2K}-\rho_0$, with $m$ odd. Then in each sum we can consider the $\lambda \neq \pm 1$ summands, and their sum must now have order strictly less than $2^{N+j-1}.$ Thus the order is the order of the sum of the two $\lambda=-1$ terms in $\R/\Z$ which is $2^{N+2j-1}$.\\

For part b), we consider any of our lens spaces $L^{4j-1}(l;a_1,\cdots ,a_{2i})$ and use the additivity of the eta invariant:
$$\eta(L^{4j-1}(1;a_1,\cdots ,a_{2i}))(\rho_m-\rho_0)-\eta(L^{4j-1}(l;a_1,\cdots ,a_{2i}))(\rho_{l-m}-\rho_0)=\eta(L^{4j-1}(l;a_1,\cdots ,a_{2i}))(\rho_m-\rho_{l-m})$$
$$=\sum_{1 \neq \lambda \in C_l} \frac{\lambda^{A}(\lambda^m-\lambda^{-m})}{(1-\lambda^{a_1})\cdots (1-\lambda^{a_{2i}})}$$
Now consider the sum of the two terms we get for $\lambda$ and $\overline{\lambda}=\lambda^{-1}$:
$$(\lambda^m-\lambda^{-m}) (\frac{\lambda^{A}}{(1-\lambda^{a_1})\cdots (1-\lambda^{a_{2i}})}-\frac{\lambda^{-A}}{(1-\lambda^{-a_1})\cdots (1-\lambda^{-a_{2i}})})$$
Now expand out over a common denominator, and write $K$ for $\lambda^m-\lambda^{-m}$ over the product of the denominators:
$$=K\lambda^{(\sum_{1 \leq k \leq 2i}a_k)/2}(1-\sum_{j\in \{1, \cdots, 2i\}}\lambda^{-a_j}+\sum_{j \neq k\in \{1, \cdots, 2i\}}\lambda^{-a_j}\lambda^{-a_k}- \cdots +\lambda^{-a_1} \cdots \lambda^{-a_{2i}})$$
$$-K\lambda^{(-\sum_{1 \leq k \leq 2i} a_k)/2}(1-\sum_{j\in \{1, \cdots, 2i\}}\lambda^{a_j}+\sum_{j \neq k\in \{1, \cdots, 2i\}}\lambda^{a_j}\lambda^{a_k}- \cdots +\lambda^{a_1} \cdots \lambda^{a_{2i}})$$
Now note that
$$\lambda^{(\sum_{1 \leq k \leq 2i} a_k)/2}=\lambda^{(-\sum_{1 \leq k \leq 2i} a_k)/2}\lambda^{a_1} \cdots \lambda^{a_{2i}}$$
and similarly,\\
$$\lambda^{(\sum_{1 \leq k \leq 2i}a_k)/2}(\sum_{j\in \{1, \cdots, 2i\}}\lambda^{-a_j})=\lambda^{(-\sum_{1 \leq k \leq 2i}a_k)/2}(\sum_{j\in \{1, \cdots, 2i\}}\lambda^{a_1 \cdots a_{2i}}\lambda^{-a_j})$$
and so on. Thus, expanding the brackets out immediately shows that the first term on the top line cancels the last term on the bottom, the second term the second last, and so on, giving zero, as required.
\end{proof}
The idea will be to mimick the proof for cyclic $2-$groups $C_l$ in \cite{2} and Chapter 2, which was to say there is a surjective map with kernel of order $2l$ from the span of the eta invariants in $\R/\Z$ of all the cyclic lens spaces $L^n$, to the range in $\R/\Z$ of all the $L^{n-4}$. Induction and extra factors due to real representations then completed the proof. By naturality there will still be such a surjective map for the image under inclusion in the dihedral groups, and we will see that the kernel has the same order. By Bott periodicity it suffices to see this for $n=7,11$. These and more general calculations may also be found in \cite{cs}.
\begin{lem}
i)We have that $ko_7(BC_{2^{N+1}})=[2^{N+4}] \oplus [2^N]$, generated by the $ko-$fundamental classes of the lens spaces $L^7(1,1,1,1)$ and $L^7(1,1,1,3)$.\\
The map
$$[L^7(l;1,1,1,1)] \rightarrow (\eta(L^7)(l;1,1,1,1)(\rho_1-\rho_0),0);$$
$$[L^7(l;1,1,1,3)]\rightarrow (\eta(L^7)(l;1,1,1,3)(\rho_1-\rho_0),\eta(L^7)(l;1,1,1,3)(2\rho_2-2\rho_0))$$
is a homomorphism, which detects the entire group.\\
ii) Similarly, $ko_{11}(BC_{8})=[2^{8}] \oplus [2^{3}] \oplus [2]^2$, while for $N>2$ we have $ko_{11}(BC_{2^{N+1}})=[2^{N+6}] \oplus [2^{N+1}] \oplus [2^{N-1}] \oplus [2]$. In both cases, the eta invariants of the $11-$dimensional lens spaces $L^{11}(l;1,1,1,1,1,1), L^{11}(l;1,1,1,1,1,3)$ and $L^{11}(l;1,1,1,1,1,5)$ span a subspace of the form $[2^{N+6}] \oplus [2^{N+1}] \oplus [2^{N-1}]$ in $(\R/\Z)^3$ detected by the eta invariant homomorphisms $\eta(\rho_1-\rho_0)$, $\eta(\rho_2+\rho_{-2}-2\rho_0)$ and $\eta(\rho_5+\rho_{-5}-\rho_1-\rho_{-1}-\rho_4-\rho_{-4}+2\rho_0)$.
\end{lem}

\begin{proof}
We prove the first part of the Lemma, and the second is analogous. We use the method of Lemma $4.3.2$ to calculate the orders of some eta invariants.\\
Indeed, separating the $\lambda=-1$ term again for each lens space, we have that:
$$\eta(L^7(1,1,1,1))(\rho_0-\rho_1)=+1/2^{N+4}+ x =X/2^{N+4}$$
$$\eta(L^7(1,1,1,3))(\rho_0-\rho_1)=-1/2^{N+4}+ y= Y/2^{N+4} $$
Here $x,y \in \R/\Z$ are terms of order no more that $2^{N+2}$ by the previous Lemma, so that $X,Y \in \Z$ are odd.\\
In particular, it follows that $\eta(L^7(1,1,1,1)-L^7(1,1,1,3))(\rho_0-\rho_1)$ has order $2^{N+3} \in \R/\Z$. Thus $X-Y \equiv 2 \mod 4$, so without loss of generality, we can assume $ X \equiv 1 \mod 4$ and $Y \equiv 3 \mod 4$.\\
Now we make an analogous calculation for $\rho_0-\rho_2$. In this case, the sum of the $\lambda= \pm i$ terms will have the greatest order, and we get
$$\eta(L^7(1,1,1,1))(\rho_0-\rho_2)=-1/2^{N+1}(2/(1-i)^4+2/(1+i)^4)+ a$$
$$=1/2^{N+1}(2/4+2/4)+ a=A/2^{N+1}$$

$$\eta(L^7(1,1,1,3))(\rho_0-\rho_2)=1/2^{N+1}(-2i/(1-i)^3(1+i)+2i/(1+i)^3(1-i))+ b$$
$$=1/2^{N+1}+ b=B/2^{N+1}$$
Thus in this case $\eta(L^7(1,1,1,1)-L^7(1,1,1,3))(\rho_0-\rho_2)$ has order at most $2^{N-1} \in \R/\Z$, meaning $A \equiv B \mod 4$, with $A,B$ odd.\\
We can now consider the resulting $2 \times 2$ matrix of eta invariants:

\begin{displaymath}
M=\left(\begin{array}{ccc}
\eta(L^7(1,1,1,1))(\rho_0-\rho_1) & \eta(L^7(1,1,1,1))(\rho_0-\rho_2) \\
\eta(L^7(1,1,1,3))(\rho_0-\rho_1) & \eta(L^7(1,1,1,3))(\rho_0-\rho_2) \\
\end{array} \right)
\end{displaymath}

\begin{displaymath}
=\left(\begin{array}{ccc}
X/2^{N+4} & A/2^{N+1} \\
Y/2^{N+4} & B/2^{N+1} \\
\end{array} \right)
\end{displaymath}
The determinant is $(XB-YA)/2^{2N+5}$ which has order $2^{2N+4} \in \R/\Z$ as required, since $XB-YA \equiv X-Y \mod 4 \equiv 2 \mod 4$. The result now follows since $(XB-YA)/2^{N+1}$ has order $2^N \in \R/\Z$. Thus we get: $$2^N((XB)\eta(L^7(1,1,1,1))(\rho_0-\rho_2)-(YA)\eta(L^7(1,1,1,3))(\rho_0-\rho_2))=0=$$
$$=2^N(\eta(L^7(1,1,1,3))(2\rho_0-2\rho_2))$$
and performing row operations on $M$ implies $ko_7(BC_{2^{N+1}})$ is generated by $[L^7(1,1,1,1)],[L^7(1,1,1,3)]$, with the relation $2^N(X[L^7(1,1,1,3)]-Y[L^7(1,1,1,1)])=0$. Both $[L^7(1,1,1,1)]$ and $[L^7(1,1,1,3)]$ have order $2^{N+4}$, so the conclusion follows from Lemma $2.2.2$.\\
The second part of the Lemma can be checked analogously, by calculating triples of eta invariants of the form
$$(\eta(L^{11})(\rho_0-\rho_1),\eta(L^{11})(\rho_0-\rho_2),\eta(L^{11})(\rho_0-\rho_4))$$
and then arguing in the same way as above (compare Lemma $2.1.5b$).
\end{proof}

The essence of the argument here is that $ko_{4j-1}(BC_l)$, with $l=2^{N+1}$, is spanned by $ko-$fundamental classes of cyclic lens spaces. By Lemma $4.3.2a)$ the $ko-$fundamental classes of all of the lens spaces $L^{4j-1}(l;(a_1, \cdots,a_{2j}))$ have the maximal order of $2^{N+2j}$. Thus generators for $ko_{4j-1}(BC_l)$ must be given by any fundamental class $[L^{4j-1}(l;(a_1, \cdots,a_{2j}))]$, together with some other classes that are formal differences of lens spaces. The orders of these formal differences are detected by the eta invariant of a suitable lens space $L$ with respect to some representation $\rho$ (compare Lemma $2.1.5b$), and $\rho$ can be chosen to restrict directly from $D_{2^{N+2}}$.\\

We can now use these results, along with chapter 2, to prove Proposition $4.3.1$. To give an idea of the method, we first consider $D_8$ as a separate example.
\begin{ex}
$D_8$
\end{ex}
We know $D_8$ has presentation $<\omega,s|\omega^4=s^2=1, s\omega s=\omega^3>$, and there are three one dimensional representations $\hat{\omega}, \hat{s}, \hat{\omega s}$ along with a natural two dimensional representation $\sigma$. We can consider $ko-$fundamental classes of real projective spaces via inclusion. Thus we take $\RP^{4j-1}_{s} \rightarrow B<s> \hookrightarrow BD_8$, and $\RP^{4j-1}_{\omega s} \rightarrow B<\omega s> \hookrightarrow BD_8$. We can also consider the lens spaces $L^{4j-1}(4;1, \cdots ,1,1)$ and $L^{4j-1}(4;1, \cdots ,1,3)$ that have natural maps into $B <\omega>$, which in turn includes into $BD_8$. The eta invariants can be read off directly from the calculations in the proof of Lemma $2.2.1$, just by restricting representations. We can define the following $4-$tuples of eta invariants:
$$\overrightarrow{\eta}(M):=(\eta(M)(1-\hat{\omega}),\eta(M)(1-\hat{s}),\eta(M)(1-\hat{\omega s}),\eta(M)(2-\hat{\sigma})+\eta(M)(1-\hat{\omega}))$$
Now from section $2.2$ we get
$$\overrightarrow{\eta}(\RP^{4j-1}_{s})=(2^{-2j},2^{-2j},0,0)$$
$$\overrightarrow{\eta}(\RP^{4j-1}_{\omega s})=(2^{-2j},0,2^{-2j},0)$$
$$\overrightarrow{\eta}(L^{4j-1}(4;1, \cdots ,1,1))=(0,(-1/2)^j,(-1/2)^j,2(1/2(-1/2)^{j+1}+(-1)^j/2^{2j+1}))$$
$$\overrightarrow{\eta}(L^{4j-1}(4;1, \cdots ,1,3))=(0,(-1/2)^j,(-1/2)^j,2(1/2(-1/2)^{j+1}+(-1)^{j+1}/2^{2j+1}))$$
Now note that $\overrightarrow{\eta}(L^{4j-1}(4;1, \cdots ,1,3))-\overrightarrow{\eta}(L^{4j-1}(4;1, \cdots ,1,1))=(0,0,0,1/2^{2j-1})$.\\
Now suppose $n=8m+3$. Then all the eta invariants take values in $\R/2\Z$ since all the representations are real. Then from above the order of the subgroup spanned in $(\R/2\Z)^4$ is bounded below by $2^{4m+2}$ times the order of the subgroup spanned in $(\R/2\Z)^3$ by the three vectors obtained by removing the last entry of $\overrightarrow{\eta}(M)$, with $M$ one of $\RP^{4j-1}_{s}, \RP^{4j-1}_{\omega s}$ or $L^{4j-1}(4;1, \cdots ,1,1)$. We display these as a $3 \times 3$ matrix \\

\begin{displaymath}
=\left(\begin{array}{ccc}
2^{-4m-2} & 2^{-4m-2} & 0 \\
2^{-4m-2} & 0 & 2^{-4m-2} \\
0 & 2^{-2m-1} & 2^{-2m-1} \\
\end{array} \right)
\end{displaymath}
Now the rows are pairwise independent, and $2^{4m+2}$ times the sum of the first two rows is $2^{2m+1}$ times the third, and equals $(0,1,1)$ which has order $2$ in $\R/2\Z$. Thus the order of the subgroup they span is the product of the orders of the individual vectors, divided by $2$, which equals $2^{4m+3}2^{4m+3}2^{2m+2}2^{-1}=2^{10m+7}$, which when combined with the $2^{4m+2}$ we already had, gives us $2^{14m+9}$, which is the same as the order of $ko_{8k+3}(BD_8)$ by Theorem $4.2.2$.\\

If $n=8m+7$ the eta invariants all lie in $\R/\Z$, and we deduce analogously that the order of the subgroup spanned in $(\R/\Z)^4$ is given by $2^{4m+3}$ times the corresponding three vectors, this time in $(\R/\Z)^3$. We again display these as a matrix:\\

\begin{displaymath}
=\left(\begin{array}{ccc}
2^{-4m-4} & 2^{-4m-4} & 0 \\
2^{-4m-4} & 0 & 2^{-4m-4} \\
0 & 2^{-2m-2} & 2^{-2m-2} \\
\end{array} \right)
\end{displaymath}
Again the rows are pairwise independent, and $2^{4m+3}$ times the sum of the first two rows is $2^{2m+1}$ times the third in $\R/\Z$, but the third row is given by $$(\eta(L^{8m+7}(4;1, \cdots ,1,1))(1-\hat{\omega})=0,\eta(L^{8m+7}(4;1, \cdots ,1,1))(1-\hat{s}),\eta(L^{8m+7}(4;1, \cdots ,1,1))(1-\hat{\omega s}))$$
Thus twice it is $(0,\eta(L^{8m+7}(4;1, \cdots ,1,1))(2-2\hat{s}),\eta(L^{8m+7}(4;1, \cdots ,1,1))(2-2\hat{\omega s}))$, and $2\hat{s}$ and $2\hat{\omega s}$ are both quaternion representations which restrict to twice the real representation of $C_4=<\omega>$, which is also a quaternion representation. Thus $\eta(L^{8m+7}(4;1, \cdots ,1,1))(1-\hat{s})$ has the same order as $\eta(L^{8m+7}(4;1, \cdots ,1,1))(2-2\hat{s})$, since the latter takes values in $\R/2\Z$ in dimensions $7 \mod 8$. It then follows that the order of the subgroup spanned is simply the product of the orders of the three vectors, and thus multiplying by the $2^{4m+3}$ factor we already had gives a total of $2^{4m+3}2^{4m+4}2^{4m+4}2^{2m+1}=2^{14m+13}$ which is the same as the order of $ko_{8k+7}(BD_8)$ by Theorem $4.2.2$.\\

The linear relation we found here will in fact be the same for all the dihedral groups. In this example, we exploited the explicit calculations we had already made for $C_4$. Note also that the $ko-$fundamental classes of the real projective spaces are completely detected by the eta invariants with respect to virtual representations $1-\rho$, where $\rho$ is a one dimensional representation. This is true for all dihedral groups, and it follows directly that $2 ko_*(\RP^{\infty})$ is a summand of $ko_{4k-1}(BD_{2^{N+2}})$ for all $n$ (compare \cite{bay}, \cite{mp}). Indeed, it is easiest to split this part off, and deal with the inclusions of lens spaces from the cyclic $C_{2^{N+1}}$ subgroup, before combining the two.\\

For the higher groups, we will instead argue similarly to Chapter 2 and the proof in \cite{2} for cyclic groups. That is, we know there is a surjective map $\delta$ from ${\L}_n(BC_l)$ to ${\L}_{n-4}(BC_l)$, where ${\L}_*$ denotes the span of eta invariants in $\R/\Z$ of lens spaces. The same is thus true for their images ${\L}_*(BD_{2l})$, and an eta invariant calculation tells us that the order of the kernel is the same as the order of the kernel of $\delta$. We can then induct and use real representations to see which eta invariants lie in $\R/2\Z$. This will give us the order of the subgroup spanned by the $ko-$fundamental classes of cyclic lens spaces, and we then check for relations with the $ko-$fundamental classes of real projective spaces. Combining everything realises all of $ko_{4i-1}(BD_{2^{N+2}})$.\\

We can now prove Proposition $4.3.1$.\\

\begin{proof} (Of Proposition $4.3.1$)\\
Let $l=2^{N+1}$. We recall from section $2.1$ that ${\L}_n(BC_l)$ was the subgroup spanned in $(\R/\Z)^{l-1}$ by the eta invariants of cyclic lens spaces, and we analogously let ${\L}_n(BD_{2^{N+2}}) \subset (\R/\Z)^{2^N+2}$ be the subgroup spanned by the eta invariants, with respect to representations of the dihedral group, of these cyclic lens spaces, viewed as $ko-$fundamental classes via the inclusion from the cyclic subgroup $C_{2^{N+1}}$. Then using naturality, we have the following commutative diagram. The horizontal maps are surjective by definition, and the vertical maps are surjective because of the proof for the cyclic groups $C_l$ (see chapter 2 and \cite{2}).
$$
\xymatrix{
{\L}_n(BC_l) \ar[r] \ar [d]^{\delta} & {\L}_n(BD_{2^{N+2}}) \ar[d]^{\overline{\delta}}\\
{\L}_{n-4}(BC_l) \ar[r] & {\L}_{n-4}(BD_{2^{N+2}})
}
$$
Note that for the cyclic groups, the map $\delta$ was induced by defining
$$\delta(M)=(\eta(M)(\sigma(\rho_1-\rho_0)), \cdots ,\eta(M)(\sigma(\rho_{l-1}-\rho_0)))$$
where $M$ is a cyclic lens space and $\sigma=\rho_{-3}(\rho_3-\rho_0)^2=\rho_{-3}+\rho_{3}-2\rho_{0}$. Thus the representation $\sigma_3-2\rho_0$ of the dihedral group $D_{2^{N+2}}$ restricts to $\sigma$, and so $\overline{\delta}$ can be defined by directly restricting representations. Thus all our calculations will be of eta invariants of cyclic lens spaces with respect to restricted representations.\\
We know that $\delta$ has kernel of order at least $2l$ (see \cite{2} and Lemma 2.1.6), and we wish to see the same for $\overline{\delta}$. Multiplying up by the Bott element means it suffices to see this for $n=7,11$.\\
The idea is to use a combination of the above method, used in \cite{2}, and direct restriction. Note that since any representation $\rho \in R_0(D)$ of the dihedral group is real, for any of our cyclic lens spaces $L^{8k+3}$ we have that $\eta(L^{8k+3})(\rho) \in \R/2\Z$. Further for $\rho=\sigma_i$ restricting representations directly gives
$$\eta(L^{8k+3})(2-\sigma_i)=\eta(L^{8k+3})(1-\rho_i)+\eta(L^{8k+3})(1-\rho_{-i})=2\eta(L^{8k+3})(1-\rho_i)$$
Thus the extra factor of two is cancelled out by having a real representation, and thus the image under inclusion in $ko_{8k+3}(BD_{2^{N+2}})$ is all of $ko_{8k+3}(BC_{2^{N+1}})$, and as we'll see below, combining with inclusion from real projective spaces coming from $\Z_2$ subgroups spans all of $ko_{8k+3}(BD_{2^{N+2}})$.\\
However, in dimensions $7 \mod 8$, it is easiest to mimick the cyclic group argument, and try to find the order of the kernel of $\overline{\delta}:{\L}_n(BD_{2^{N+2}})\rightarrow {\L}_{n-4}(BD_{2^{N+2}})$. Starting in dimension $3$, we know from Lemma $4.3.2$ that, letting $i$ denote inclusion from the cyclic subgroup of order $l=2^{N+1}$,
$$\eta(i_*(L^{3}(l,(1, \cdots,1))))(\sigma_1-\rho_0)=
\eta(L^{3}(l,(1, \cdots,1)))(i^*(\sigma_1-\rho_0))$$
$$=\eta(L^{3}(l,(1, \cdots,1)))(\rho_1+\rho_{l-1}-2\rho_0)=
2\eta(L^{3}(l,(1, \cdots,1)))(\rho_1-\rho_0)$$
which has order $l=2l/2=2^{N+1} \in \R/\Z$. Note also that $\delta$ is only used to detect the orders in $\R/\Z$, and if $n \equiv 3 \mod 8$ and we have a real representation, or if $n \equiv 7 \mod 8$ and the representation is quaternion, then the eta invariant takes values in $\R/2\Z$. This will be taken this into consideration later.\\
Further, Lemma $4.3.3$ shows that $ko_{7}(BC_{2^{N+1}})=[2^{N+4}] \oplus [2^{N}]$, and may be spanned by the $ko-$fundamental classes of the manifolds $L^7(l;1,1,1,1), L^7(l;1,1,1,3)$. In particular, the homomorphism
$$[L^7(l;1,1,1,1)] \rightarrow (\eta(L^7)(l;1,1,1,1)(\rho_1-\rho_0),0);$$
$$[L^7(l;1,1,1,3)]\rightarrow (\eta(L^7)(l;1,1,1,3)(\rho_1-\rho_0)\eta(L^7)(l;1,1,1,3)(2\rho_2-2\rho_0))$$
$$=(0,\eta(L^7)(l;1,1,1,3)(\rho_2 +\rho_{-2}-2\rho_0))$$
detects the entire group. The case of $D_8$ is slightly exceptional since $\rho_2$ is the real representation of $C_4$, and $\hat{s}$ and $\hat{\omega s}$ both restrict to $\rho_{2^N}$ which is just $\rho_2$ here, while for the higher groups $\sigma_2$ restricts to $\rho_2 \oplus \rho_{-2}$, and $\sigma_1$ always restricts to $\rho_1 \oplus \rho_1$, meaning in particular that we span $[2^{N+3}] \oplus [2^{N}] \subset ko_{8k+7}(BD_{2^{N+2}})$, which has order $2^{2N+3}$. Thus the order of the kernel of $\overline{\delta}:{\L}_{8k+7}(BD_{2^{N+2}})\rightarrow $${\L}_{8k+3}(BD_{2^{N+2}})$ is at least $2^{2N+3}2^{-N-1}=2^{N+2}=2l$.\\
The same argument goes through in dimension $11$. Calculations in \cite{3} show $ko_{11}(BC_4)=[2^7] \oplus [2^{3}] $, and since this maps monomorphically into $ko_{11}(BD _8)$, a subspace of order at least $2^8$ is spanned in $(\R/\Z)^2$. The order is divided by four only with respect to eta invariants in $(\R/\Z)^2$ here, and the map induced by inclusion is still monomorphic in $ko_*$, because the remaining two factors of two come from the representations of the dihedral group being real, and thus the eta invariants are actually in $(\R/2\Z)^2$ . Thus the map $\overline{\delta}$ has kernel of order at least $8$.\\
For the cases $N>1$ we still know from \cite{2} that that $|ko_{11}(BC_{2^{N+1}})|=2^{(3N+7)}$. Further, we can use the second part of Lemma $4.3.3$ to deduce that the eta invariants of the $ko-$fundamental classes of $L^{11}(l;1,1,1,1,1,1)$, $L^{11}(l;1,1,1,1,1,3)$ and $L^{11}(l;1,1,1,1,1,5)$ span a subspace of the form $[2^{N+5}] \oplus [2^{N+1}] \oplus [2^{N-1}]$ in $(\R/\Z)^3$, detected by the eta invariant homomorphisms $\eta(\rho_1-\rho_0)$, $\eta(\rho_2+\rho_{-2}-2\rho_0)$ and $\eta(\rho_5+\rho_{-5}-\rho_1-\rho_{-1}-\rho_4-\rho_{-4}+2\rho_0)$. Since $\sigma_i$ restricts to $\rho_i \oplus \rho_{-i}$, it follows that a subspace of order at least $2^{3N+7}2^{-2}$ is spanned in ${\L}_{11}(BD_{2^{N+2}})$. Thus we again have that the kernel of $\overline{\delta}$ has order at least $2^{3N+5}2^{-2N-3}=2^{N+2}=2l$.\\

Mimicking the cyclic group argument further thus means we have realized a subspace of order at least $l(2l)^{2m+1}$ in $ko_{8m+7}(BD_{2^{N+2}})$, and $(2l)^{2m+2}$ in $ko_{8m+3}(BD_{2^{N+2}})$, where we now use the fact that we have real representations.\\
We now combine this with the real projective spaces $\RP^{4j+3}_{s}, \RP^{4j+3}_{\omega s}$ we already have via inclusions from $\Z_2$ subgroups.\\

For $\RP^{4m+3}_{s}, \RP^{4m+3}_{\omega s}$, it is easy to see directly that
$$\eta(\RP^{4m+3}_{s})(\hat{\omega s}-\rho_0)=\eta(\RP^{4m+3}_{s})(\hat{\omega }-\rho_0)=2^{-2j-2};\eta(\RP^{4m+3}_{s})(\hat{s}-\rho_0)=0;$$
$$\eta(\RP^{4m+3}_{\omega s})(\hat{s}-\rho_0)=\eta(\RP^{4m+3}_{s})(\hat{\omega }-\rho_0)=2^{-2j-2};\eta(\RP^{4m+3}_{\omega s})(\hat{\omega s}-\rho_0)=0$$
It follows that these two manifolds span a subspace of the form $[2^{4m+3}]^2$ in $ko_{8m+3}(BD_{2^{N+2}})$, and of the form $[2^{4m+4}]^2$ in $ko_{8m+7}(BD_{2^{N+2}})$.
Further, performing elementary row operations tells us that any linear dependence in $\R/\Z$ with the ${\L}_n(BD_{2^{N+2}})$ coming from the cyclic subgroup must come from what is detected by the eta invariant with respect to one dimensional representations $\hat{\omega s},\hat{s},\hat{\omega }$; that is, between the three vectors $(\eta(\RP^{4m+3}_{\omega s})(\rho-\rho_0),\eta(\RP^{4j+3}_{s})(\rho-\rho_0),\eta(L^n(BC_l)(\rho-\rho_0))$, with $\rho \in \{\hat{\omega s},\hat{s},\hat{\omega }\}$, where by $\eta(L^n(BC_l)(\rho-\rho_0)$, we mean the set of eta invariants with respect to these representations of any of the $C_l$ lens spaces $L^n$ that we use. This means there can be at most one relation of linear dependence, and it is easy to now see that there will be a linear relation between the images of the $\RP^{8j+3}_{s}, \RP^{8j+3}_{\omega s}$ and what is detected by taking the eta invariant with respect to the \emph{real} representation $\rho_{l/2}$ of the cyclic subgroup (the latter is non-trivial by Lemma $2.1.6$):\\
Start in $n=3$. Then $ko_3(BC_l)=[2l] \oplus [2]$, generated by $L(l;1,1)$ (or any other lens space) and $Y^3=L(l;1,1)-3L(l;1,3)$. Then
$$(\eta(\RP^{3}_{\omega s})(\hat{\omega}-\rho_0),\eta(\RP^{3}_{s})(\hat{\omega}-\rho_0),\eta(Y^3)(\hat{\omega}-\rho_0))
=(1/4,1/4,0) \in (\R/2\Z)^3$$
$$(\eta(\RP^{3}_{\omega s})(\hat{\omega s}-\rho_0),\eta(\RP^{3}_{s})(\hat{\omega s}-\rho_0),\eta(Y^3)(\hat{\omega s}-\rho_0))
=(0,1/4,1) \in (\R/2\Z)^3$$
$$(\eta(\RP^{3}_{\omega s})(\hat{s}-\rho_0),\eta(\RP^{3}_{s})(\hat{s}-\rho_0),\eta(Y^3)(\hat{s}-\rho_0))
=(1/4,0,1) \in (\R/2\Z)^3$$
Thus we have three vectors of order $8$ in $(\R/2\Z)^3$, and there is one relation, namely $(1/4,0,1)+(0,1/4,1)=(1/4,1/4,0)$. Note that $\hat{\omega s}$ and $\hat{s}$ both restrict to the real representation on $C_l$. Further, by Lemma $2.1.3$ and Lemma $2.1.6$ we can multiply up by the Bott element to deduce that we will always have such a relation for every $n=8m+3$.\\
So, in total, we have realized a subspace of $ko_{8m+3}(BD_{2^{N+2}})$ of order at least $2(2^{N+2})^{2m+1}(2^{4m+3})^2 2^{-1}$ by positive scalar curvature manifolds. Thus the exponent of two is $1+(N+2)(2m+1)+8m+6-1=(2N+12)m+N+8$ as required.\\
For $ko_{8m+7}(BD_{2^{N+2}})$, we have a subspace of order $2^{N+1}(2^{N+2})(2m+1)(2^{4m+4})^2 2^{-1}$, and calculating the exponent of two again gives $N+1+(N+2)(2m+1)+8m+8-1=((2N+12)(m+1)-1)-1$ meaning we are a factor of two short.\\
However, the relation we found just by restricting to the real representation isn't actually present in $8m+7$, just like in Example $4.3.4$ for $D_8$. The reason for this is that for any lens space $L^7$ we have
$$\eta(L^7)(i^*(\hat{\omega s}-\rho_0))+\eta(L^7)(i^*(\hat{s}-\rho_0))=\eta(L^7)(2\rho_{l/2}-2\rho_0)=\eta(L^7)(i^*(2\hat{\omega s}-2\rho_0))$$
and, as a sum of a one dimensional irreducible real representation with itself, $2\hat{\omega s}$ is in fact quaternion, as is $2\rho_{l/2}$. So, since we are in dimensions $7 \mod 8$, we have that $\eta(L^7)(i^*(\hat{\omega s}-\rho_0)) \in \R/\Z$ has the same order as $\eta(L^7)(i^*(2\hat{\omega s}-2\rho_0)) \in \R/2\Z$. The relation in $8m+3$ was essentially that we could halve the order of whatever the real representation detected for $C_l$, and this observation shows that we can not do this in $n=8m+7$.
\end{proof}
We have from Theorem $4.2.1$ that the order $|ko_{8k+3}(BD_{2^{N+2}})|=2^{(2N+12)k+N+8}=|[2^{4k+3}]^2\oplus ko_{8k+3}(BC_l)|2^{-1}$, and similarly for $8k+7$. Indeed, $ko_{8k+3}(BD_{2^{N+2}})=2^{(2N+12)k+N+8}=[2^{4k+3}]^2 + ko_{8k+3}(BC_l)$, where the sum is not direct, and the the $[2^{4k+3}]^2$ comes from the two projective spaces. Thus we only need to understand where this factor of two is lost in order to understand the exact structure of the groups, at least in terms of $ko_*(BC_l)$.\\
Indeed, in $n=8k+3$, we saw above that the linear relation can be detected by comparing the eta invariants with respect to the real representation, of real projective and cyclic lens spaces. Further by Proposition $4.3.2$ we see $\eta(i_*(L^{4m-1}(l;1, \cdots ,1)))(\sigma_1-2\rho_0)=\eta(L^{4m-1}(l;1, \cdots ,1))(\rho_1 +\rho_{-1}-2\rho_0)$ has order $2^{N+2m-1} \in \R/\Z$, but because we are in dimensions $3 \mod 8$ and $(\rho_1 +\rho_{-1}-2\rho_0)$ is real, we can work in $\R/2\Z$, so that the eta invariant does indeed have order $2^{N+2m}$. It follows that $ko_{8k+3}(BD_{2^{N+2}})=[2^{4k+3}]^2 \oplus ko_{8k+3}(BC_{2^{N+1}})/\Z_2$, where the order of the smallest cyclic summand of $ko_{8k+3}(BC_{2^{N+1}})$ is halved, due to the relation with the eta invariants of real projective spaces we described.\\
By contrast, in dimensions $7 \mod 8$, so the eta invariant $\eta(i_*(L^{4m-1}(l;1, \cdots ,1)))(\sigma_1-2\rho_0)=\eta(L^{4m-1}(l;1, \cdots ,1))(\rho_1 +\rho_{-1}-2\rho_0)$ has order $2^{N+2m-1} \in \R/\Z$ here has order only $2^{N+2m-1}$ since we stay in $\R/\Z$, meaning it is here that the factor of two must be lost.\\
Thus $ko_{8k+7}(BD_{2^{N+2}})=[2^{4k+4}]^2 \oplus ko_{8k+7}(BC_{2^{N+1}})/\Z_2$, where the order of the largest cyclic summand of $ko_{8k+3}(BC_{2^{N+1}})$, which is $[2^{N+4k}]$, is halved.\\

An easy example of this can be seen in the case of $D_8$. Here the cyclic group in question is $C_4$, and $ko_{4k+3}(BC_4)$ has two generators, which we may choose as $L^{4k+3}(4;1, \cdots,1)$ and a formal difference $a(L^{4k+3}(4;1, \cdots,1))-b(L^{4k+3}(4;1,1, \cdots,1,3))$, with $a$ and $b$ odd.\\
It is known, from \cite{3}, example $7.3.3$, and Chapter 2, that $ko_{8m+3}(BC_4)=[2^{4m+3}] \oplus [2^{2m+1}]$, and $ko_{8m+7}(BC_4)=[2^{4m+5}] \oplus [2^{2m+1}]$, where the $[L^{4k+3}(4;1, \cdots,1)]$ class has the larger order. Thus from the discussion above we may deduce that in dimensions $3 \mod 8$, we lose a factor of two from the $[2^{2m+1}]$, from which we deduce
$$ko_{8m+3}(BD_8) \cong [2^{4m+3}]^2 \oplus [2^{4m+3}] \oplus [2^{2m}]; $$
While in $7 \mod 8$ it is the other way, and we get
$$ko_{8m+7}(BD_8) \cong [2^{4m+4}]^2 \oplus [2^{4m+4}] \oplus [2^{2m+1}] $$
and similarly for the higher groups.\\

We also note here that there is a stable $2-$local decomposition in \cite{mp}:
$$BD_{2^{N+2}} \cong 2\RP^{\infty} \vee 2L(2) \vee BPSL_{2}(q)$$
where $L(2)$  is a certain spectrum whose $2-$local $ko-$homology groups are entirely in even degrees, and $PSL_{2}(q)$ is the $2 \times 2$ projective special linear group over the field with $q$ elements, where $q=q(N)$ is an odd prime power, chosen so that $D_{2^{N+2}}$ is the Sylow 2-subgroup of $BPSL_2(q)$. Thus, as we've seen using our eta invariant calculations, we can always split off two $ko_*(\RP^{\infty})$ summands. It is then a matter of checking that $2-$locally, $ko-$fundamental classes of cyclic lens spaces span all of $ko_{4i-1}(BPSL_{2}(q))$, which our calculations of course imply. However, this could also be done by restricting representations directy from $PSL_{2}(q)$ and using the eta invariant.

\section{The classes detected in ordinary homology}
Here we realize the second column of the local cohomology spectral sequence, along with a single extra class in the first column. These are concentrated in even degrees, and is all annihilated by 2. The method is to realize sufficiently many classes in the ordinary $\Z_2$ homology of the dihedral groups. All cohomology and homology is with $\Z_2$ coefficients here.\\
Note that from the previous chapter, we know the conjecture is true for the Klein four group $V(2)=D_4$, and further, we know $H_{2n}^{+}(BV(2))$ explicitly. Thus the first step will be to consider how big a subspace is spanned by inclusion. We will see that this will almost suffice, except in dimensions $8k$ where an extra geometric construction is needed.\\
So consider the higher dihedral groups, starting with the group of order $8$. These all have the same cohomology ring, namely $\Z_{2}[\alpha,\beta,\delta]/\alpha \beta +\beta^{2}$, where $\delta$ is in degree $2$, for which we will always use the basis $\{\alpha^i \delta^j,\beta^i \delta^j\}$, with $i,j \geq 0$. We denote by $\xi(x)$ the dual element in homology with respect to this basis, for $x\in H^{*}(BD_{2^{N+2}})$ . The number of classes we must detect in $H_{*}(BD_8)$ in even dimensions is known from the first section to be the same as the number we must detect in $H_*(BV(2))$. The following Proposition calculates the image under inclusion from Klein 4 subgroups, and verifies the conjecture for $D_8$ in all dimensions except $0 \mod 8$.
\begin{prop}
The fundamental classes of spin projective bundles over projective spaces, mapped via inclusion from Klein 4 subgroups, span subspaces of dimension $2^{k}$ for $n=4k+2$, $2^{k+1}$ for $n=8k+4$, and $2^{k}$ for $n=8k$, in $H_{*}(BD_8)$.
\end{prop}
\begin{proof}
 There are two inclusions $V(2), V(2)'\hookrightarrow D_{8}$, so the first step is to try and see how many classes we get from the induced maps $ko_{*}(BV(2)) \rightarrow ko_{*}(BD_8)$ and $ko_{*}(BV(2)') \rightarrow ko_{*}(BD_8)$. We again detect these classes in ordinary homology, and this is done by restricting representations to calculate the induced restriction maps $H^{*}(BD_{8}) \rightarrow H^{*}(BV(2))$ respectively $V(2)'$, and then considering the dual maps in homology.\\
To this end we denote by $\omega \in D_{8}$ the rotation by $\pi/2$, and $s,s',t$, and $t'$ the reflections through the lines $y=0,x=0,y=x$, and $y=-x$ respectively. The two Klein four subgroups are $V(2)=<s,s'>$ and $V(2)'=<t,t'>$. Then it is known that $H^{*}(BD_{8},\Z_{2})=\Z_{2}[\alpha,\beta,\delta]/\alpha \beta +\beta^{2}$, where $\alpha=w_{1}(\hat{\omega}),\beta=w_{1}(\hat{s}),\delta=w_{2}(\hat{\sigma})$. Here $\hat{ }$ denotes the inflated representation, and $\sigma$ is the unique two dimensional representation. Further we have $H^{*}(BV(2))=\Z_{2}[p,q]$, with $p=w_{1}(\hat{(\omega)^{2}}),q=w_{1}(\hat{s})$, and $H^{*}(BV(2)')=\Z_{2}[p',q']$, with $p'=w_{1}(\hat{(\omega)^{2}}),q'=w_{1}(\hat{t})$. Restricting representations then gives us the induced maps in cohomology. The map $H^{*}(BD_{8})\rightarrow H^{*}(BV(2))$ is determined by
$$ \alpha \mapsto p, \beta \mapsto 0, \delta \mapsto q(p+q),$$
and similarly for $V(2)'$, we get
$$ \alpha \mapsto p', \beta \mapsto p', \delta \mapsto q'(p'+q')$$
 So if $A=\Z_{2}[\alpha,\delta]$ we deduce $H^{*}(BV(2))=A \oplus qA$ and similarly for $V(2)'$.
 $$H^{*}(BD_{8},\Z_{2}) \rightarrow A \rightarrowtail H^{*}(BV(2))=A \oplus qA $$
 The plan now is to dualise, and choose dual bases in homology and compute where the classes in $H_{*}^{+}(BV(2)$ ( which we know explicitly) map to.
 $$H_{*}(BD_{8},\Z_{2}) \leftarrow A^{\vee} \twoheadleftarrow H_{*}(BV(2))=A^{\vee} \oplus qA^{\vee} $$
 We now consider $\xi_{(a,b)} \in H_{n}(BV(2))$. Then using this decomposition $\xi_{(a,b)} \mapsto \Sigma \xi(k)$ where $k \in H^{*}(BD_{8})$ and $k \mapsto p^{a}q^{b} + \cdots $, where $\xi$ again denotes the dual basis. This means that $p^aq^b$ occurs as a non-zero summand in the monomial basis decomposition of the image of the cohomology class $k$ under restriction. Notice that the situation for $A'$ is almost the same, except that here both $\alpha, \beta \rightarrow p'$, so whenever we get a $\xi(\alpha^{i}\delta^{j})$ summand from $A$, we get $\xi(\alpha^{i}\delta^{j})+\xi(\beta^{i}\delta^{j})$ here.\\
We start with $n=4k+2$. Then, using the dual basis to the monomial basis $H_{n}^{+}(BV(2))$ is $k-$ dimensional generated by $\xi_{(a,b)}$, the images of $\RP^{a} \times \RP^{b}$, with $a,b\equiv 3 \mod 4$.  We see that $\alpha^{4j}\delta^{a} \rightarrow p^{4k}q^{a}(p+q)^{a}=p^{4k+a}q^{a} + \cdots$ so that inlcusion from $V(2)$ gives $\xi_{(a+4j,a)} \rightarrow \xi(\alpha^{4j} \delta^{a})$, while inclusion from $V(2)'$ gives $\xi_{(a+4j,a)} \rightarrow \xi(\alpha^{4j} \delta^{a})+\xi(\beta^{4j} \delta^{a})$. Since we start at $a=3$, adding these up gives us all the $k$ required classes in $H_{*}^{+}(BD_{8})$.\\

For $n=4k$ $H_{n}^{+}(BV(2))$ is $(k+1)$ dimensional generated by $\xi_{(4k-1,1)},\xi_{(1,4k-1)}$ and by all the $\xi_{(a,n-a)} +\xi_{(a-2,n-a+2)}$, with $5 \leq a \equiv 1 \mod 4$. Recall from Lemma  $3.4.2$ that these classes are a basis for the images of fundamental classes of $\RP^{n-1} \times \RP^1$, along with $M_{(a,b)}=\RP(2L_{0} \oplus (n-1-a)\varepsilon \rightarrow \RP^{a})$, with $5 \leq a \equiv 1 \mod 4$ and $b=n-a+2$. Dualising again, we check that the classes we obtain are exactly all those of the form $\xi(\alpha^{n-8i-2} \delta^{4i+1}),\xi(\beta^{n-8i-2} \delta^{4i+1})$, with $0 \leq i \leq (n-2)/8$. Adding these gives the necessary $k+1$ if $n \equiv 4 \mod 8$, but only $k$ if $n \equiv 0 \mod 8$. $(*)$\\
Indeed, $\alpha^{n-2} \delta$ restricts to $p^{n-1}q + \cdots $ on $V(2)$ so from the product $\RP^{n-1} \times \RP^{1}$, we get the class $\xi(\alpha^{n-2} \delta)$ and similarly  $\xi(\alpha^{n-2} \delta)+\xi(\beta^{n-2} \delta)$ from restriction to $V(2)'$. Recall from section $3.4$ that $[M_{(a,b)}]$ gives us the class $\xi_{(a,n-a)} +\xi_{(a-2,n-a+2)}$, so it will be mapped to $\Sigma \xi(k)$, where the sum is over all elements $k \in H^{*}(BD_8)$ that map under restriction to a class $x$, whose monomial basis decomposition is of the form $p^{a}q^{n-a}+ \cdots$ or $p^{a-2}q^{n-a+2}+\cdots$, but not of the form $p^{a}q^{n-a}+ p^{a-2}q^{n-a+2}+\cdots$ . This means that exactly one of $p^{a}q^{n-a}$ and $p^{a-2}q^{n-a+2}$ occurs as a non-zero summand in the monomial basis decomposition of the image of $k$. Note that when $n=12$, under restriction to $V(2)$ we get that $\alpha^{2}\delta^{5} \rightarrow p^2q^{10}+p^3q^9+p^6q^6+p^7q^5$, and the same for $\beta^{2}\delta^{5}$ under restrction to $V(2)'$, so that when we dualise by the above method $[M_{(9,5)}]$ (and $[M_{(5,9)}]$) map to $\xi(\alpha^{2}\delta^{5})$ and $\xi(\beta^{2}\delta^{5})$. We now observe that if we have realized $\xi(\alpha^{i}\delta^{j}) \in H_{*}^{+}(BD_{8})$ this way, then we also realize $\xi(\alpha^{i+4k}\delta^{j}) \in H_{*}^{+}(BD_{8})$, because if $j=1$ it is just a product with a higher dimensional projective space, and otherwise the original class was realized by some $M_{(a,b)}$, so that $\xi(\alpha^{i+4k}\delta^{j})$ will be realized by one of $M_{(a+4,b)}$ or $M_{(a,b+4)}$. Similarly, if we have $\xi(\alpha^{2}\delta^{j}) \in H_{*}^{+}(BD_{8})$ then we also have $\xi(\alpha^{2}\delta^{j+4k})$. The $j=k=1$ case is the $n=12$ scenario done above, so just as before, if the original class was realized by some $M_{(a,b)}$, then $\xi(\alpha^{2}\delta^{j+4k})$ will be realized by one of $M_{(a+8,b)}$ or $M_{(a,b+8)}$.\\

So, to fully verify $(*)$, we claim we can have no term of the form $\xi(\alpha^{i}\delta^{4j+3})$ if $n=4k$. Indeed, note that $n \equiv 0 \mod 4$ so $i \equiv 2 \mod 4$, and so $p^{n-1}q$ and $pq^{n-1}$ can not be summands in the monomial basis decomposition of the image of $\alpha^{i}\delta^{4j+3}$ under restriction. Thus the fundamental class of a product $\RP^{n-1} \times \RP^{1}$ can not map to $\xi(\alpha^{i}\delta^{4j+3})$ under inclusion. This leaves the possibility of some $M_{(a,b)}$, so suppose the homology class $[M_{(a,b)}]$ maps to $\xi(\alpha^{i}\delta^{4j+3})+ \cdots$. This means that $\alpha^{i}\delta^{4j+3} \rightarrow p^{i}(q(p+q))^{4j+3}$ must map to $p^aq^{b-2} + \cdots$ or $p^{a-2}q^b + \cdots$ but not $p^{a-2}q^{b}+p^a q^{b-2} + \cdots$. So equating binomial coefficients we must have $(4j+3)C_{a-i} \equiv 1 \mod 2$ and $(4j+3)C_{a-i-2} \equiv 0 \mod 2$, or vice versa.\\
However note that $a-i \equiv 3 \mod 4$, and for any $I,J$ we have that the quotient $(4J+3)C_{4I+3}/(4J+3)C_{4I+1}$ is given by $(4(J-I)+2)(4(J-I)+3)/((4I+2)(4I+3))=(2(J-I)+1)(4(J-I)+3)/((2I+1)(4I+3))$, which is a fraction of two odd numbers, so working modulo $2$, the two coefficients are the same.
\end{proof}
Thus we have realized everything we need to, except in dimensions $0 \mod 8$.\\

To generate the rest of the subspace we use the fundamental class of the manifold given by the following construction. The idea is to start with the exact sequence $C_{4} \rightarrow D_{8} \rightarrow \Z_{2}$, lift it to obtain a sequence $C_{4} \rightarrow G \rightarrow \Z$, and take classifying spaces. Of course $S^{1} =B\Z$, and since $C_{4}$ acts on $n\Comp$ by multiplication, we have a map from a lens space $L=S(n \Comp)/C_4$ into $BC_{4}=S(\infty \Comp)/C_4$, and thus a fibre bundle $ L^{2n-1} \rightarrow M^{2n} \rightarrow S^{1}$, as shown. As a lens bundle this carries positive scalar curvature, and we claim that when $n \equiv 0 \mod 4$, this is a spin manifold which gives us a  homology class independent of those induced from the Klein 4 subgroups which we already know.
$$
\xymatrix{
L^{2n-1} \ar[r] \ar [d] & M^{2n} \ar[r] \ar [d]^{f} & S^{1} \ar[d]^{\simeq} \\
BC_{4} \ar[r] \ar [d] & BG \ar[r] \ar [d]^{g} & B\Z \ar[d] \\
BC_{4} \ar[r] & BD_{8} \ar[r] & B\Z_{2} }
$$
Let $F=g \circ f$ in the diagram above. The following Proposition then completes the proof of the conjecture for $D_8$.
\begin{prop}
The cohomology ring of the manifold is $M^{2n}$ is given by
$$H^{*}(M^{2n},\Z_{2})=\Z_{2}[\sigma, \tau, x]/\sigma^{2}, \sigma \tau + \tau^{2}, x^{n}$$
where $|\sigma|=|\tau|=1, |x|=2$. Further, $F_*([M^{2n}])=\xi(\beta^2 \delta^{n-1})$, and when $n$ is even, $M^{2n}$ is a spin manifold.
\end{prop}
\begin{proof}
We know $H^{*}(L^{2n-1}; \Z_{2})=\Z_{2}[X, \widetilde{\tau}]/X^{n}, \widetilde{\tau}^{2}$ where $X$ is a dimension $2$ generator, and $\widetilde{\tau}$ dimension $1$, and $H^{*}(S^{1})=\Z_{2}[\sigma]/\sigma^{2}$. By the universal coefficient Theorem, all the cohomology groups except of course $H^{0}(M),H^{2n}(M)$ have rank $2$, and the Serre spectral sequence implies that there are two degree one and one degree two generators, one of the degree one generators being $\sigma=F^{*}(\alpha)$, with $\sigma^{2}=0$. We can then simply define $\tau=F^{*}(\beta)$, and $x=F^{*}(\delta)$. Since $\alpha \beta +\beta^2=0$ we must have $\sigma \tau+ \tau^2=0$.\\
 To see that $x^{n}=0$ we recall that $\delta = w_{2}(\sigma)=w_{2}(\sigma_1)$, where $\sigma_1$ is the $2$ dimensional representation, and $\sigma$ is the representation obtained by restricting to $O(2)$, thus giving a real two-dimensional bundle. Now we can view $L^{2n-1}=S(n \sigma)/ C_{4}$. Notice that the universal cover $\widetilde{M}$ of $M=S(n \sigma) \times_{G} \R$ is just $ S(n \sigma) \times \R$, and so the induced vector bundle over $M$ is given by
 $$ n \sigma \times_{G}(S(n \sigma) \times \R) \rightarrow  S(n \sigma) \times_{G} \R$$
This has a section via the diagonal map, which implies the top Stiefel-Whitney class $x^{n}$ must be zero. \\
Further, we claim this manifold $M$ is spin when it is of dimension $0 \mod 4$. To see this we recall that in a smooth manifold, we have by the Wu formulae \cite{micc} that $w_{k}=\sum_{i+j=k} Sq^{i}(v_{j})$, where $v_{j} \in H^{j}(M^{n})$ is the unique class such that $v_{j}y=Sq^{j}(y)$, for every $y \in H^{n-j}(M)$.\\
 For our manifold $M^{2n}$ we wish to see that $w_{1}=w_{2}=0$, if $2n \equiv 0 \mod 4$. It is known \cite{dihcoh} that in $H^{*}(BD_{8})$ we have $Sq^{1}(\delta)=\alpha \delta$, so $Sq^{1}(x)=x \sigma$. This implies $Sq^{1}(x^{k})=x^{k} \sigma$ if $k$ is odd, and zero if $k$ is even.The generators of $H^{2n-1}(M^{2n})$ are $x^{n-1}\sigma$ and $x^{n-1} \tau$. We then have
 $$ Sq^{1}(x^{n-1}\sigma)= x^{n-1}\sigma ^{2} + \sigma Sq^{1}(x^{n-1})=0$$
 Similarly, since $n-1$ is odd,
 $$ Sq^{1}(x^{n-1}\tau)=x^{n-1}\tau ^{2} +  \tau Sq^{1}(x^{n-1})=x^{n-1}\tau ^{2}+x^{n-1}\tau \sigma=0$$
 So we deduce $v_{1}=w_{1}=0$. So next we consider
 $$ Sq^{2}(x^{n-1})=x^{n} + x \sigma Sq^{1}(x^{n-2}) + x Sq^{2}(x^{n-2})=0$$
 This follows since $Sq^{1}(x^{n-2})=0$ as $n-2=2m$ is even, and $Sq^{2}(x^{2m})= Sq^{1}(x^{m})Sq^{1}(x^{m})=0$, because $\sigma ^{2}=0$. Finally
 $$ Sq^{2}(x^{n-2} \tau ^{2})=\tau ^{4}x^{n-2} + \tau ^{2}Sq^{2}(x^{n-2}) +  Sq^{1}(x^{n-2})Sq^{1}(\tau^{2})=0$$
 So we conclude $v_{2}=0$ and so $w_{2}=v_{1}^{2}+v_{2}=0$, and we have a spin manifold, if $n$ is even. The top cohomology class is $x^{n-1} \tau^{2}=x^{n-1}\sigma \tau$, and dualising $F^{*}$ immediately gives $F_{*}(\xi(x^{n-1} \tau^{2}))=\xi(\beta^{2} \delta^{n-1})$.
\end{proof}
If we are working in dimensions $4 \mod 8$ ( so that $n-1 \equiv 1 \mod 4$), then, as seen above, this class is realized from $V(2)'$. However, in dimensions $0 \mod 8$, this gives $\xi(\beta^{2} \delta^{j})$, with $j \equiv 3 \mod 4$, which is the independent of the classes induced by inclusion from $V(2)$ and $V(2)'$. Combining this new class with Proposition $4.4.1$, we have realized the whole of the $Ker(Ap)$ for $BD_8$. The fact that this class lies in H1 follows from the Adams spectral sequence calculations in \cite{mjam}, which show that it has a non-trivial eta-multiple. This lies in H1 in the local cohomology spectral sequence, which implies the class constructed does too.\\

We now wish to generalize this to higher dihedral groups. The orders of the subspaces we must realize in the ordinary homology for $n=4k+2$ and $n=4k$ are $2^i$ with $i\leq k$ and $i \leq k+1$ respectively, and in fact we will show that the upped bounds are attained.\\
 We start by observing how the smaller dihedral groups include into the larger ones, in homology.\\
 We briefly recap the representation theory of a dihedral group $D$ of order $2^{N+2}$. There are four one dimensional representations, namely the trivial one, $\hat{D'}=\hat{s},\hat{D''}=\hat{\omega s}$, and $\hat{D'}\hat{D''}=\hat{\omega }$, where $D'$ and $D''$ are the dihedral subgroups of order $2^{N+1}$ generated by the even and odd conjugacy classes of reflections respectively. The remaining simple representations $\sigma_{1},\cdots ,\sigma_{2^N-1}$ are of dimension $2$, restricting from $O(2)$. Just as for $D_8$, the cohomology ring is
 $$H^{*}(BD;\Z_{2})=\Z_{2}[\alpha,\beta,\delta]/\beta(\alpha+\beta)$$
 where $\alpha=w_{1}(\hat{D'})$ and $\alpha+\beta=w_{1}(\hat{D''})$. Further, while $\alpha, \delta$ restrict from $O(2)$, $\beta=\beta_{N}$ depends on $N$, and it restricts to $0$ on $D'$ and $\alpha$ on $D''$.\\
 It thus follows that the restrictions from $H^{*}(BD;\Z_{2})$ to $H^{*}(BD';\Z_{2})$ and $H^{*}(BD'';\Z_{2})$ both send $\alpha \rightarrow \alpha$, $\delta \rightarrow \delta$, and $\beta \rightarrow 0, \alpha$, for $D'$ and $D''$ respectively.\\
 Dualising the maps in homology we deduce just as in Proposition $4.4.1$ that $\xi(\alpha^{i}\delta^{j})$ maps to $\xi(\alpha^{i}\delta^{j}), \xi(\beta_{N}^{i}\delta^{j})$,  whilst the $\xi(\beta_{N-1}^{i}\delta^{j})$ map to zero. Using this, we now deduce the GLR conjecture for all dihedral groups.
\begin{prop}
Fundamental classes of projective bundles over projective space, together with a lens space bundle over the circle, all of them spin, span subspaces of dimension $2^{k}$ for $n=4k+2$ and $2^{k+1}$ for $n=4k$, in $H_{*}(BD_{2^{N+2}})$.
\end{prop}
\begin{proof}
In $n=4k+2$ a $k$-dimensional space in $H_{*}^{+}(BD_{8})$ is spanned by the set of $\xi(\alpha^{4j} \delta^{a}),\xi(\beta_{1}^{4j} \delta^{a})$, with $n-3 \geq a \equiv 3 \mod 4$. It's clear that this spanning set maps to the set of $\xi(\alpha^{4j} \delta^{a}),\xi(\beta_{2}^{4j} \delta^{a})$ in $H_{*}^{+}(BD_{16})$, and by induction to $\xi(\alpha^{4j} \delta^{a}),\xi(\beta_{N}^{4j} \delta^{a})$ in $H_{*}^{+}(BD_{2^{N+2}})$. These still span a $k-$dimensional space, as required.\\
In $n=4k$ a $k$-dimensional space in $H_{*}^{+}(BD_{8})$ is spanned by $\xi(\alpha^{n-8i-2} \delta^{4i+1}),\xi(\beta^{n-8i-2} \delta^{4i+1})$, with $0 \leq i \leq (n-2)/8$, along with an extra $\beta_{1}^2 \delta^{2k-1}$ if $n \equiv 0 \mod 8$. This latter class will be mapped to zero in $H_{*}^{+}(BD_{16})$.\\
Thus, using induction as before, the classes $\xi(\alpha^{n-8i-2} \delta^{4i+1}),\xi(\beta^{n-8i-2} \delta^{4i+1})$ are mapped to $\xi(\alpha^{n-8i-2} \delta^{4i+1}),\xi(\beta_{N}^{n-8i-2} \delta^{4i+1})$ in $H_{*}^{+}(BD_{2^{N+2}})$, which still span the required $(k+1)$ dimensional space if $n \equiv 4 \mod 8$, but only a $k$-dimensional one if $n \equiv 0 \mod 8$.\\
We produce an extra class exactly as before, replacing $C_4$ by a cyclic group of higher order, namely, $C_{2^{N+1}}$ for $D_{2^{N+2}}$.
$$
\xymatrix{
L^{2n-1} \ar[r] \ar [d] & M^{2n} \ar[r] \ar [d]^{f} & S^{1} \ar[d]^{\simeq} \\
BC_{2^{N+1}} \ar[r] \ar [d] & BG \ar[r] \ar [d]^{g} & B\Z \ar[d] \\
BC_{2^{N+1}} \ar[r] & BD_{2^{N+2}} \ar[r] & B\Z_{2} }
$$
The cohomology of the lens space we get like this is the same as for $C_4$, and the rest of the calculation proceeds exactly as before, to give us $\xi(\beta_{N}^2 \delta^{n-1})$, which is an independent extra class when $n \equiv 0 \mod 4$, just as required.
\end{proof}
This last construction of a lens space bundle over the circle is the only class we need to realize for which inclusion from proper subgroups does not suffice, and as we've seen, it is required separately for all the dihedral groups. Indeed, the $H^2$ part of the local cohomology spectral sequence being the same as for $V(2)$ is misleading here. In this regard, computations using the Adams spectral sequence are more enlightening; see \cite{bay} and \cite{mjam}. The method is to again use the stable $2-$local decomposition in \cite{mp}:
$$BD_{2^{N+2}} \cong 2\RP^{\infty} \vee 2L(2) \vee BPSL_{2}(q)$$
The class in question then lies in $ko_*(BPSL_{2}(q))$ in Adams filtration zero.\\
A similar phenomenon occurs in the next chapter for the semi-dihedral groups, and an analogous construction is used. Once again, the class in question always lies in $H^1$.

\chapter{The Semi-Dihedral groups}
\section{Preliminaries}
We recall that the Semi-Dihedral group $SD_{2^{N+2}}$ has presentation $<s,t;s^{2^{N+1}}=t^2=1,tst=s^{2^{N}-1}>$. There are three maximal subgroups, namely the cyclic $<s>$, the dihedral $<s^2,t>$, and the quaternion $<s^2,ts>$, each of which is the kernel of a nontrivial irreducible one dimensional real representation. The remaining representions are two dimensional, with $2^{N-1}$ complex representations occuring in conjugate pairs, and $2^{N-1}-1$ real representations. Here is the character table of $SD_{16}$:
\begin{center}
\begin{tabular}{|c|c|c|c|c|c|c|c|}
\hline
$ $& $1$ & $1$ & $2$ & $2$ & $2$ & $4$ & $4$\\
$\rho$&$[1]$ & $[s^4]$ & $[s]$ & $[s^2]$ & $[s^5]$ & $[t]$ & $[ts]$\\
\hline
\hline
$1=\rho_0$ & $1$ & $1$ & $1$ & $1$ & $1$& $1$ & $1$\\
$\chi_2=\hat{C_8}$ & $1$ & $1$ & $1$ & $1$ & $1$ & $-1$ & $-1$\\
$\chi_3=\hat{D_8}$ & $1$ & $1$ & $-1$ & $1$ & $-1$ & $1$ & $-1$\\
$\chi_4=\hat{Q_8}$ & $1$ & $1$ & $-1$ & $1$ & $-1$ & $-1$ & $1$\\
$\chi_{\rho}$ & $2$ & $-2$ & $\sqrt{2}i$ & $0$ & $-\sqrt{2}i$ & $0$ & $0$\\
$\chi_{\rho^2}$ & $2$ & $2$ & $0$ & $-2$ & $0$ & $0$ & $0$\\
$\chi_{\rho^5}$ & $2$ & $-2$ & $-\sqrt{2}i$ & $0$ & $\sqrt{2}i$ & $0$ & $0$\\
\hline
\hline
\end{tabular}
\end{center}

The cohomology ring $H^*(BSD_{2^{N+2}}; \Z_2)$ is given by $$\Z_2[x,y,u,P]/x^3,x^2+xy,xu,u^2+(x^2+y^2)P$$
with $|x|=|y|=1,|u|=3,|P|=4$, see \cite{ep}, \cite{kijti}. Further, we have $x=w_1(\chi_3), y=w_1(\chi_2)$ and $P=\overline{c}_2(\chi_{\rho})$, the mod $2$ reduction of the second Chern class.\\
The Steenrod square actions are determined by
$$Sq^1(u)=Sq^1(P)=0;Sq^2(u)=0; Sq^2(P)=u^2$$

\section{The local cohomology spectral sequence and ko calculations}

The connective $ko-$homology $ko_*(BSD_{16})$ was calculated in $\cite{kijti}$, Theorem $6.5.2$, using the local cohomology spectral sequence. Here is the $E_{\infty}$ page of the spectral sequence, together with the Kernel $Ker(Ap)$:
\begin{thm}
The $E_{\infty}$ page of the spectral sequence, together with the kernel $Ker(Ap)$, are as follows:
\end{thm}

\setlength{\unitlength}{1cm}
\begin{center}
\begin{picture}(20,17)
\multiput(1.5,2)(0,0.5){28}%
{\line(1,0){11.5}}
\put(1.5,2){\line(0,1){14}}
\put(3,2){\line(0,1){14}}
\put(11,2){\line(0,1){14}}
\put(13,2){\line(0,1){14}}
\multiput(11.9,4.6)(0,0.5){3}%
{$0$}
\multiput(11.9,8.6)(0,0.5){3}%
{$0$}
\multiput(11.9,12.6)(0,0.5){3}%
{$0$}

\multiput(11.9,3.6)(0,4){3}%
{$0$}
\multiput(11.8,2.6)(0,0.5){2}%
{$2^{3}$}
\multiput(11.8,6.6)(0,0.5){2}%
{$2^{4}$}
\multiput(11.8,10.6)(0,0.5){2}%
{$2^{5}$}
\multiput(11.8,14.6)(0,0.5){2}%
{$2^{5}$}
\multiput(11.8,2.1)(0,2){7}%
{$\mathbb{Z}$}

\multiput(6.8,2.6)(0,1){3}%
{$0$}
\multiput(6.8,7.6)(0,1){3}%
{$0$}
\multiput(6.8,11.6)(0,1){3}%
{$0$}
\multiput(6.8,3.1)(0,1){1}%
{$0$}

\put(6.8,2.1){$0$}
\put(6.8,5.6){$0$}
\put(5.6,4.1){$ [4]\oplus[8]\oplus[8]$}
\put(6.7,5.1){$[2]$}
\put(5.4,6.1){$[2]\oplus [4]\oplus[16]\oplus[32]$}
\put(6.7,6.6){$2$}
\put(6.3,7.1){$2\oplus [2]$}
\put(4.9,8.1){$[8]\oplus[16]\oplus[128]\oplus[128]$}
\put(6.7,9.1){$[4]$}
\put(4.7,10.1){$[2]\oplus [8]\oplus[16]\oplus[256]\oplus[512]$}
\put(6.6,10.6){$2^{2}$}
\put(6.2,11.1){$2^{2}\oplus [4]$}
\put(4.6,12.1){$[2]\oplus [32]\oplus[64]\oplus[2048]\oplus[2048]$}
\put(6.7,13.1){$[8]$}
\put(4.4,14.1){$[4]\oplus [32]\oplus[64]\oplus[16^{3}]\oplus[2\cdot16^{3}]$}
\put(6.6,14.6){$2^{2}$}
\put(6.2,15.1){$2^{2}\oplus [8]$}

\put(6.8,15.6){$\vdots$}
\put(2,15.6){$\vdots$}
\put(11.8,15.6){$\vdots$}

\put(2,4.1){$0$}

\put(2,2.1){$0$}
\multiput(2,2.6)(0,1){13}%
{$0$}
\multiput(2,3.1)(0,1){1}%
{$0$}

\put(2,6.1){$0$}
\put(2,5.1){$2$}
\put(2,7.1){$2$}
\put(2,8.1){$2$}
\put(1.9,9.1){$2^{2}$}
\put(2,10.1){$2$}
\put(1.9,11.1){$2^{2}$}
\put(1.9,12.1){$2^{2}$}
\put(1.9,13.1){$2^{3}$}
\put(1.9,14.1){$2^{2}$}
\put(1.9,15.1){$2^{3}$}

\put(13.3,2.1){0}
\put(13.3,2.6){1}
\put(13.3,3.1){2}
\put(13.3,3.6){3}
\put(13.3,4.1){4}
\put(13.3,4.6){5}
\put(13.3,5.1){6}
\put(13.3,5.6){7}
\put(13.3,6.1){8}
\put(13.3,6.6){9}
\put(13.2,7.1){10}
\put(13.2,7.6){11}
\put(13.2,8.1){12}
\put(13.2,8.6){13}
\put(13.2,9.1){14}
\put(13.2,9.6){15}
\put(13.2,10.1){16}
\put(13.2,10.6){17}
\put(13.2,11.1){18}
\put(13.2,11.6){19}
\put(13.2,12.1){20}
\put(13.2,12.6){21}
\put(13.2,13.1){22}
\put(13.2,13.6){23}
\put(13.2,14.1){24}
\put(13.2,14.6){25}
\put(13.2,15.1){26}
\put(13.2,15.6){27}

\put(12.5,16.2){degree(t)}

\put(6.3,1.2){$H^{1}_{I}(R)$}
\put(11.5,1.2){$H^{0}_{I}(R)$}
\put(1.8,1.2){$H^{2}_{I}(R)$}

\put(0.2,0.2){where$[n]:=$ cyclic group of order $n$, $2^{r}$:= elementary abelian group of order $r$.}
\put(2.1,-1.2){ The $E_{\infty}$-page for $ko_{*}(BSD_{16}).$}
\end{picture}
\end{center}

\newpage
The kernel Ker(Ap) is thus given by the classes in first and second column:\\
\begin{center}
\begin{tabular}{|c|c|c|}
  \hline
   & & \\
  $n$ & $[H_{I}^{1}(ko^{*}(BSD_{16}))]_{n+1}$ & $[H_{I}^{2}(ko^{*}(BSD_{16}))]_{n+2}$ \\
  & & \\
  \hline
  \hline
   $n\leq 3$ & $0$ & $0$ \\
  $3$ & $[4]\oplus[8]\oplus[8]$ & $0$ \\
  $4$ & $0$ & $2$ \\
  $5$ & $[2]$ & $0$ \\
  $6$ & $0$ & $0$ \\
  $7$ & $[2]\oplus[4]\oplus[16]\oplus[32]$ & $0$ \\
  $8$ & $2$ & $2$ \\
  $9$ & $2\oplus [2]$ & $0$ \\
  $10$& $0$ & $2$ \\
  $11$& $[8]\oplus[16]\oplus[128]\oplus[128]$ & $0$ \\
  $12$& $0$ & $2^{2}$ \\
  $13$ & $[4]$ & $0$ \\
  $14$ & $0$ & $2$ \\
  $15$ & $[2]\oplus[8]\oplus[16]\oplus[256]\oplus[512]$ & $0$ \\
  $8k\geq16$ & $2^{2}=2 \oplus \eta(ko_{8k-1}(BSD_{16}))$ & $2^{k}$ \\
  $8k+1\geq 17$ & $2^{2} \oplus [2^{k}]=\eta(\widetilde{ko}_{8k}(BSD_{16}))\oplus [2^{k}]$ & $0$ \\
  $8k+2\geq 18$ & $0$ & $2^{k}$ \\
  $8k+3\geq 19$ & $[2^{k-1}]\oplus[2\cdot 4^{k}]\oplus[4^{k+1}]\oplus[8\cdot16^{k}]\oplus[8\cdot16^{k}]$ & $0$ \\
  $8k+4\geq 20$ & $0$ & $2^{k+1}$ \\
  $8k+5\geq 21$ & $[2^{k+1}]$ & $0$ \\
  $8k+6\geq 22$ & $0$ & $2^{k}$ \\
  $8k+7\geq 23$ & $[2^{k}]\oplus[2\cdot 4^{k}]\oplus[4^{k+1}]\oplus[16^{k+1}]\oplus[2\cdot16^{k+1}]$ & $0$ \\
  \hline
  \hline
\end{tabular}.
\end{center}

Here, $[H_{I}^{1}(ko^{*}(BSD_{16}))]_{n+1}$ and $[H_{I}^{2}(ko^{*}(BSD_{16}))]_{n+2}$ contribute to $(\ker{Ap})_{n}\subseteq ko_{n}(BSD_{16})$.  Note further from \cite{kijti} (Theorem 6.5.2) that the two column, $[H_{I}^{2}(ko^{*}(BSD_{16}))]_{n+2}$, is embedded in $H_{n}(BSD_{16};\mathbb{Z}_{2})$.  For the 1-column, $[H_{I}^{1}(ko^{*}(BSD_{16}))]_{n+1}$, we have that the generator of $[H_{I}^{1}(ko^{*}(BSD_{16}))]_{8+1}$ is detected in $H_{8}(BSD_{16};\mathbb{Z}_{2})$, as the class $(yuP)^{\vee}=\xi(yuP)$ (the dual element of $yuP \in H^{8}(BSD_{16};\mathbb{Z}_{2})$).  In higher degrees, $[H_{I}^{1}(ko^{*}(BSD_{16}))]_{(8k)+1}$ contains two generators of order 2 where one of them is an $\eta$-multiple and the other one is detected in $H_{8k}(BSD_{16};\mathbb{Z}_{2})$ as $(yuP^{2k-1})^{\vee}=\xi(yuP^{2k-1})$.  The generator of $[H_{I}^{1}(ko^{*}(BSD_{16}))]_{9+1}$ is an $\eta$-multiple and in higher degrees $[H_{I}^{1}(ko^{*}(BSD_{16}))]_{(8k+1)+1}$ contains one $\eta$-multiple generator and one $\eta^{2}$-multiple generator (coming from the $\eta$-multiple in degree $(8k)+1$).  The other generators in this column which are not mentioned above come from $H_{I}^{1}(\overline{QO})$ which can be dealt with by character theory.\\
Note again that the $H^2$ part is the same for all semi-dihedral groups, and much like the dihedral groups, we show there are no more differentials for the higher semi-dihedral groups by realizing sufficiently many classes by positive scalar curvature spin manifolds in the ordinary homology.

\section{Eta invariant calculations in odd dimensions}
In odd dimensions, all of $Ker(Ap)$ lies in the 1-column of the local cohomology spectral sequence, and we will again show that it is realized by inclusion from periodic subgroups, together with eta multiples. We use the results from the second chapter for $Q_8$, together with some extra calculations for $C_8$ to prove the main result of this section, which verifies the conjecture in odd dimensions. Write $SD_{16}=<s,t;s^8=t^2=1,tst=s^3>$.
\begin{prop}
i) The images of inclusions from the cyclic subgroups $<s>,<t>$ and the quaternion subgroup $<s^2,ts>$, together span all of $ko_{4m+3}(BSD_{16})$.\\
ii) Inclusion from $<s>$ spans all of $ko_{8m+5}(BSD_{16})$, and together with Eta multiples, all of $Ker(Ap) \subset ko_{8m+9}(BSD_{16})$
\end{prop}
\begin{proof}
We see from Theorem $5.2.1$ that $|ko_{8m+3}(BSD_{16})|=2^{8+13m}$, while $|ko_{8m+7}(BSD_{16})|=2^{12+13m}$.\\

We start by considering explicitly the restrictions to $Q_8$, and using the calculations known for $Q_8$. From \cite{2} and chapter 2 we know that $ko_{4k-1}(BQ_8)$ is spanned by the $ko$-fundamental classes of the manifolds $M_{t}^{4k-1}=S^{k-1}/H_t$ with $t=1,2,3$ together with $M_{Q}^{4k-1}=S^{4k-1}/Q_l$ and $M_{Q}^{4k-9}=S^{4k-9}/Q_l$. For any such $M$ we defined
$$\overrightarrow{\eta}(M)=(\eta(M)(\rho_0-\kappa_1),\eta(M)(\rho_0-\kappa_3),\eta(M)(2-\tau),\eta(M)(2-\tau)^2)$$
and then made explicit calculations to check that thse spanned a subgroup of order at least $|ko_{4k-1}(BQ_8)|$.\\
We know that the real representation $\chi_{\rho^2}$ restricts to $\kappa_1+\kappa_3$, so since $\overrightarrow{\eta}(M_{1}^n-M_{2}^n)=(2^{-2m-i},0,0,0)$ with $i=1$ and $i=2$ when $n=8m+3$ and $n=8m+7$ resectively, we get a factor of $2^{2m+2}$ in both cases. This is because $\eta(M_1^n-M_2^n)(\rho_0-\kappa_1)=2^{-2m-1}$, and so when $n=8m+3$, since $\rho_0-\kappa_1$ is real, this eta invariant takes values in $\R/2\Z$ where it has order $2^{2m+2}$.\\
Further, the representations $\chi_{\rho}, \chi_{\rho^5}$ both restrict to the natural representation $\tau$ of $Q_8$, and $(2-\tau)^2=4+\tau^2-4\tau$ is restricted to by $4+\chi_{\rho}.\chi_{\rho^5}-2(\chi_{\rho} + \chi_{\rho^5})$ which is a real representation since $\chi_{\rho}$ and $\chi_{\rho^5}$ are complex conjugate.\\
So, by restricting representations we see directly that in dimensions $n=8m+3$ we have spanned a subspace of order $2^{2m+2}8^{2m+1}$ by including from $Q_8$.\\
In $n=8m+7$, since $\chi_{\rho}$ and $\chi_{\rho^5}$ are not quaternion representations, we must halve the order of whatever is detected by $\eta(M)(2-\tau)$. This is because $\tau$ is a quaternion representation, and thus $\eta(M)(2-\tau)$ takes values in $\R/2\Z$ for $Q_8$ when $n \equiv 7 \mod 8$, but $\eta(M)(2-\chi_{\rho})$ (and $\chi_{\rho^5}$) for $SD_{16}$ takes values in $\R/\Z$, and so even though the eta invariants of quaternion lens spaces for $SD_{16}$ can be calculated by just restricting to $\tau$, they will take values in $\R/\Z$ rather than $\R/2\Z$.\\
The rest of the argument is exactly the same as in $n=8m+3$, thus giving a subspace of order $2^{2m+2}8^{2m+2}2^{-1}=2^{2m+1}8^{2m+2}$.\\

Note further that the representation $\chi _4$ of $SD_{16}$ restricts trivially on the quaternion subgroup $<s^2,ts>$ and to the non trivial representation on $<t>$, and thus in addition to what is induced from $Q_8$ we have the $ko-$fundamental class $[\RP^n] \rightarrow ko_n(B<t>) \hookrightarrow ko_n(BSD_{16})$ of orders $2^{4m+3}, 2^{4m+4}$ when $n=8m+3, 8m+7$ respectively.\\
Putting this all together, we have in $8m+3$ a subgroup of order $2^{2m+2}8^{2m+1}2^{4m+3}=2^{8+12m}$ (this is already enough in $n=3$), while in $8m+7$ the order is $2^{4m+4} 2^{2m+1}8^{2m+2}=2^{11+12m}$. We now realize what is remaining by using the cyclic lens spaces $L^{n}$ described above, viewing the fundamental group $C_8=<s>$ sitting inside $SD_{16}$.\\
We start by calculating the eta invariant with respect to the real representation $\rho_0-\rho_4$ of $C_8$. We start by applying Lemma 3.4 to show that $\eta(L^{8m+j}(l=8;\overrightarrow{a}))(\rho_{4}-\rho_0)$ has order at least $2^m$ for both $j=3,7$ for suitable $\overrightarrow{a}$. Firstly, let $\omega=(1+i)/\sqrt{2}$ be the generator of $C_8$. Then $\rho_4(\omega)=-1, \omega^5=-\omega, \omega^7=-\omega^3$ and we have:
$$\eta(L^{3}(8;(1,1)))(\rho_{4}-\rho_0)=\frac{1}{8}\{\frac{2\omega}{(1-\omega)^2}+\frac{2\omega^3}{(1-\omega^3)^2}+\frac{2\omega^5}{(1-\omega^5)^2}+\frac{2\omega^7}{(1-\omega^7)^2}\}$$
$$=\frac{1}{4}\{\frac{\omega((1+\omega)^2-(1-\omega)^2)}{(1-i)^2}+\frac{\omega^3((1+\omega^3)^2-(1-\omega^3)^2)}{(1+i)^2}\}$$
$$=\frac{1}{8i}\{-\omega(4\omega)+\omega^3(4\omega^3)\}=-8i/8i=-1$$
which has order $2 \in \R/2\Z$. Similarly, in $n=7$ we have:
$$\eta(L^{7}(8;(1,1,1,1)))(\rho_{4}-\rho_0)=\frac{1}{8}\{\frac{2\omega^2}{(1-\omega)^4}+\frac{2\omega^6}{(1-\omega^3)^4}+\frac{2\omega^{10}}{(1-\omega^5)^4}+\frac{2\omega^{14}}{(1-\omega^7)^4}\}$$
$$=\frac{1}{4}\{\frac{i((1+\omega)^4+(1-\omega)^4)}{(1-i)^4}-\frac{i((1+\omega^3)^4+(1-\omega^3)^4)}{(1+i)^4}\}$$
$$=\frac{-1}{16}\{12i^2-12i(-i)\}=24/16=3/2$$
which has order $2 \in \R/\Z$.\\
For the general case we let $\overrightarrow{a}=(a_1, \cdots,a_{2k}), K=(a_1+\cdots a_{2i})/2$ and use the following inductive trick:
$$\eta(L^{4k+7}(8;(\overrightarrow{a},1,1,5,5)))(\rho_{4}-\rho_0)=$$
$$=\frac{1}{8}\{\frac{2\omega^{K+6}}{(1-\omega^{a_1})\cdots(1-\omega^{a_{2k}})(1-i)^2}+\frac{2\omega^{3K+18}}{(1-\omega^{3a_1})\cdots(1-\omega^{3a_{2k}})(1+i)^2}+$$
$$\frac{2\omega^{5K+30}}{(1-\omega^{5a_1})\cdots(1-\omega^{5a_{2k}})(1-i)^2}+\frac{2\omega^{7K+42}}{(1-\omega^{7a_1})\cdots(1-\omega^{7a_{2k}})(1+i)^2}\}$$
Noting again that $(1 \pm i)^2=\pm 2i$, and $\omega^{4j+2}=i,-i$ for $j$ even, odd respectively, we may take the common terms out of the bracket to simplify:
$$=\frac{i}{16i}(\frac{2\omega^{K}}{(1-\omega^{a_1})\cdots(1-\omega^{a_{2k}})}+\frac{2\omega^{3K}}{(1-\omega^{3a_1})\cdots(1-\omega^{3a_{2k}})}+$$
$$\frac{2\omega^{5K}}{(1-\omega^{5a_1})\cdots(1-\omega^{5a_{2k}})}+\frac{2\omega^{7K}}{(1-\omega^{7a_1})\cdots(1-\omega^{7a_{2k}})} )$$
$$=(1/2) \eta(L^{4k-1}(8;(\overrightarrow{a}))(\rho_{4}-\rho_0)$$
Thus $\eta(L^{4k+7}(8;(\overrightarrow{a},1,1,5,5)))(\rho_{4}-\rho_0)$ has twice the order in $\R/\Z$ of $\eta(L^{4k-1}(8;(\overrightarrow{a}))(\rho_{4}-\rho_0)$, and the claim is immediate by induction.

Thus we can consider $6-$tuples of Eta invariants as follows. The three non-trivial one dimensional representations of $SD_{16}$ can each be characterized by their kernels, which are $C_8, D_8$ and $Q_8$ respectively. We will denote the representation with kernel $C_8$ as $\hat{C_8}$ for example. Then for a manifold $M$ we set:
$$ \overrightarrow{\eta}(M)=(\eta(M)(1-\hat{D_8}),\eta(M)(1-\hat{C_8}),\eta(M)(1-\hat{Q_8}),
\eta(M)(2-\chi_{\rho}),\eta(M)(4+\chi_{\rho}.\chi_{\rho^5}-2(\chi_{\rho} + \chi_{\rho^5}))$$
Then using \cite{2} and the calculations we've already given we have upto order at least in $n=8m+3$:\\
$$\overrightarrow{\eta}(L^{8m+3}(8; \overrightarrow{a})=(2^{-m-1},0,2^{-m-1},*,*,*)$$
$$\overrightarrow{\eta}(\RP^n)=(0,2^{-4m-3},2^{-4m-3},2^{-4m-3},2^{-4m-2},2^{-4m-3})$$
$$\overrightarrow{\eta}(M_{1}^n-M_{2}^n)=(*,*,0,2^{-2m-2},0,0)$$
$$\overrightarrow{\eta}(M_{Q}^{n-8} \times B^8)=(*,*,0,*,1/2^{4m-1}+3/2^{2m},1/2^{4m-2}+3/2^{2m})$$
$$\overrightarrow{\eta}(M_{Q}^{n-5})=(*,*,0,*,1/2^{4m+3}+3/2^{2m+2},1/2^{4m+2}+3/2^{2m+2})$$
Here $*$ is a term we aren't interested in, and we have divided by two already for the real representations. Thus in $\R/\Z$ we may consider the following matrix of vectors spanned by positive scalar curvature manifolds:
\begin{displaymath}
\left(\begin{array}{cccccc}
2^{-m-1} & 0 & 2^{-m-1} & * & * & * \\
0 &2^{-4m-3} & 2^{-4m-3} & 2^{-4m-3} & 2^{-4m-2} & 2^{-4m-3}\\
* & * & 0 & 2^{-2m-2} & 0 & 0 \\
* & * & 0 & * & 1/2^{4m+3}+3/2^{2m+2} & 1/2^{4m+2}+3/2^{2m+2} \\
* & * & 0 & * & 1/2^{4m-1}+3/2^{2m}  & 1/2^{4m-2}+3/2^{2m}\\
\end{array} \right)
\end{displaymath}
Note that since $\hat{D_8}$ and $\hat{C_8}$ both restrict to the same representation on $Q_8$, we can do a cancellation with the first three columns, namely, adding the first and the third and then subtracting the second gives us the column $(2^{-m},0,0,0,0)^T$, and so this, combined with the zeroes down the $\hat{Q_8}$ column immediately imply that we have spanned a subspace of order at least $2^m$ times what is spanned by $\RP^n$ and the quaternion lens spaces, which we already calculated as $2^{8+12m}$. Thus in total we have a subspace of order $2^{8+13m}$ as required.\\
Proceeding in exactly the same manner in $n=8m+7$ gives us a subspace of order $2^{11+13m}$ in $\R/\Z$. However here, the cancellation we use above is slightly different, because when we restrict to the real representation of $C_8$ we have that $\eta(L^{8m+7}\overrightarrow{a})(\rho_{4}-\rho_0) \in \R/\Z$ infact has the same order as $\eta(L^{8m+7}(1, \overrightarrow{a})(2\rho_{4}-2\rho_0)$, since $2\rho_4$ is in fact a quaternion representation so that the eta invariant takes values in $\R/2\Z$. Thus, just like in the previous chapter, the order of the group spanned is doubled, giving the required order of $2^{12+13m}$ in total. This proves the first part of the proposition.

In dimensions $4m+1$ we use the Lens space bundles $L^{4m+1}$ described above, with fundamental group $C_8=<s>$, and calculate some eta invariants, again closely following the methods in \cite{2}. We note from \cite{2} that there is a surjective map from ${\L}_n(BC_8)$ to ${\L}_{n-4}(BC_8)$, where ${\L}_n(BC_8) \subset \R/\Z$ is the subspace spanned by the set of $\eta(L^{n})(\rho)$. Thus using naturality there is also a surjective map $\delta:{\L}_n(BSD_{16}) \rightarrow {\L}_{n-4}(BSD_{16})$.\\
Recall from the second section that in dimensions $8m+5$ we have $Ker(Ap)=[2^{m+1}]$, while in $8m+9$ we have $Ker(Ap)=[2^{m+1}]$ together with eta multiples. Since the order of the subgroups we need to realize using these lens space bundles is the same in $8m+5, 8m+9$, the above paragraph implies we need only realize a subgroup of order $2^{m+1}$ in dimensions $8m+5$. Further, again using \cite{2}, we need only check this claim in dimensions $5,13$, since then we would deduce that $\delta:{\L}_{13}(BSD_{16})=[4] \rightarrow {\L}_{9}(BSD_{16})=[2]$ is surjective with kernel $[2]$, and then by periodicity $\delta:{\L}_{8m+13}(BSD_{16}) \rightarrow {\L}_{8m+9}(BSD_{16})$ is also surjective with kernel at least $[2]$, so that inductively we would have realized a subgroup of order $2^{m+2}$ as required.\\
So we do some calculations in dimensions $5$ and $13$. Note that $\chi_{\rho}$ restricts to $\rho_1 \oplus \rho_3$ on the cyclic subgroup $C_8=<s>$, where $\rho_1$ is the natural representation sending $s \rightarrow \omega=(1+i)/\sqrt{2}$, and $\rho_3$ sends $s \rightarrow \omega^3$. So we take:
$$\eta(i_*(L^5(1,1)))(2\rho_0-\chi_{\rho})=\eta(L^5(1,1)(2\rho_0-\rho_1-\rho_3)$$
$$=\eta(L^5(1,1))(\rho_0-\rho_1)+\eta(L^5(1,1)(\rho_0-\rho_3)$$
and we calculate each summand directly using Lemma 3.4:
$$\eta(L^5(1,1))(\rho_0-\rho_1)=\frac{1}{8} \sum_{1 \neq \lambda \in C_8} \frac{\lambda(1+\lambda)(1-\lambda)}{(1-\lambda)^3}$$
$$=\frac{1}{8}\{\frac{\omega(1+\omega)}{(1-\omega)^2}+ \frac{i(1+i)}{(1-i)^2}+\frac{\omega^3(1+\omega^3)}{(1-\omega^3)^2}+0-\frac{\omega(1-\omega)}{(1+\omega)^2}-\frac{i(1-i)}{(1+i)^2}-
\frac{\omega^3(1-\omega^3)}{(1+\omega^3)^2}\}$$
$$=\frac{1}{8}\{\omega(\frac{(1+\omega)^3-(1-\omega)^3}{(1-i)^2})+\omega^3(\frac{(1+\omega^3)^3-(1-\omega^3)^3}{(1+i)^2})-\frac{(1+i)}{2}-\frac{(1-i)}{2}\}$$
$$=-3/4-1/8$$
Where we use $\omega^2=i, \omega^{4+j}=-\omega^j$, and then separate the $i,-i$ terms in the summand. We proceed analogously for $\rho_3$:
$$\eta(L^5(1,1))(\rho_0-\rho_3)=\frac{1}{8} \sum_{1 \neq \lambda \in C_8} \frac{\lambda(1+\lambda)(1-\lambda^3)}{(1-\lambda)^3}$$
$$=\frac{1}{8}\{\omega(\frac{(1+\omega)^4(1-\omega^3)-(1-\omega)^4(1+\omega^3)}{(1-i)^3})+\omega^3(\frac{(1+\omega^3)^4(1-\omega)-(1-\omega^3)^4(1+\omega)}{(1+i)^3})-\frac{2}{(1-i)^3}-\frac{2}{(1+i)^3}\}$$
$$=-3/4+1/8$$
So that adding up we get $-3/2 \in \R/\Z$ of order 2, as required. We now make the analogous calculation in dimension $13$ for $L^{13}(1,1,1,1,1,1)$:
$$\eta(L^5(1,1,1,1,1,1))(\rho_0-\rho_1)=\frac{1}{8} \sum_{1 \neq \lambda \in C_8} \frac{\lambda^3(1+\lambda)(1-\lambda)}{(1-\lambda)^7}$$
$$=\frac{1}{8}\{\omega^3(\frac{(1+\omega)^7-(1-\omega)^7}{(1-i)^6})+\omega(\frac{(1+\omega^3)^7-(1-\omega^3)^7}{(1+i)^6})-\frac{i^3(1+i)}{(1-i)^6}+\frac{i(1-i)}{(1+i)^6}\}$$
$$=-17/8-1/32$$
While for $\rho_3$ we have:
$$\eta(L^5(1,1,1,1,1,1))(\rho_0-\rho_3)=\frac{1}{8} \sum_{1 \neq \lambda \in C_8} \frac{\lambda^3(1+\lambda)(1-\lambda^3)}{(1-\lambda)^7}$$
$$=\frac{1}{8}\{\omega^3(\frac{(1+\omega)^8(1-\omega^3)-(1-\omega)^8(1+\omega^3)}{(1-i)^7})+\omega(\frac{(1+\omega^3)^8(1-\omega)-(1-\omega^3)^8(1+\omega)}{(1+i)^7})+\frac{2}{(1-i)^7}+\frac{2}{(1+i)^7}\}$$
$$=-17/8+1/32$$
So that adding gives us $-17/4$ of order $4 \in \R/\Z$, which completes the proof.

\end{proof}
\section{The extra class in $8k$ on the 1-column}
The result of the calculations of the previous section is that we can realize all of $Ker(Ap)$ in dimensions $4k+3$ and together with eta-multiples, in $4k+1$ also. However, in dimensions $8k$, there is a class which is not detected in periodic K-theory, but is in first local cohomology, for which we need a separate geometric construction. This class is also detected in the ordinary $\Z_{2}$-homology of $BSD_{16}$, and its image is $\xi(yuP^{2k-1})$, the class dual to $yuP^{2k-1}$. We realize this class as follows:\\
The idea is to start with the exact sequence $C_{8} \rightarrow SD_{16} \rightarrow \Z_{2}$, lift it to obtain a sequence $C_{8} \rightarrow G \rightarrow \Z$, and take classifying spaces. Of course $S^{1} =B\Z$, and since $C_{8}$ acts on $n\Comp$ by multiplication, we have a map from a lens space $L$ into $BC_{8}$, and thus a fibre bundle $ L^{n-1} \rightarrow M^{n} \rightarrow S^{1}$, as shown. As a lens bundle this carries positive scalar curvature, and we claim that this gives the remaining class when $n=8k$ and $k \geq 1$.
$$
\xymatrix{
L^{n-1} \ar[r] \ar [d] & M^{n} \ar[r] \ar [d]^{f} & S^{1} \ar[d]^{\simeq} \\
BC_{8} \ar[r] \ar [d] & BG \ar[r] \ar [d]^{g} & B\Z \ar[d] \\
BC_{8} \ar[r] & BSD_{16} \ar[r] & B\Z_{2} }
$$

\begin{prop}
When $n=8k$, the manifold $M^{n}$ constructed above is a spin manifold, whose image in $H_{8k}(BSD_{16})$ is $\xi(yuP^{2k-1})$, the class dual to $yuP^{2k-1}$.
\end{prop}
\begin{proof}
Let $F=g \circ f$. We now calculate the cohomology of the manifold $M$, along with the map $F^*:H^*(BSD_{16})\rightarrow H^*(M)$. We know $H^{*}(L^{n-1}; \Z_{2})=\Z_{2}[X, \widetilde{\tau}]/X^{4k}, \widetilde{\tau}^{2}$ where $X$ is a dimension $2$ generator, and $\widetilde{\tau}$ dimension $1$, and $H^{*}(S^{1})=\Z_{2}[\sigma]/\sigma^{2}$. By the universal coefficient Theorem, all the cohomology groups except of course $H^{0}(M),H^{n}(M)$ have rank $2$, and the Serre spectral sequence implies that there are two degree one and one degree two generators, one of the degree one generators being $\sigma$, with $\sigma^{2}=0$. Since $y=w_{1}(\chi_{2})$ has kernel $C_8$, restricting representations implies that $\sigma=F^{*}(y)$, and we can then define $\tau=F^{*}(x)$, and we will deduce  $H^{*}(M^{n},\Z_{2})=\Z_{2}[\sigma, \tau, Z]/\sigma^{2}, \sigma \tau + \tau^{2}, Z^{4k}$, where $Z$ is a degree 2 class restricting to $X \in H^{*}(L^{n-1})$.\\
Since $P$ restricts non-trivially on the cyclic group also, it follows that $F^*(P)=Z^2$ or $Z^2+Z\tau^2$. Either way $F^*(P^2)=Z^4$, and recall that $P = \overline{c}_{2}(\rho_1)$ and we can view $L^{n-1}=S(2k \rho_1)/ C_{8}$. Notice that the universal cover $\widetilde{M}$ of $M=S(2k \rho_1) \times_{G} \R$ is just $ S(2k \rho_1) \times \R$, and so the induced vector bundle over $M$ is given by
 $$ 2k \rho_1 \times_{G}(S(2k \rho_1) \times \R) \rightarrow  S(2k \rho_1) \times_{G} \R $$
 This has a section via the diagonal map, which implies $Z^{4k}=F^*(P^{2k})=\overline{c}_{4k}(\rho_1)$ must be zero. \\
 Now $u^2=P(x^2+y^2) \in H^*(SD_{16})$ implies $F^*(u^2)=\tau^2 Z^2$, and since $xu=0$ we deduce $F^*(u)=Z(\tau + \sigma)$.
  We now claim this manifold $M$ is spin, and $F^*(P)=Z^2+Z\tau^2$. We have $Sq^1(u)=0$, so that $0=Sq^1(Z(\tau + \sigma))=Sq^1(Z)(\tau + \sigma)+Z\tau^2$ so that $Sq^1(Z)=Z\sigma$ or $Z(\tau + \sigma)$. Seeking a contradiction, assume the latter. Then $Sq^1(Z^{4k-1} \sigma)=Z^{4k-1}\sigma \tau$ and $Sq^1(Z^{4k-1} \tau)=Z^{4k-1}\tau^2$.\\
  We recall that in a smooth manifold N, by the Wu fornulae \cite{micc} we have $w_{k}(N)=\sum_{i+j=k} Sq^{i}(v_{j})$, where $v_{j} \in H^{j}(M^{n})$ is the unique class such that $v_{j}y=Sq^{j}(y)$, for every $y \in H^{n-j}(M)$. Thus we deduce from above that $w_1=v_1=\tau$. However, $M$ is a fibre bundle with orientable fibre, and so $w_1(M)$ must restrict to zero in $H^*(L^{4n-1})$. Since $\tau$ restricts to $\widetilde{\tau}$, we have a contradiction.\\
  So $Sq^1(Z)=Z\sigma$, implying
  $$Sq^1(Z^{4k-1} \sigma)=Z^{4k-1}\sigma^2=0$$ and
  $$Sq^1(Z^{4k-1} \tau)=Z^{4k-1}\tau^2 + Z^{4k-1}\sigma \tau=0$$
  so that $w_1(M)=v_1=0$ .\\
  Further
  $$Sq^2(Z^{4k-1})=Z^{4k} +(Sq^1(Z^{2k-1}))^2 =0$$
  $$Sq^2(Z^{4k-2} \tau^2)=\tau^2 Sq^2(Z^{4k-2})= \tau^2 (Sq^1(Z^{2k-1}))^2=0$$
  so that $w_2(M)=v_2=0$ also.\\
  Now $Sq^2(P)=u^2$ implies $Sq^2(F^*(P))=Z^2 \tau^2$, so that $F^*(P) \neq Z^2$.
 The top cohomology class of $M$ is $Z^{4k-1} \tau^{2}=Z^{4k-1}\sigma \tau$, and observe that $F^*(yuP^{2k-1})=\sigma Z(\sigma +\tau)(Z^2+Z\tau^2)^{2k-1}=Z^{4k-1}\sigma \tau$. Any other class mapping to $Z^{4k-1}\sigma \tau$ must have a factor of $uP^{2k-1}$ and as $xu=0$, dualising $F^{*}$ immediately gives $F_{*}(\xi(Z^{4k-1} \tau^{2}))=\xi(yuP^{2k-1})$.
 \end{proof}
 This construction is entirely analogous to the one used to obtain a similar extra class for dihedral groups, which is in second local cohomology in high enough dimensions. This is a representation theoretic geometric construction, and the only immediate diference between the two cases is that the representation in question here is complex.

\section{ Ordinary cohomology and the 2-column}
The two column in the local cohomology spectral sequence is the same for all semi-dihedral groups, and it is a $\Z_{2}$ vector space of dimension $k+1$ if $n=8k+4$,$k$ if $n=8k+j$ with $j \neq 4$ even, and $0$ if $n$ is odd. It again suffices to detect these classes in ordinary $\Z_{2}$ homology. We claim that sufficiently many classes are realized by inclusion from a Klein $4-$subgroup $V(2)$.\\
Note that the dihedral group of order $8$ includes into all the semi-dihedral groups.
As in the previous chapter, we denote by $\omega \in D_{8}$ the rotation by $\pi/2$, $s,s',t,t'$ the reflections through the lines $y=0,x=0,y=x,y=-x$ respectively. The Klein $4-$subgroup we consider is $V(2)=<s,s'>$ (note that $<t,t'>$ is another one, but it turns out not to be needed).\\
We recall that $H^{*}(BD_{8},\Z_{2})=\Z_{2}[\alpha,\beta,\delta]/\alpha \beta +\beta^{2}$, where $\alpha=w_{1}(\hat{\omega}),\beta=w_{1}(\hat{s}),\delta=w_{2}(\sigma)$. Here $\hat{ }$ denotes the inflated representation, and $\sigma$ is the unique two dimensional representation. Further we have $H^{*}(BV(2))=\Z_{2}[p,q]$, with $p=w_{1}(\hat{\omega^{2}}),q=w_{1}(\hat{s})$. Restricting representations then gave us the induced maps in cohomology. The map $H^{*}(BD_{8})\rightarrow H^{*}(BV(2))$ is determined by $ \alpha \rightarrow p, \beta \rightarrow 0, \delta \rightarrow q(p+q)$. So if $A=\Z_{2}[\alpha,\delta]$ we deduce $H^{*}(BV(2))=A \oplus qA$.
$$H^{*}(BD_{8},\Z_{2}) \rightarrow A \rightarrowtail H^{*}(BV(2))=A \oplus qA $$
 We then chose dual bases in homology and computed where the classes in $H_{*}^{+}(BV(2)$ ( which we know explicitly) mapped to.
 $$H_{*}(BD_{8},\Z_{2}) \leftarrow A^{v} \twoheadleftarrow H_{*}(BV(2))=A^{v} \oplus qA^{v} $$.
 The following Proposition summarizes the calculation in the last chapter.
\begin{prop}
The image under the above inclusion of $H_{*}^{+}(BV(2))$ in $H_*^+(BD_8)$ is spanned by the following classes:\\
In dimensions $4k+2$ it is spanned by $\xi(\alpha^{4i}\delta^{4j+3})$, with $0 \leq i,j \in Z$ and $4i+8j+6=4k+2$.\\
In dimensions $4k$ it is spanned by $\xi(\alpha^{4i+2}\delta^{4j+1})$, with $0 \leq i,j \in Z$ and $4i+2+8j+2=4k$.
\end{prop}

We now have a sequence of inclusions $ko_{*}(BV(2)) \rightarrow ko_{*}(BD_{8}) \rightarrow ko_{*}(BSD_{2^{N+2}})$, and we have understood what classes are induced from the first inclusion. We know consider the second inclusion, starting with the cohomology of semi-dihedral groups and seeing how it restricts to $D_{8}$. This follows from the following Proposition, see \cite{ep}, \cite{kijti}.
\begin{prop}
We have $$H^{*}(BSD_{2^{N+2}},\Z_{2})=\Z_{2}[x,y,u,P]/(xy+x^{2},xu,x^{3},u^{2}+(x^{2}+y^{2})P)$$
where $\mid x \mid=\mid y \mid=1,\mid u \mid=3,\mid P \mid=4 $. Further the restriction map $f^*:H^{*}(BSD_{2^{N+2}}) \rightarrow H^{*}(BD_{8})$ sends $x \rightarrow 0,y \rightarrow \alpha, u \rightarrow \alpha \delta, P \rightarrow \delta^2$.
\end{prop}
So we can immediately dualise again to check that we get sufficiently many classes in $H_{*}^{+}(BSD_{2^{N+2}};\Z_{2})$, and we can now prove:
\begin{prop}
The image under the above inclusion spans all of the two-column in the local cohomology filtration for $ko_*(BSD_{2^{N+2}})$.
\end{prop}
\begin{proof}
As always, we check that we have sufficiently many classes in ordinary homology.\\
We claim it suffices to check that if $i>0$, then $f_*(\xi(\alpha^i \delta^j))\neq 0$ and $f_*(\xi(\alpha^i \delta^j))=f_*(\xi(\alpha^{i'} \delta^{j'})) \Rightarrow i=i', j=j'$\\
Indeed, we know that for $n=4k, 4k+2$, we have $\lfloor(k+1)/2 \rfloor$ classes in $H_*^+(BD_8)$ induced from $V(2)$. By Theorem $5.2.1$ we need to produce $K+1$ classes in $H_n^+(BSD_{2^{N+2}})$ for $n=8K+4$, and $K$ classes for $n=8K, 8K+2$ and $8K+6$.\\
Now if we have $8K+4=4k$ then $8K+8=4k+4$ which implies that $K+1=(k+1)/2=\lfloor(k+1)/2 \rfloor$. Similarly $8K=4k$ and $8K+2=4k+2$ both imply $K=k/2=\lfloor(k+1)/2 \rfloor$ since $k$ must be even. Finally if $8K+6=4k+2 \Rightarrow K=(k-1)/2=\lfloor(k+1)/2 \rfloor-1$, which is exactly what we need since in these dimensions we have an extra $\xi(\delta^{4K+3}) \in H_{*}^{+}(BD_{8},\Z_{2})$ which clearly maps to zero under $f_*$.\\
So it suffices to check that terms dual to $\alpha^i \delta^j$ with $i>0$ are mapped monomorphically by $f_*$. Recall that the classes $\xi(\alpha^i \delta^j) \in H_*^+(BD_8)$ always have $i$ even and $j$ odd, so that by Proposition $5.5.2$ we have that $f^*(y^{i-1}u P^{(j-1)/2})=\alpha^i \delta^j$. Thus $f_*(\xi(\alpha^i \delta^j)) \neq 0$, and we claim that $f_*(\xi(\alpha^i \delta^j))=\xi(y^{i-1}u P^{(j-1)/2})$.\\
This follows since if $f^*(y^a u^b P^c)=\alpha ^{a+b} \delta^{2c+b}$ and $f^*(y^{a'} u^{b'} P^{c'})=\alpha ^{a'+b'} \delta^{2c'+b'}$ both map to $\alpha^i \delta^j$ with $i>0$ even and $j \geq 1$ odd, then $a, a', b,b'$ must all be odd. Further $a+b=a'+b',2c+b=2c'+b'$, which implies that $a'=a-2j,b'=b+2j,c'=c-j$, for some $j \in \Z$, and $j \geq 0$ without loss of generality. However in $H^{*}(BSD_{2^{N+2}},\Z_{2})$, we have $u^2=x^2 P +y^2 P$ which implies $y u^3=yx^2uP +y^3 u P=y^3 u P$ since $xu=0$. Thus we can repeatedly apply this formula to deduce
$$y^{a'}u^{b'}P^{c'}=y^{a-2j}u^{b+2j}P^{c-j}=yu^3(y^{a-2j-1}u^{b+2j-3}P^{c-j})=y^3uP(y^{a-2j-1}u^{b+2j-3}P^{c-j})=$$
$$=y^{a-2j+2}u^{b+2j-2}P^{c-j+1}= \cdots =y^au^bP^c.$$
Now since $f_*(\xi(\alpha^i \delta^j))= \sum \xi(k)$ where the sum is over all $k \in H^*(BSD_{2^{N+2}})$ that map under restriction to $\xi(\alpha^i \delta^j)$, we deduce that $f_*(\xi(\alpha^i \delta^j))=\xi(y^{i-1}u P^{(j-1)/2})$, and thus that $f_*(\xi(\alpha^i \delta^j))=f_*(\xi(\alpha^{i'} \delta^{j'})) \Rightarrow i=i', j=j'$, as required.
\end{proof}

\end{document}